\documentclass[10pt,reqno]{amsart}

\usepackage[square,numbers]{natbib}

\iffalse 
\usepackage[autostyle]{csquotes}
\usepackage[
    backend=biber,
    style=numeric,
    firstinits = true,
    isnb=false,
    url=false, 
    doi=true,
    eprint=true
]{biblatex}

\renewbibmacro{in:}{
  \ifentrytype{article}{}{
  \printtext{\bibstring{in}\intitlepunct}}}

\DeclareFieldFormat[article,unpublished,book]{pages}{#1}  

\DeclareFieldFormat[article]{title}{#1}
\DeclareFieldFormat[unpublished]{title}{#1}
\DeclareFieldFormat[misc]{title}{#1}
\fi 

\usepackage[hmarginratio=1:1,top=32mm,columnsep=20pt]{geometry} 
\usepackage[latin2]{inputenc}
\usepackage{amsmath}
\usepackage{graphicx}
\usepackage{amssymb}
\usepackage{esint}
\usepackage{color}
\usepackage{amsthm}
\usepackage{epsfig}

\usepackage{todonotes}

\usepackage[]{todonotes}
\usepackage{mathrsfs}

\usepackage{xxcolor}\usepackage[linktocpage=true,colorlinks=true,linkcolor=Blue,citecolor=Green]{hyperref}
\usepackage{enumitem}
\usepackage{mathtools}
\usepackage{tikz}
\usetikzlibrary{arrows}
\usepackage[english]{babel}

\newtheorem{theorem}{Theorem}
\newtheorem{proposition}[theorem]{Proposition}
\newtheorem{lemma}[theorem]{Lemma}

\newtheorem*{theorem*}{Theorem}
\newtheorem*{theorem**}{{``Theorem''}}
\theoremstyle{definition}

\def\Xint#1{\mathchoice
{\XXint\displaystyle\textstyle{#1}}%
{\XXint\textstyle\scriptstyle{#1}}%
{\XXint\scriptstyle\scriptscriptstyle{#1}}%
{\XXint\scriptscriptstyle\scriptscriptstyle{#1}}%
\!\int}
\def\XXint#1#2#3{{\setbox0=\hbox{$#1{#2#3}{\int}$ }
\vcenter{\hbox{$#2#3$ }}\kern-.6\wd0}}

\def\dashint{\Xint-}

\definecolor{Yellow}{rgb}{0.95,0.9,0.0} 
\definecolor{Red}{rgb}{0.8,0.1,0.1}
\definecolor{Green}{rgb}{0.1,0.65,0.2}
\definecolor{Blue}{rgb}{0.1,0.1,0.8}
\definecolor{Purple}{rgb}{0.7,0.1,0.7}
\definecolor{Grey}{rgb}{0.6,0.6,0.6}

\newcommand{\domain}{{\mathbb{T}^d}}

\newcommand{\Ghd}{G_{h,d}}
\newcommand{\Lh}{L^2(G_{h,d})}

\allowdisplaybreaks[2]

\newcommand{\f}{\boldsymbol}

\newcommand{\m}{\mbox{d}}

\usepackage{delimdelim}

\newtheorem{corol}[theorem]{Corollary}
\newtheorem{prop}[theorem]{Proposition}

\theoremstyle{definition}

\newtheorem{rem}[theorem]{Remark}

\newtheorem{innercustomgeneric}{\customgenericname}
\providecommand{\customgenericname}{}
\newcommand{\newcustomtheorem}[2]{%
  \newenvironment{#1}[1]
  {%
   \renewcommand\customgenericname{#2}%
   \renewcommand\theinnercustomgeneric{##1}%
   \innercustomgeneric
  }
  {\endinnercustomgeneric}
}

\newcustomtheorem{customthm}{Assumption}
\newcustomtheorem{customthm2}{Definition}

\begin{document}

\title[Rigorous derivation of the Dean--Kawasaki equation]{The Dean--Kawasaki equation and the structure of density fluctuations in systems of diffusing particles}

\author{Federico Cornalba}
\address{Institute of Science and Technology Austria (ISTA), Am~Campus~1, 
3400 Klosterneuburg, Austria}
\email{federico.cornalba@ista.ac.at}
\author{Julian Fischer}
\address{Institute of Science and Technology Austria (ISTA), Am~Campus~1, 
3400 Klosterneuburg, Austria}
\email{julian.fischer@ista.ac.at}

\begin{abstract}
The Dean--Kawasaki equation -- a strongly singular SPDE -- is a basic equation of fluctuating hydrodynamics; it has been proposed in the physics literature to describe the fluctuations of the density of $N$ independent diffusing particles in the regime of large particle numbers $N\gg 1$.
The singular nature of the Dean--Kawasaki equation presents a substantial challenge for both its analysis and its rigorous mathematical justification. Besides being non-renormalisable by the theory of regularity structures by Hairer et al., it has recently been shown to not even admit nontrivial martingale solutions.

In the present work, we give a rigorous and fully quantitative justification of the Dean--Kawasaki equation by considering the natural regularisation provided by standard numerical discretisations: We show that structure-preserving discretisations of the Dean--Kawasaki equation may approximate the density fluctuations of $N$ non-interacting diffusing particles to arbitrary order in $N^{-1}$ (in suitable weak metrics).
In other words, the Dean--Kawasaki equation may be interpreted as a ``recipe'' for accurate and efficient numerical simulations of the density fluctuations of independent diffusing particles.
\end{abstract}

\maketitle

\section{Introduction}

The Dean--Kawasaki equation
\begin{align}
\label{DeanKawasaki}
\partial_t\rho = \frac{1}{2} \Delta \rho + N^{-1/2} \nabla \cdot (\sqrt{\rho} \hspace{0.05pc}\f{\xi}) 
\end{align}
-- with $\f{\xi}$ denoting a vector-valued space-time white noise -- has been proposed by Dean \cite{Dean} and Kawasaki \cite{Kawasaki} as a model for density fluctuations in a system of $N$ indistinguishable particles undergoing diffusion in the regime of large particle numbers $N\gg 1$.
Its mathematical analysis is complicated by its highly singular nature: A scaling argument shows that \eqref{DeanKawasaki} is not renormalisable by the theory of regularity structures by Hairer et al., even in one spatial dimension $d=1$.

Recently, Konarovskyi, Lehmann, and von~Renesse \cite{KonarovskyiVonRenesse} have obtained a striking rigidity result for the Dean--Kawasaki equation \eqref{DeanKawasaki}: They show that the only martingale solutions to \eqref{DeanKawasaki} are of the form of an empirical measure for $N$ independent Brownian motions
\begin{align}
\label{EmpiricalMeasure}
\mu_t^N:=\frac{1}{N}\sum_{k=1}^N \delta_{\f{w}_k(t)},
\end{align}
where the $\{\f{w}_k\}_{k=1}^{N}$ denote the $N$ independent Brownian motions. In particular, no solution exists for non-integer values of $N$. This result may be viewed as casting doubts on the mathematical meaningfulness of the Dean--Kawasaki equation: It amounts to stating that the Dean--Kawasaki equation is just a mathematically complex way of rewriting the diffusion of $N$ particles.
In turn, this naturally raises the question what the advantages of the Dean--Kawasaki equation \eqref{DeanKawasaki} might be over the particle formulation of diffusion \eqref{EmpiricalMeasure} from the point of view of physics. 

In the present work, we provide a rigorous justification for the Dean--Kawasaki equation: We show that standard numerical discretisations of the Dean--Kawasaki equation \eqref{DeanKawasaki} -- such as finite difference or finite element schemes -- provide accurate descriptions of the density fluctuations in a system of $N$ diffusing particles if measured in suitably weak metrics. Roughly speaking, we show that, under certain conditions, the solutions $\rho_h$ to the discretised Dean--Kawasaki equation achieve the approximation quality
\begin{align}\label{d:5}
& d_{weak,2j-1}\big(\rho_h-\mathbb{E}[\rho_h],\mu^N-\mathbb{E}[\mu^N]\big) \lesssim C(j) \left(\mean{\|\rho^{-}_h\|} + h^{p+1} + N^{-j/2}\right),
\end{align}
where $j\in \mathbb{N}$ is arbitrary, $h$ is the spatial discretisation parameter, $p+1$ is the order of convergence of the numerical scheme in the Sobolev space $H^{-1}$, and $d_{weak,2j-1}$ is a suitable weak metric of negative Sobolev type.
In particular, the accuracy is of arbitrarily large order in $N^{-1/2}$ and hence only limited by the numerical discretisation error and the negative part $\rho^{-}_h$. In addition, we show that $\mean{\|\rho^-_h\|}$ decays exponentially fast in -- roughly speaking -- $(hN^{1/d})^{1/2}$, demonstrating that the term becomes quickly negligible in the scaling regime
\begin{align}
\label{ScalingRegime}
h\gg N^{-1/d}
\end{align}
(where we have dropped logarithmic corrections in $N$ and contributions on the final time horizon for the sake of this introductory exposition).
In a nutshell, the bound \eqref{d:5} implies that the Dean--Kawasaki equation can be used as a ``recipe'' for accurate simulations of density fluctuations in systems of diffusing particles.

Note that our scaling regime \eqref{ScalingRegime} is not an actual restriction in the context of numerical simulations: It ensures that the average number of particles per cell is strictly larger than one. In fact, in the opposite regime $h\leq N^{-1/d}$, the direct simulation of particles would become less expensive than the approximation by the Dean--Kawasaki equation, as the numerical effort for the Dean--Kawasaki equation is strictly larger than $h^{-d}$.

While the Dean--Kawasaki equation correctly describes the fluctuations around the mean-field limit to arbitrarily large order in $N^{-1/2}$, the well-known linearised description of fluctuations given by the solution $\hat \rho$ to
\begin{equation}
 \label{b:3}
  \left\{\quad
    \begin{aligned}
    \!\!\!\!\! \partial_t\hat\rho&= \frac{1}{2} \Delta \hat \rho + N^{-1/2} \nabla \cdot (\sqrt{\bar \rho}  \hspace{0.05pc}\f{\xi}),\\
    \!\!\!\!\! \hat \rho(\cdot,0) &=\rho_0
    \end{aligned}
  \right.
\end{equation}
is limited to the approximation quality $d_{weak}(\hat \rho-\mathbb{E}[\hat\rho],\mu^N-\mathbb{E}[\mu^N])\leq C N^{-1}$.
Here, $\bar \rho$ denotes the mean-field limit given as the solution to
\begin{equation*}
  \left\{\quad
    \begin{aligned}
    \!\!\!\!\! \partial_t\bar\rho&= \frac{1}{2} \Delta \bar \rho,\\
    \!\!\!\!\! \bar \rho(\cdot,0) &=\rho_0.
    \end{aligned}
  \right.
\end{equation*}
In fact, the linearised description $\hat\rho$ only captures the leading-order fluctuation correction to the mean-field limit correctly and hence carries a relative error of order $N^{-1/2}$ with respect to the fluctuation scaling. We provide numerical evidence of such difference between the two models, and we also numerically verify convergence rates for suitable discretisations of the Dean--Kawasaki model \eqref{DeanKawasaki}.

\subsection{Related Literature}
The theory of fluctuating hydrodynamics describes fluctuations in interacting particle systems in the regime of large particle numbers using suitable SPDEs; 
see e.\,g.\ \cite{SpohnBook}.
In its form \eqref{DeanKawasaki}, the Dean--Kawasaki equation describes non-interacting particles, with similar equations being available for weakly interacting particles undergoing overdamped Langevin dynamics. In the recent contribution \cite{djurdjevac2022weak}, the authors also address quantitative fluctuation bounds in the non-interacting particle case, but by means of a suitable approximated SPDE model rather than a numerical scheme. While their setting grants several well-posedness results (non-negativity of the solution, comparison principles, entropy estimates) and allows to consider initial particle profiles which are more general than those treated here, their relative fluctuation error is however limited to $N^{-1/(d/2+1)}\log N$, while the rate of fluctuations in \eqref{d:5} is -- in suitable metrics -- arbitrarily high. 

For a more general particle setting, the SPDE of fluctuating hydrodynamics for the zero range process given by 
\begin{align}
\label{FluctHydZeroRange}
\partial_t\rho = \Delta (\Phi(\rho)) + \nabla \cdot \left(\sqrt{\Phi(\rho)} \hspace{0.05pc}\f{\xi}\right)
\end{align}
has been addressed in \cite{DirrStamatakisZimmer}, 
and linked it to the large deviation principle for such process in a suitable thermodynamic setting.
A corresponding well-posedness result for truncated (low spatial frequency) noise and regularised nonlinearity has been obtained in \cite{FehrmanGessWellPosedness}, see also \cite{FehrmanGess2021}. 
In \cite{fehrman2022ergodicity}, the construction of random dynamical systems for conservative SPDE is discussed, together with well-posedness theory of invariant measures and mixing of the related Markov process.
In \cite{FehrmanGess}, a large deviation principle for regularised variants of \eqref{FluctHydZeroRange} is shown; in a suitable limit, the rate functional of such large deviations principle and the corresponding one of the interacting particle system are shown to approach each other.

The paper \cite{DirrFehrmanGess}, written independently of -- and simultaneously to -- the present manuscript, gives a rigorous justification of the fluctuating hydrodynamics SPDE associated with the symmetric simple exclusion process
\begin{align*}
\partial_t\rho = \Delta \rho  - \nabla \cdot \left(\sqrt{\rho(1-\rho)}\hspace{0.05pc}\f{\xi}\right).
\end{align*}
While in contrast to our work the authors in \cite{DirrFehrmanGess} only consider the continuum SPDE, they regularise it by truncating the noise for small spatial wavelengths. In a certain sense, this introduces a regularisation in the same spirit as our numerical approach; however, their truncation criterion is somewhat more restrictive than our condition $h\gg N^{-1/d}$. While they face a more challenging problem with the more complex noise intensity factor $\sqrt{\rho(1-\rho)}$ (whose square is not a linear function of the density $\rho$) and also prove convergence results for the rate functions for large deviation principles, they only establish a leading-order description of fluctuations in the low deviations regime. In other words, in contrast to our present work, they do not show superiority of fluctuating hydrodynamics to a linearised approach on fluctuations for the ``bulk'' of the probability distribution.

For recent numerical approaches to fluctuating hydrodynamics, we refer the reader e.\,g.\ to \cite{russo2019,PhysRevEDelong,DonevVandenEijndenGarciaBell, cornalba2022regularised, cornalba2023density,helfmann2021interacting, djurdjevac2022feedback, li2019harnessing, bavnas2020numerical} (in particular,  \cite{cornalba2023density} contains the extension of the current work to the case of weakly interacting particles).
Note that the small prefactor of the noise term in the Dean--Kawasaki equation \eqref{DeanKawasaki} enables the use of certain higher-order timestepping schemes \cite{SDELab}, a fact that we also use in our numerical simulations.

Concerning further mathematical results on Dean--Kawasaki models, the link between fluctuating hydrodynamics and Wasserstein geometry has long been understood, and extensively studied in several works, see for instance \cite{JackZimmer,SturmEtAl,konarovskyi2020dean, konarovskyi2017reversible, konarovskyi2019modified,konarovskyi2020conditioning, schiavo2022dirichlet}.

Dean--Kawasaki type models including the effects of inertia have been derived and analysed by the first author, Shardlow, and Zimmer \cite{CornalbaShardlowZimmer,CornalbaShardlowZimmer2,CornalbaShardlowZimmer3}.

The fluctuation-dissipation relation -- implicitly contained for instance in the Dean--Kawasaki equation -- may be used to recover macroscopic diffusion properties from fluctuations in finite particle number simulations, see for instance \cite{embacher2018computing,li2019harnessing}.
Outside of the realm of physics, the concept of fluctuating hydrodynamics has also been applied to systems of interacting agents, see e.\,g.\  \cite{helfmann2021interacting, djurdjevac2022feedback, kim2017stochastic}.

Finally, conservative stochastic PDEs have recently been shown to give optimal convergence rates in the mean-field limit approximation of stochastic interacting particle systems, such as those encountered in the stochastic gradient descent methods for overparametrised, shallow neural networks \cite{gess2022conservative}.

\begin{rem}
Given the nature of the metric $d_{weak,2j-1}$ in \eqref{d:5}, it is natural to ask whether or not the high-order fluctuation error of \eqref{d:5} could be formally derived from suitable a priori estimates of negative Sobolev type for the continuous Dean--Kawasaki model \eqref{DeanKawasaki}. An a purely formal level, testing the mild formulation of \eqref{DeanKawasaki}
\begin{align}\label{DKmild}
\rho(\f{x},t) = G(t, \cdot)\ast \rho(\cdot,0)(\f{x}) + \int_{0}^{t}{\int_{\mathbb{T}^d}G(t-s,\f{x}-\f{y})\nabla\cdot\left[\sqrt{\rho(\f{y},s)}\xi(\f{y},s)ds \right]d \f{y}}
\end{align}
-- where $G$ is the heat semigroup kernel -- with trigonometric functions, and performing elementary computations, one arrives at the a priori estimate
\begin{align}
\mean{\left\|\rho(\cdot,t)\right\|^2_{H^{-j}}} & \lesssim \mean{\left\|\rho(\cdot,0)\right\|^2_{H^{-j}}} + N^{-1}\|\rho(0)\|_{L^1}\label{APrioriNegSob},
\end{align} 
which is valid in the regime $j>d/2$. 

Despite its formal validity -- which, however, relies on the non-trivial negativity requirement for the density $\rho$ -- the inequality \eqref{APrioriNegSob} does not give any information beyond the leading order $N^{-1}$, and therefore is too weak an estimate to justify the high-order fluctuation error bound in \eqref{d:5}.
\end{rem}

\section{Main results and summary}

The methodology of this paper can be applied to several standard numerical discretisations of the Dean--Kawasaki model \eqref{DeanKawasaki}, including finite difference and finite element schemes. In the interest of brevity, we only focus on a finite difference discretisation: The corresponding results in the finite element case are given in Appendix \ref{FEM_Appendix}. Specifically, on the periodic domain $\mathbb{T}^d:=[-\pi,\pi)^d$, we denote the uniform square grid with spacing $h$ by $G_{h,d}$, the discrete inner product of $L^2(G_{h,d})$ by $(\cdot,\cdot)_h$, the interpolating operator on $G_{h,d}$ by $\mathcal{I}_h$, and define the distance $d_{-j}[\f{X},\f{Y}]$ between two $\mathbb{R}^M$-valued random variables as
\begin{align}\label{b:9}
d_{-j}[\f{X},\f{Y}]
:=
\sup_{\psi\colon \max_{1\leq \tilde j\leq j}\|D^{\tilde j} \psi\|_{L^\infty}\leq 1}
&
\left|\mean{\psi\left(\f{X}-\mathbb{E}[\f{X}]\right)}
-\mean{\psi\left(\f{Y}-\mathbb{E}[\f{Y}]\right)}
\right|.
\end{align}

Our first main result reads as follows. 
\begin{theorem}[Accuracy of description of fluctuations by the finite-difference discretised Dean--Kawasaki model of order $p+1\in\mathbb{N}$] 

\label{main1} 
Assume the validity of Assumption \ref{ass:1} (discretised differential operators), Assumption  \ref{ass:2} (Brownian particle system), Assumption \ref{ass:3} (scaling assumptions), and Assumption \ref{ass:4} (discretised mean-field limit), all given below. In particular, assume that the mean-field limit $\overline{\rho}_{h}$ in \eqref{e:8} satisfies $ \rho_{min}\leq \overline{\rho}_h \leq \rho_{max}$ for some positive $ \rho_{min},\rho_{max}$ on $[0,T]$.
 Let $\rho_h$ be the solution of the discretised Dean--Kawasaki model given in Definition \ref{def:1} on $[0,T]$. Set
\begin{align}\label{b:7}
\Theta :=
\left\{ \begin{array}{rl}\displaystyle
0, & \mbox{if the discretisation \eqref{eq:29} admits a maximum principle,}  \\ 
\displaystyle d/2+1, & \mbox{otherwise.}\end{array}
\right.
\end{align} 

Then, for any $j\in \mathbb{N}$, 
the discrete Dean--Kawasaki model \ref{def:1} captures the fluctuations of the empirical measure $\mu^N$ in the sense that, for any $\f{T}=(T_1,\ldots,T_M)\in [0,T]^M$ with $0\leq T_1\leq \dots\leq T_M$, the inequality
\begin{align*}
&d_{-(2j-1)}\left[
N^{1/2}
\begin{pmatrix}
(\rho_h(T_1), \mathcal{I}_h\varphi_1)_h
\\
\vdots
\\
(\rho_h(T_M), \mathcal{I}_h\varphi_M)_h
\end{pmatrix}
,~
N^{1/2}
\begin{pmatrix}
\langle\mu^N_{T_1}, \varphi_1\rangle
\\
\vdots
\\
\langle\mu^N_{T_M}, \varphi_M\rangle
\end{pmatrix}
\right]
\\&~~~~~~~~~~~~~~~~~~~
\leq C(M,p,j,\|\boldsymbol{\varphi}\|_{W^{p+\Theta+j+1,\infty}}, \rho_{min},\rho_{max},\f{T}) {\mean{\sup_{t\in[0,T]}\|\rho_h^-(t)\|_h^2}^{1/2}}
\\&~~~~~~~~~~~~~~~~~~~~~~
+ C(M,p,j,\|\boldsymbol{\varphi}\|_{W^{p+\Theta+j+1,\infty}}, \rho_{min},\rho_{max},\f{T}) h^{p+1}
\\&~~~~~~~~~~~~~~~~~~~~~~
+ C(M,p,j,\|\boldsymbol{\varphi}\|_{W^{p+\Theta+j+1,\infty}}, \rho_{min},\rho_{max},\f{T}) N^{-j/2}
\\&~~~~~~~~~~~~~~~~~~~
=: \mathrm{Err}_{neg} + \mathrm{Err}_{num} + \mathrm{Err}_{fluct,rel}
\end{align*}
holds for any $\boldsymbol{\varphi}=(\varphi_1,\dots,\varphi_M)\in [W^{p+\Theta+j+1,\infty}(\mathbb{T}^d)]^M$  such that $\|\varphi_m\|_{L^2}=1, \forall m=1,\dots,M$ and $\int_{\mathbb{T}^d}{\varphi_k\varphi_l\emph{\m}\f{x}}=0$ whenever $T_k=T_l$.
Finally, we have the a posteriori bound
$$
{\mean{\sup_{t\in[0,T]}\|\rho_h^-(t)\|_h^2}^{1/2}} \leq C\mathcal{E}\!\left(N,h\right),
$$
where we have set
\begin{align}\label{b:8}
& \mathcal{E}\!\left(N,h\right) := C(d,\rho_{min},\rho_{max}) \left\{\exp\bigg(-\frac{\rho_{min} N^{1/2}h^{d/2}}{C\rho_{max}^{1/2}}\bigg)
+ \exp\big(-ch^{-1}\big)\right\}.
\end{align}
\end{theorem}

We make some observations in order to better illustrate the meaning of Theorem \ref{main1}:
\begin{itemize}[leftmargin=0.9 cm]
\item The quantities $(\rho_h(T_m),\mathcal{I}_h\varphi_m)_h$, 
and $\langle \mu_{T_m}^N, \varphi_m\rangle$ are rescaled with the factor $N^{1/2}$, as the natural order of density fluctuations is $N^{-1/2}$. In other words, our main error estimate basically provides an estimate for the relative error in the fluctuations.
\item The distances $d_{-j}[\f{X},\f{Y}]$ correspond to negative Sobolev norm differences of the probability measures on $\mathbb{R}^M$ given by the laws of $\f{X}$ and $\f{Y}$. In particular, it holds $d_{-1}[\f{X},\f{Y}]=\mathcal{W}_1[\f{X}-\mathbb{E}[\f{X}],\f{Y}-\mathbb{E}[\f{Y}]]$, 
where $\mathcal{W}_1$ is the $1$-Wasserstein distance.
\item The above estimates contain three types of error terms. The term $ \mathrm{Err}_{neg}$ quantifies the \emph{a priori} lack of knowledge concerning non-negativity of the solution $\rho_h$; the term $ \mathrm{Err}_{num}$ encodes the numerical precision of the scheme;
finally, the term $\mathrm{Err}_{fluct,rel}$ bounds the relative error in the fluctuations.
\item The order of differentiation required for the functions $\f{\varphi}$ should be thought of as the sum of $p+2+\Theta$ (accounts for the requirements of the spatial discretisation, discussed below) and $j-1$ (necessary due to an induction argument over $j$). 
\end{itemize}

If one is only interested in moment bounds (i.e., in a polynomial $\psi$) then the following estimate with no relative error in the fluctuations can be produced.

\begin{theorem}[Estimates on the error for stochastic moments]\label{main2} 

In the same setting of Theorem~\ref{main1},
fix times $\f{T}=(T_1,\dots,T_M)\in[0,T]^M$, a vector $\boldsymbol{j}=(j_1,\dots,j_M)$ with $j:=|\boldsymbol{j}|_1=\sum_{m=1}^{M}{|j_m|}$, and a vector $\boldsymbol{\varphi}=(\varphi_1,\dots,\varphi_M)\in [W^{p+j+1+\Theta,\infty}]^M$. 

Then the difference of moments between $\rho_h$ and the empirical density $\mu^N$ \eqref{EmpiricalMeasure} reads \begin{align}
 & \quad \left|\mean{\prod_{m=1}^{M}{\left[N^{1/2}(\rho_h(T_m)-\mean{\rho_h(T_m)},\mathcal{I}_h\varphi_{m})_h\right]^{j_m}}}\right. \nonumber\\
 & \quad\quad\quad \left.-\mean{\prod_{m=1}^{M}{\left[N^{1/2}\langle \mu^N_{T_m}-\mean{\mu^N_{T_m}},\varphi_{m}\rangle\right]^{j_m}}}\right|\nonumber\\
& \quad \leq \left\{C(d,\rho_{max},\rho_{min})\right\}^{j/2}\left[\prod_{m=1}^{M}{T^{j_m/2}_m}\right]j^{C_1j+C_2}\left[\prod_{m=1}^{M}{\|\varphi_m\|_{W^{j-1+\Theta,\infty}}^{j_m}}\right] {\mean{\sup_{t\in[0,T]}\|\rho_h^-(t)\|_h^2}^{1/2}}\nonumber\\
& \quad \quad + h^{p+1}\left\{C(d,\rho_{max},\rho_{min})\right\}^{j/2}\left[\prod_{m=1}^{M}{\left[T_m\vee \sqrt{T}_m\right]^{j_m/2}}\right]j^{C_3j+C_4}\left[\prod_{m=1}^{M}{\|\varphi_m\|_{W^{p+j+1+\Theta,\infty}}^{j_m}}\right]\nonumber\\
& \quad =: \mathrm{Err}_{neg} + \mathrm{Err}_{num}\label{eq:63c},
\end{align}
with constants $C,C_1,\dots,C_4>0$ independent of $j$, $h$, $N$, $T$, and where we have the bound 
$$
{\mean{\sup_{t\in[0,T]}\|\rho_h^-(t)\|_h^2}^{1/2}} \leq C\mathcal{E}\!\left(N,h\right),
$$
where $\mathcal{E}\!\left(N,h\right)$ has been defined in \eqref{b:8}.
\end{theorem}

\subsection{Structure of the paper}
Section \ref{fdiff} lays out the finite difference discretisation of the Dean--Kawasaki model. Subsection \ref{ss:not} (respectively, Subsection \ref{ss:ass}) lays out the necessary notation (respectively, the relevant assumptions and definitions) related to the model. Subsection \ref{ideas} -- which has an informal flavour -- brings forward some of the main ideas used in the paper. This section lays the ground for Subsection \ref{key1} (respectively, Subsection \ref{ss:rec}), which contains preparatory results for the proofs of Theorem \ref{main1} (respectively, Theorem \ref{main2}). The proof of Theorem \ref{main1} (respectively, Theorem \ref{main2}) is finalised in Subsection \ref{ss:p1} (respectively, Subsection \ref{ss:end}). Technical details are deferred to Subsection \ref{ss:exp} (bounds for all moments of $\rho_h$, and exponentially decaying bound for the negative part $\rho^-_h$), and Appendix \ref{app:a} (deterministic finite difference arguments and relevant It\^o calculus). The statements of results for finite element schemes are given in Appendix \ref{FEM_Appendix}. 
Finally, Section \ref{num} contains numerical simulations associated with Theorem \ref{main1}, using a first-order finite difference discretisation (i.e., $p=1$) in the one-dimensional case $d=1$.

\section{Analysis for finite difference discretisations}\label{fdiff}

\subsection{Notation}\label{ss:not}

\emph{Domain and function spaces}. Let $\mathbb{N}\ni d\leq 3$, and let $\mathbb{T}^d:=[-\pi,\pi)^d$. Let $h:=2\pi/L$, for some $L\in2\mathbb{N}$, be the discretisation parameter of the periodic square grid
$$
G_{h,d}:= h\mathbb{Z}^d \cap \mathbb{T}^d = \{-\pi,-\pi+h,\dots,\pi-h\}^d.
$$
We always work with periodic functions (defined either on $\mathbb{T}^d$ or $G_{h,d}$). From now on, this fact will be implicitly assumed and no longer stated. In particular, we abbreviate $C^{\beta}=C^{\beta}(\mathbb{T}^d)$ and $W^{r,p}=W^{r,p}(\mathbb{T}^d)$.
We use bold characters to denote vector fields.

For $m\in\mathbb{N}$, let $[L^2(G_{h,d})]^m$ be the space of $\mathbb{R}^m$-valued functions defined on $G_{h,d}$. Such space is endowed with the inner product
$$
(\f{u}_h,\f{v}_h)_h:=\sum_{\f{x}\in G_{h,d}}{h^d\f{u}_h(\f{x})\cdot\f{v}_h(\f{x})},\qquad \f{u}_h,\f{v}_h\in [L^2(G_{h,d})]^m,
$$
and admits an orthonormal basis $\{\f{e}^m_{\f{x},\ell}\}_{(\f{x},\ell)\in (G_{h,d},\{1,\dots,m\})}$, whose elements are defined as 
\begin{align*}
\f{e}^m_{\f{x},\ell}(\f{y})=h^{-d/2}\delta_{\f{x},\f{y}}\f{f}_\ell,
\end{align*}
where $\{\f{f}_{\ell}\}_{\ell=1}^{d}$ is the canonical basis of $\mathbb{R}^d$. If $m=1$, the notation is stripped down to 
$$
{e}_{\f{x}}(\f{y}) = h^{-d/2}\delta_{\f{x},\f{y}}.
$$

\emph{Interpolator operator}.
For $\f{\phi}\in [C^0]^m$, we define $\mathcal{I}_{h}\f{\phi}\in [L^2(G_{h,d})]^m$ as the function agreeing with $\f{\phi}$ on $G_{h,d}$. When there is no ambiguity, we simply write $\f{\phi}$ instead of $\mathcal{I}_h\f{\phi}$.

\emph{Discrete differential operators}.
We use the notation $\partial_{h,x_\ell}$ to denote a finite difference operator approximating the partial derivative $\partial_{x_\ell}$. We denote by $\nabla_h:=[\partial_{h,x_{1}},\dots,\partial_{h,x_{d}}]$ the associated finite difference gradient operator. Furthermore, for each $\ell$, we define the discrete second partial derivative $D^2_{h,x_{\ell}}$ as the operator for which the standard integration by parts formula 
\begin{align}\label{e:1}
(D^2_{h,x_{\ell}}\rho_h,v_h)_h = -(D_{h,x_{\ell}}\rho_h,D_{h,x_{\ell}}v_h)_h
\end{align}
holds, where $D_{h,x_{\ell}}$ is some (possibly different) finite difference operator approximating the partial derivative $\partial_{x_\ell}$. We abbreviate $\nabla_{D,h}:=[D_{h,x_{1}},\dots,D_{h,x_{d}}]$. 
As a result of \eqref{e:1}, the discrete operators $D^2_{h,x_\ell}$ are symmetric (in the sense of finite difference operators). We abbreviate 
\begin{align*}
\Delta_h:=\sum_{\ell=1}^{d}{D^2_{h,x_{\ell}}}
\end{align*} to indicate the discrete Laplace operator. Specific details on $\nabla_h$ and $\Delta_h$ will be provided in the following subsection.
\begin{rem}
The operators $\nabla_{h}$ and $\nabla_{D,h}$ (both providing an approximation of the continuous gradient $\nabla$) may be different, and have different uses in our discretised Dean--Kawasaki model (Definition \ref{def:1} below). The operator $\nabla_{h}$ is deployed in the noise, while the operator $\nabla_{D,h}$ in the integration by parts formula \eqref{e:1}.
\end{rem}
For reasons which will become clear in Subsection \ref{ideas} (see \emph{Block 3} therein), we set the notation for suitable continuous and discrete backwards heat flows. Specifically, for a sufficiently regular function $\varphi$ and a final time $T$, we denote by $\phi^t$ the solution the continuous backwards heat equation
\begin{align}\label{eq:28}
\partial_t\phi^t=-\frac{1}{2}\Delta\phi^t\quad\mbox{on }\mathbb{T}^d\times (0,T),
\end{align}
with final datum $\phi^T=\varphi$. Analogously, we denote by $\phi_h^t$ the solution to the discrete backwards heat equation
\begin{align}\label{eq:29}
\partial_t\phi_{h}^t=-\frac{1}{2}\Delta_h\phi_{h}^t\quad\mbox{on }G_{h,d}\times (0,T),
\end{align}
with final datum $\phi_h^T=\mathcal{I}_h\varphi$.
In the following, we also use the alternative notation $\mathcal{P}^z(\varphi):=\phi^{T-z}$ (respectively, $\mathcal{P}_h^z(\mathcal{I}_h\varphi):=\phi_{h}^{T-z}$), to stress that $\mathcal{P}^z(\varphi)$ (respectively, $\mathcal{P}_h^z(\mathcal{I}_h\varphi)$) is the result of evolving a backwards heat equation (respectively, a discrete backwards heat equation) starting from $\varphi$ (respectively, from $\mathcal{I}_h\varphi$) for a timespan $z$. 

For $y\in\mathbb{R}$, we define $y^{+}:=\max\{y;0\}$ and $y^{-}:=-\min\{y;0\}$. In addition, as usual, we use the letter $C$ to denote a generic constant, whose value may change from line to line in the computations.

\subsection{Assumptions and discretised Dean--Kawasaki model}\label{ss:ass}

\begin{customthm}{FD1}[Discrete differential operators]\label{ass:1}
Let $p\in\mathbb{N}$ be fixed. We make the following assumptions on the discrete operators $\partial_{h,x_\ell}$  and $D^2_{h,x_\ell}$:
\begin{itemize}[leftmargin=0.9 cm] 
\item the discrete operators $\partial_{h,x_\ell}$ and $D^2_{h,x_\ell}$ are finite difference operators of order $p+1$. Explicitly, this means that
\begin{align}\label{e:2}
\left|\partial_{h,x_\ell}\mathcal{I}_{h}{\phi}(\f{x})-\partial_{x_\ell}\phi(\f{x})\right| & \leq C\|\phi\|_{C^{p+1}}h^{p+1},\qquad \f{x}\in G_{h,d},\qquad \ell\in\{1,\dots,d\},\\
\left|D^2_{h,x_\ell}\mathcal{I}_{h}{\phi}(\f{x})-D^2_{x_\ell}\phi(\f{x})\right| & \leq C\|\phi\|_{C^{p+2}}h^{p+1},\qquad \f{x}\in G_{h,d},\qquad \ell\in\{1,\dots,d\},
\end{align}
for any $\phi\in C^{p+2}(\mathbb{T}^d)$;
\item The operators $\partial_{h,x_\ell}$ and $D^2_{h,x_\ell}$ commute.
\end{itemize}
\end{customthm}
\begin{customthm}{FD2}[Brownian particle system and initial datum of Dean--Kawasaki dynamics]\label{ass:2}
Let $p$ be as in Assumption \ref{ass:1}. We assume to have $N\in\mathbb{N}$ independent $d$-dimensional Brownian motions $\{\f{w}_k\}_{k=1}^{N}$ moving in $\mathbb{T}^d$. Moreover: 
\begin{itemize}[leftmargin=0.9 cm]
\item\label{a}
the initial positions $\{\f{w}_k(0)\}_{k=1}^{N}$ are deterministic; 
\item\label{c} there exists a deterministic function $\rho_{0,h}\in L^2(G_{h,d})$ (which will serve as the initial datum of the discretised Dean--Kawasaki dynamics in  Definition \ref{def:1} below), satisfying the following properties: 
\begin{itemize}[leftmargin=0.9 cm]
\item there exist $h$-independent constants $ \rho_{min}$ and $\rho_{max}$ such that
\begin{align*}
0< \rho_{min}\leq \rho_{0,h} \leq \rho_{max};
\end{align*}
\item the empirical density of the initial configuration $\mu^N_0:=N^{-1}\sum_{k=1}^{N}{\delta_{\f{w}_k(0)}}$ approximates $\rho_{0,h}$ with accuracy $p+1$, in the sense that the inequality
\begin{align}\label{eq:401a}
& \left|\langle\mu^N_0,\eta\rangle - (\rho_{0,h},\mathcal{I}_h\eta)_h\right| \nonumber\\
& \quad = \left|N^{-1}\sum_{k=1}^{N}{\eta(\f{w}_k(0))} - (\rho_{0,h},\mathcal{I}_h\eta)_h\right| \leq Ch^{p+1}\|\eta\|_{C^{p+1}},
\end{align}
holds for each function $\eta\in C^{p+1}$.
\end{itemize}
\end{itemize}
\end{customthm}
\begin{customthm}{FD3}[Scaling of relevant parameters]\label{ass:3}
We assume the scaling
\begin{align}\label{eq:202}
h\geq C(d,\rho_{min},\rho_{max})N^{-1/d} |\log N|^{2/d}(1+T),
\end{align}
for some $T>0$, and where $\rho_{min}$ and $\rho_{max}$ have been introduced in Assumption \ref{ass:2}. This scaling will be needed to produce an exponentially decaying estimate associated with $\rho_h^{-}$, see \eqref{r:3-a} below.
\end{customthm}
\begin{customthm}{FD4}[Mean-field limit]\label{ass:4} The solution to the discrete heat equation 
\begin{equation}
 \label{e:8}
  \left\{\quad
    \begin{aligned}
    \!\!\!\!\! \partial_t \overline{\rho}_h&= \frac{1}{2}\Delta_h\overline{\rho}_h, \\
    \!\!\!\!\! \overline{\rho}_h(0)&= \rho_{0,h},
    \end{aligned}
  \right.
\end{equation}
is such that $\rho_{min}\leq \overline{\rho}_h \leq \rho_{max}$ (where $\rho_{min}$ and $\rho_{max}$ have been introduced in Assumption \ref{ass:2}) for all times up to $T$ (where $T$ has have been introduced in Assumption \ref{ass:3}).
\end{customthm}

We can now state the precise definition of our finite difference Dean--Kawasaki model.

\begin{customthm2}{FD-DK}[Finite difference Dean--Kawasaki model of order $p+1$]\label{def:1}  Assume the validity of Assumptions \ref{ass:1}--\ref{ass:4}. We say that the $L^2(G_{h,d})$-valued process $\rho_h$ solves a finite difference Dean--Kawasaki model of order $p+1$ if it solves the system of stochastic differential equations
\begin{equation}
  \left\{
    \begin{aligned}\label{e:3}
     \m \left(\rho_h,e_{\f{x}}\right)_h & = \frac{1}{2}\left(\Delta_h \rho_h,e_{\f{x}}\right)_h\m t - N^{-1/2}\!\!\!\!\!\!\!\!\!\sum_{(\f{y},\ell)\in(G_{h,d},\{1,\dots,d\})}{\!\!\!\!\left(\mathcal{F}_\rho\f{e}^d_{h,\f{y},\ell},\nabla_h e_{\f{x}}\right)_h\m\beta_{(\f{y},\ell)}},\quad \forall e_{\f{x}},  \\
      \rho_h(0) & = \rho_{0,h},
    \end{aligned}
    \right.
\end{equation}
where $\{\beta_{(\f{y},\ell)}\}_{(\f{y},\ell)\in(G_{h,d},\{1,\dots,d\})}$ are standard independent Brownian motions, and where $\mathcal{F}_\rho\in L^2(G_{h,d})$ is defined as 
\begin{align}\label{e:13}
\mathcal{F}_\rho(\f{x}):=\sqrt{\rho_h^{+}(\f{x})},\qquad \forall \f{x}\in G_{h,d}.
\end{align}
\end{customthm2}

\begin{rem}
If \eqref{e:8} admits a discrete maximum principle, then Assumption \ref{ass:4} is satisfied for any $T>0$ and any non-negative datum $\rho_{0,h}$. For example, the discrete maximum principle applies for the second-order symmetrical discrete Laplace operator 
\begin{align}\label{e:915}
\Delta_{h}f(\f{x}):=\frac{-2d\hspace{0.005pc}f_h(\f{x})+\sum_{\f{y}\sim\f{x}}{f_h(\f{y})}}{h^2},
\end{align}
where $\f{y}\sim\f{x}$ indicates that $\f{y}$ and $\f{x}$ are adjacent grid points.


\end{rem}

\begin{rem}

One may also omit the contribution $(1+T)$ in the scaling \eqref{eq:202}, at the expense of obtaining results with a worse dependency on the final time $T$. We are not interested in optimising time dependencies in this work, and we simply include the term $1+T$ in order to get cleaner final results.

\end{rem}

\subsection{Key ideas behind the proofs of the main results}\label{ideas}

The proofs of Theorems \ref{main1} and \ref{main2} are of inductive type. In order to simplify their exposition, it is useful to first list a skeleton of the main building blocks.\\

\emph{Block 1. Discrete Dean--Kawasaki model: cross-variation analysis}.
At their core, both proofs use basic It\^o calculus to describe the time evolution of suitable nonlinear functionals $\psi$ of the quantities
\begin{align}\label{b:1}
(\rho_h,\phi_h)_h,\qquad (\mu^N,\phi),
\end{align}
and of their expected values, where $\phi_h$ and $\phi$ are suitable test functions.
The quantities in \eqref{b:1} are linear functionals of $\rho_h$ and $\mu^N$, respectively. What is crucial, is that the cross-variation of the processes \eqref{b:1} are -- up to a small error -- also linear functionals of $\rho_h$ and $\mu^N$. The argument for $\mu^N$ is straightforward, and we can thus defer it to the proofs themselves. As for $\rho_h$, we use Definition \ref{def:1} to write
\begin{align}\label{b:2}
\m (\rho_h, \phi_{i,h})_h & = \frac{1}{2}(\Delta_h\rho_{h},\phi_{i,h})_h\m t- N^{-1/2}\!\!\!\!\!\!\!\!\sum_{(\f{y},\ell)\in(G_{h,d},\{1,\dots,d\})}{\!\!\!\left(\mathcal{F}_\rho\f{e}^d_{h,\f{y},\ell},\nabla_h \phi_{i,h}\right)_h\m\beta_{(\f{y},\ell)}}
\end{align}
for two different test functions $\phi_{i,h}$, $i\in\{1,2\}$. Using the It\^o formula and the Parseval identity in $[L^2(G_{h,d})]^d$, one finds that the cross-variation of the stochastic noise of \eqref{b:2} is 
\begin{align}
& \left\langle\sum_{(\f{y},\ell)\in(G_{h,d},\{1,\dots,d\})}{\!\!\!\!\!\left(\mathcal{F}_\rho\f{e}^d_{h,\f{y},\ell},\nabla_h \phi_{1,h}\right)_h\dot \beta_{(\f{y},\ell)}},\,
\sum_{(\f{y},\ell)\in(G_{h,d},\{1,\dots,d\})}{\!\!\!\!\!\left(\mathcal{F}_\rho\f{e}^d_{h,\f{y},\ell},\nabla_h \phi_{2,h}\right)_h\dot \beta_{(\f{y},\ell)}}\right\rangle\nonumber\\
& \quad = \sum_{(\f{y},\ell)\in(G_{h,d},\{1,\dots,d\})}{\left(\f{e}^d_{h,\f{y},\ell},\mathcal{F}_\rho\nabla_h \phi_{1,h}\right)_h\left(\f{e}^d_{h,\f{y},\ell},\mathcal{F}_\rho\nabla_h \phi_{2,h}\right)_h}\label{e:7b}\\
& \quad = \left(\mathcal{F}^2_\rho,\nabla_h \phi_{1,h}\cdot \nabla_h \phi_{2,h}\right)_h \label{e:7c}\\
& \quad \stackrel{\mathclap{\eqref{e:13}}}{=}\, \left(\rho^{+}_h,\nabla_h \phi_{1,h}\cdot \nabla_h \phi_{2,h}\right)_h\label{e:7}\\
& \quad = \left(\rho_h,\nabla_h \phi_{1,h}\cdot \nabla_h \phi_{2,h}\right)_h + \left(\rho^{-}_h,\nabla_h \phi_{1,h}\cdot \nabla_h \phi_{2,h}\right)_h.\label{e:7d}
\end{align}
The first term in \eqref{e:7d} is indeed a linear functional of $\rho_h$. The second term (which we will show to be negligible for suitable scaling regimes, see Subsection \ref{ss:exp}) takes into account the \emph{a priori} lack of knowledge concerning the non-negativity of solutions to the discrete Dean--Kawasaki model \eqref{e:3}. We also stress that the validity of the computations above is independent of the order of the finite difference scheme (i.e., $p$).

Expression \eqref{e:7} crucially preserves the cross-variation structure associated with the continuous Dean--Kawasaki \eqref{DeanKawasaki} for nonnegative densities. More precisely, \emph{formally} testing \eqref{DeanKawasaki} with a smooth test functions $\phi_i$, $i\in\{1,2\}$, gives
\begin{align}\label{e:5}
\int_{\mathbb{T}^d}{\partial_t \rho\,\phi_i}\m \f{x}& = \frac{1}{2}\int_{\mathbb{T}^d}{\Delta \rho\,\phi_i \m \f{x}}-N^{-1/2}\int_{\mathbb{T}^d}{\sqrt{u}\,\f{\xi}\cdot\nabla\phi_i \m \f{x}}\nonumber\\
& = \frac{1}{2}\int_{\mathbb{T}^d}{\Delta \rho\,\phi_i \m \f{x}} -N^{-1/2}\sum_{{\f{s}}\in\mathbb{Z}}{\int_{\mathbb{T}^d}{\sqrt{\rho}\,\f{e}_{\f{s}}\cdot\nabla\phi_i \m \f{x}\dot\beta_{\f{s}}}},
\end{align}
where the last inequality if justified by the representation $\f{\xi}=\sum_{\f{s}\in\mathbb{Z}^d}\f{e}_{\f{s}}\dot\beta_{\f{s}}$, where $\{\f{e}_{\f{s}}\}_{{\f{s}}\in\mathbb{Z}^d}$ is an orthonormal basis of $[L^2(\mathbb{T}^d)]^d$ and $\{\beta_{\f{s}}\}_{\f{s}\in\mathbb{Z}^d}$ are independent Brownian motions. The noise cross-variation is then obtained using the It\^o formula and the Parseval idendity -- this time in $[L^2(\mathbb{T}^d)]^d$ -- to obtain
\begin{align}\label{e:6}
& \left\langle\sum_{\f{k}\in\mathbb{Z}^d}{\int_{\mathbb{T}^d}{\sqrt{\rho}\f{e}_{\f{k}}\cdot\nabla\phi_1\m \f{x}\dot\beta_{\f{k}}}},\sum_{{\f{l}}\in\mathbb{Z}^d}{\int_{\mathbb{T}^d}{\sqrt{\rho}\f{e}_{\f{l}}\cdot\nabla\phi_2 \m \f{x}\dot\beta_{\f{l}}}}\right\rangle\nonumber\\
& \quad = \sum_{{\f{k}}\in\mathbb{Z}}{\int_{\mathbb{T}^d}{\sqrt{\rho}\f{e}_k\cdot\nabla\phi_1}\m \f{x}}\int_{\mathbb{T}^d}{\sqrt{\rho}\f{e}_k\cdot\nabla\phi_2\m \f{x}} = \int_{\mathbb{T}^d}{\rho \nabla\phi_1\cdot\nabla\phi_2\m \f{x}},
\end{align}
and thus the cross-variations \eqref{e:6} and \eqref{e:7} are (modulo positive part $\rho_h^{+}$) structurally identical.\\

\emph{Block 2. Numerical error}. There are two contributions to the numerical error, namely:
\begin{enumerate}
\item[-] the difference of initial data $\mu^N_0$ and $\rho_h(0)$, and
\item[-] the difference in the evolution of test functions (say, $\phi$ and $\phi_h$),
\end{enumerate}
and both are proportional to $h^{p+1}$. While the first contribution has the correct bound by Assumption \ref{ass:2}, the second contribution needs to be estimated: The main difficulty is that the interpolation of the test function arising from the cross-variation of the second quantity in \eqref{b:1} (i.e., $\mathcal{I}_h(\nabla\phi_1\cdot\nabla\phi_2)$) does not coincide -- in general -- with $\nabla_{h}\phi_{1,h}\cdot\nabla_h\phi_{2,h}$ (i.e., the cross-variation of the first quantity in \eqref{b:1}). We therefore need to show the bound
$$
\left|\mathcal{I}_h(\nabla\phi_1\cdot\nabla\phi_2)-\nabla_{h}\phi_{1,h}\cdot\nabla_h\phi_{2,h}\right| \lesssim h^{p+1}
$$
in order not to lose $h$-regularity in consecutive steps of our inductive proofs (more details in Block 5 below).
The necessary tools for this task are contained in Subsection \ref{ss:approx}. \\

\emph{Block 3. Deterministic dynamics of the test functions}. As we are interested only in the analysis of the fluctuations for the Dean--Kawasaki model, it is convenient to choose the deterministic functions $\psi$, $\phi$, $\phi_h$ in such a way that as many drift terms as possible in relevant It\^o differentials vanish. This is the reason behind the choice of the backwards heat equation \eqref{eq:28} (respectively, \eqref{eq:29}) for $\phi$ (respectively, for $\phi_h$), which directly compensates the diffusive nature of the particle system (respectively, of the Dean--Kawasaki model). In practice, this is reflected in the useful equalities (which follow from Lemma \ref{lem:10})
\begin{align}\label{b:6a}
(\rho_h(t),\phi_h^t)_h-(\rho_h(0),\phi_h^0)_h& = (\rho_h(t)-\mean{\rho_h(t)},\phi_h^t)_h,\\
\langle \mu^N_t, \phi^t \rangle - \langle \mu^N_0, \phi^0 \rangle & = \langle \mu^N_t-\mean{\mu^N_t}, \phi^t \rangle,\label{b:6b}
\end{align}
for $\phi, \phi_h$ as in \eqref{eq:28}, \eqref{eq:29}. 
The discussion for $\psi$ in the case of Theorem \ref{main1} is conceptually analogous, but technically more involved, and is devolved to the proof itself. As for Theorem \ref{main2}, $\psi$ is chosen to be static, therefore this discussion does not apply.

We expand these considerations in Appendix \ref{ss:mom}.\\

\emph{Block 4. Stretched exponential bounds for centred moments of the particle system and the Dean--Kawasaki solution}. This block associates the scaling regime of Assumption \ref{ass:3} to the validity of the moment bounds 
\begin{align*}
\max_{t\in[0,T]}{\mean{\left|\prod_{m=1}^{M}{\langle \mu^{N}_{T_m}-\mean{\mu^{N}_{T_m}},\varphi_{m}\rangle^{j_m}}\right|}} & \leq  \left\{N^{-1} T \right\}^{j/2}j^{j}\left[\prod_{m=1}^{M}{\|\nabla\varphi_m\|_{\infty}^{j_m}}\right]
\end{align*}
and
\begin{align*}
&\max_{t\in[0,T]}{\mean{\left|\prod_{m=1}^{M}{\left(\rho_h(T_m)-\mean{\rho_h(T_m)},\mathcal{I}_h\varphi_{m}\right)_h^{j_m}}\right|}} \\
&\quad \leq \left\{2N^{-1} TC\left(d,\rho_{min},\rho_{max}\right)\right\}^{j/2}j^{3j}\left(\prod_{m=1}^{M}{\|\varphi_{m}\|_{C^{1+\Theta}}^{j_m}}\right),
\end{align*}
where $T_1,\dots,T_m\in [0,T]$, and $\Theta$ was introduced in \eqref{b:7}. The difference in the norms of the test functions stems from a difference in underlying mathematical arguments (depending on the circumstance, we will either use the maximum principle or the Sobolev embedding Theorem).
The necessary tools for this point are contained in Subsection \ref{ss:mom}.\\

\emph{Block 5. Inductive argument}.
Block 1 essentially states that computing cross-variations of discrete Dean--Kawasaki models yields linear functionals \eqref{b:1}, as well as negligible corrections related to the negative part $\rho^{-}_h$. Taking Block 2 also into account, this leads to the following crucial observation.

\emph{The It\^o correction term in the It\^o differential of smooth enough nonlinear functions $\psi$ applied to \eqref{b:1} and their expected values is a sum of:
\begin{enumerate}
\item[-] negligible terms featuring $\rho^{-}_h$ and the numerical error, and
\item[-] yet another (possibly different) nonlinear function $\tilde{\psi}$ applied to \eqref{b:1} and their expected values.
\end{enumerate}
}
This property allows to set up both proofs using an induction argument whose inductive step is the change in nonlinear function (from $\psi$ to $\tilde{\psi}$): the residual terms (featuring $\rho^{-}_h$ and the numerical error) are estimated at each step, and are not fed to the next step.\\

\subsection{The key step for the accuracy estimate for fluctuations in Theorem \ref{main1}}\label{key1}

For use in the next proposition, we define the two function spaces $\mathcal{L}_\text{pow,r}^{q},{\tilde{\mathcal{L}}}_\text{pow,r}^{q}$ as
\begin{align*}
\mathcal{L}_\text{pow,r}^{q} := \left\{ \psi:\mathbb{R}^M\rightarrow \mathbb{R}\colon \|\psi\|_{\mathcal{L}_\text{pow,r}^{q}}
:=\max_{0\leq \tilde q\leq q} \left\|(1+|\cdot|^2)^{-r/2} D^{\tilde q}\psi(\cdot)\right\|_{L^\infty}
<\infty\right\}, \\
{\tilde{\mathcal{L}}}_\text{pow,r}^{q} := \left\{ \psi:\mathbb{R}^M\rightarrow \mathbb{R}\colon \|\psi\|_{\mathcal{L}_\text{pow,r}^{q}}
:=\max_{1\leq \tilde q\leq q} \left\|(1+|\cdot|^2)^{-r/2} D^{\tilde q}\psi(\cdot)\right\|_{L^\infty}
<\infty\right\}.
\end{align*}

Furthermore, we emphasise that we use the shorthand notations
\begin{align*}
\langle \mu_{\f{T}}^N-\mathbb{E}[\mu_{\f{T}}^N], \f{\phi} \rangle 
& := \begin{pmatrix}
\langle \mu_{T_1}^N-\mathbb{E}[\mu_{T_1}^N], \phi_1 \rangle 
\\
\vdots
\\
\langle \mu_{T_M}^N-\mathbb{E}[\mu_{T_M}^N], \phi_M \rangle 
\end{pmatrix},\\
\big((\rho_h(\f{T})-\mathbb{E}[\rho_h(\f{T})]),\f{\phi}_h\big)_h
& := 
\begin{pmatrix}
\big((\rho_h(T_1)-\mathbb{E}[\rho_h(T_1)]),\phi_{1,h}\big)_h
\\
\vdots
\\
\big((\rho_h(T_M)-\mathbb{E}[\rho_h(T_M)]),\phi_{M,h}\big)_h
\end{pmatrix}, \\
\big((\rho_h-\mathbb{E}[\rho_h]),\f{\phi}_h\big)_h(\f{T})
& := 
\big((\rho_h(\f{T})-\mathbb{E}[\rho_h(\f{T})]),\f{\phi}_h^{\f{T}}\big)_h,\\
t \wedge \f{T}
& := (t\wedge T_1,\dots,t\wedge T_M),
\end{align*}
i.e., we implicitly multiply vectors in an element-wise fashion respectively evaluate vectorial functions by a vector of (time) parameters in an element-wise way. 

Theorem~\ref{main1} will be seen to be an easy consequence of the following crucial proposition and an inductive argument.

\begin{proposition}\label{prop:recursive}
Let $\mu_t^N$ denote the empirical measure of $N$ independent Brownian particles as defined in \eqref{EmpiricalMeasure}.

Let $\rho_h$ be a solution to the Dean--Kawasaki equation discretised using finite differences on a uniform grid \eqref{e:3}. Suppose furthermore that Assumption \ref{ass:1} (details of operators $\Delta_h$ and $\nabla_h$), Assumption \ref{ass:2} (initial condition on Brownian particle system), Assumption \ref{ass:3} (scaling assumptions), and Assumption \ref{ass:4} (positivity-preserving properties of mean-field limit) hold.

Let $M$, $p\in \mathbb{N}$, $q\in \mathbb{N}$, and $r\in\mathbb{N}_0$. Let $\psi:\mathbb{R}^M\rightarrow \mathbb{R}$ satisfy $\psi\in \mathcal{L}_\text{pow,r}^{q+2}$. Let $\f{\varphi} \in [W^{2+p+\Theta,\infty}]^M$. Finally, let $\f{T}=(T_1,\dots,T_M)$ such that $0<T_1 \leq \ldots \leq T_M\leq T$.

Then there exist test functions $\tilde \psi^t_{kl}$, $\tilde{\f{\phi}}^t_{kl}$, $\psi^0$, and $\f{\phi}^0$ as well as $\f{\tilde T}_{kl}\in \mathbb{R}^{M+1}$ such that
\begin{subequations}
\begin{align}
\label{InductiveEquationDistributionEmpiricalMeasure}
&
\mathbb{E} \Bigg[\psi \bigg(N^{1/2} \left\langle \mu_{\f{T}}^N-\mathbb{E}[\mu_{\f{T}}^N], \f{\varphi} \right\rangle\bigg)\Bigg]
\\&
\nonumber
=
\psi^0(0)
+\frac{1}{2N^{1/2}} \sum_{k,l=1}^M \int_0^{T_k\wedge T_l} \mathbb{E} \Bigg[\tilde \psi^t_{kl} \bigg(N^{1/2} \left\langle  \mu_{t\wedge \f{\tilde T}}^N-\mathbb{E}[\mu_{t\wedge \f{\tilde T}}^N],\tilde{\f{\phi}}^t_{kl} \right\rangle \bigg) \Bigg] \,\emph{\m} t
\end{align}
and
\begin{align}
\nonumber
&
\mathbb{E} \Bigg[\psi \bigg(N^{1/2}\big((\rho_h(\f{T})-\mathbb{E}[\rho_h(\f{T})]),\mathcal{I}_h\f{\varphi}\big)_h\bigg)\Bigg]
\\&\nonumber
= \mathbb{E} \Bigg[\psi \bigg(N^{1/2}\big((\rho_h-\mathbb{E}[\rho_h]),\f{\phi}_h\big)_h(\f{T}) \bigg)\Bigg]
\\&
\label{InductiveEquationDistributionDK}
=
\psi^0(0)
+\frac{1}{2N^{1/2}}  \sum_{k,l=1}^M \int_0^{T_k\wedge T_l} \mathbb{E} \Bigg[\tilde \psi^t_{kl} \bigg(N^{1/2}\big( \rho_h-\mathbb{E}[\rho_h], \mathcal{I}_h\tilde{\f{\phi}}_{kl}\big)_h (t\wedge \f{\tilde T}_{kl})\bigg) \Bigg] \,\emph{\m} t
\\&~~~~~
\nonumber
+\mathrm{Err}_{num}
+\mathrm{Err}_{neg}
\end{align}
\end{subequations}
hold.
Here, $\tilde{\f{\phi}}^t_{kl}$ is subject to the estimate
\begin{subequations}
\begin{align}
\label{EstimateTildePhi}
\|\tilde{\f{\phi}}^t_{kl}\|_{W^{q-1,\infty}}
&\leq C(q,M,\|\f{\varphi}\|_{W^{q,\infty}})
\quad\text{for all }t\leq T,
\end{align}
while, if $q\geq 2$, $\tilde \psi^t$ is subject to the estimate
\begin{align}
\label{EstimateTildePsiFirst}
\|\tilde \psi^t\|_{\mathcal{L}_\text{pow,r+1}^{q-2}}
&\leq C(q,r,M,\|\f{\varphi}\|_{W^{1,\infty}}^2, T) \|\psi\|_{{\tilde{\mathcal{L}}}_\text{pow,r}^{q}}
\quad\text{for all }t\leq T.
\end{align}
Furthermore, $\mathrm{Err}_{num}$ and $\mathrm{Err}_{neg}$ are subject to the estimate 
\begin{align}
\label{EstimateErrNum}
|\mathrm{Err}_{num}| &\leq C(M,\rho_{max},r,\|\f{\varphi}\|_{C^{p+2+\Theta}},T) \big(\|\psi\|_{{\tilde{\mathcal{L}}}^{2}_{pow,r}} + N^{-1/2} \|D\psi\|_{{\tilde{\mathcal{L}}}^{2}_{pow,r}}\big) 
 h^{p+1},
\\
\label{EstimateErrNeg}
|\mathrm{Err}_{neg}| &\leq C(M,\rho_{max},r,\|\f{\varphi}\|_{C^{1+\Theta}},T) \|\psi\|_{{\tilde{\mathcal{L}}}^{2}_{pow,r}} \mathcal{E}\!\left(N,h\right),
\end{align}
where $\mathcal{E}\!\left(N,h\right)$ is defined in \eqref{b:8}.

Under the additional assumption that $\|\varphi_k\|_{L^2}=1$ and $\int_\domain \varphi_k \,\emph{d}\f{x}=0$ for all $k$, that $\int_\domain \varphi_k \varphi_l \,\emph{d}\f{x}=0$ whenever $T_k=T_l$, and that 
$$
\overline{m}_{(1/2)T_1} := \inf_{x\in \domain,t>\frac{1}{2}T_1} \mathbb{E}[\mu^N_t](x) \geq  \rho_{min}>0,
$$
we have the additional bounds
\begin{align}
\label{EstimateTildePsiSecond}
\|\tilde \psi^t_{kl}\|_{\mathcal{L}_\text{pow,r+1}^{q-1}}
&\leq
\frac{C(q,r,M,\|\f{\varphi}\|_{W^{1,\infty}}^2, T)}{\sqrt{  \rho_{min} \min\big\{\min_{m:T_m\geq t}(T_m-t)~,~\min_{k,l:T_k\neq T_l}|T_k-T_l|\big\} }} \|\psi\|_{{\tilde{\mathcal{L}}}_\text{pow,r}^{q}}
\end{align}
and
\begin{align}
|\mathrm{Err}_{num}|
&\leq C(r,\rho_{max},\rho_{min},d,M,\|\f{\varphi}\|_{C^{p+2+\Theta}}) \big(\|\psi\|_{{\tilde{\mathcal{L}}}^{1}_{pow,p}} + N^{-1/2} \|D\psi\|_{{\tilde{\mathcal{L}}}^{1}_{pow,p}}\big) \nonumber\\
\label{EstimateErrNumSecond}
&\quad\quad\quad \times\frac{1}{\sqrt{  \rho_{min} \min_{k,l:T_k\neq T_l}|T_k-T_l| }}
h^{p+1},
\end{align}
as well as
\begin{align}
|\mathrm{Err}_{neg}| &\leq
C(M,\rho_{max},\rho_{min},d,r,\|\f{\varphi}\|_{C^{1+\Theta}},T)
\|\psi\|_{{\tilde{\mathcal{L}}}^{1}_{pow,r}} \mathcal{E}\!\left(N,h\right)
\nonumber
\\&
\label{EstimateErrNegSecond}
\quad\quad\quad\times\frac{1}{\sqrt{  \rho_{min} \min_{k,l:T_k\neq T_l}|T_k-T_l| }}.
\end{align}
\end{subequations}
\end{proposition}

The proof is split into four steps. In Step 1, we provide deterministic estimates of suitable backwards diffusive equations of relevance, as well as basic stochastic estimates associated with the Dean--Kawasaki dynamics \ref{def:1}. Step 2 (respectively, Step 3) is devoted to obtaining \eqref{InductiveEquationDistributionEmpiricalMeasure} (respectively, \eqref{InductiveEquationDistributionDK}). Step 4 bounds the residual terms $\mathrm{Err}_{num}, \mathrm{Err}_{neg}$ in \eqref{InductiveEquationDistributionDK}.

\begin{proof}[Proof of Proposition \ref{prop:recursive}]

{\bf Step 1: definitions and elementary estimates.}
Let $\phi_m^t$ satisfy the backwards heat equation \eqref{eq:28} subject to $\phi_m^{T_m}:=\varphi_i$.
Define the function $\psi^t:\mathbb{R}^M\rightarrow \mathbb{R}$ by setting $\psi^T:=\psi$ and by evolving $\psi^t$ backward in time using the backward diffusion equation
\begin{align}
\label{EvolutionPsi}
-\partial_t \psi^t = \frac{1}{2}\sum_{k,l=1}^M \bigg(\chi_{t\leq T_k}\, \chi_{t\leq T_l} \left\langle \mathbb{E}[\mu^N_t], \nabla \phi_{k}^t \cdot \nabla \phi_{l}^t \right\rangle \partial_k \partial_l \psi^t\bigg).
\end{align}
The purpose of the definitions of $\phi^t_m$ and $\psi^t$ will become clear in Step~2 and 3 below.
Note that these definitions entail
\begin{align*}
D^{\tilde q} \psi^t(\f{y})
&=\int_{\mathbb{R}^M} \frac{1}{(\det (2\pi\Lambda))^{1/2}}\exp\big(-\tfrac{1}{2}\Lambda^{-1} \tilde{\f{z}} \cdot \tilde{\f{z}} \big) D^{\tilde q} \psi(\f{z}-\tilde{\f{z}}) \,d\tilde{\f{z}},
\end{align*}
(where for simplicity we have assumed that the eigenvalues of $\Lambda$ are nondegenerate; otherwise, we replace the formula by its natural analogue)
with
\begin{align}
\label{CovMatrix}
\Lambda_t:=\int_t^T \frac{1}{2}\sum_{k,l=1}^M \chi_{\tilde t<T_k} \chi_{\tilde t<T_l} \left\langle \mathbb{E}[\mu^N_t], \nabla \phi_{k}^t \cdot \nabla \phi_{l}^t\right\rangle~ e_k\otimes e_l \,\m\tilde t.
\end{align}
This implies
\begin{align*}
&\big|(1+|\f{z}|^2)^{r/2} D^{\tilde q} \psi^t(\f{z})\big|
\\
&\quad \leq C(r) \int_{\mathbb{R}^M} \frac{1}{(\det (2\pi\Lambda))^{1/2}}\exp\big(-\tfrac{1}{2}\Lambda^{-1} \tilde{\f{z}} \cdot \tilde {\f{z}}\big) \big|(1+|\f{z}|^2)^{r/2} D^{\tilde q} \psi(\f{z}-\tilde{\f{z}})\big| \,\m\tilde{\f{z}}
\\
&\quad\leq C(r) \int_{\mathbb{R}^M} (1+|\tilde{\f{z}}|^2)^{r/2} \frac{1}{(\det (2\pi\Lambda))^{1/2}}\exp\big(-\tfrac{1}{2}\Lambda^{-1} \tilde{\f{z}} \cdot \tilde{\f{z}} \big) \\
& \quad\quad\quad\quad\times\big|(1+|\f{z}-\tilde{\f{z}}|^2)^{r/2} D^{\tilde q} \psi(\f{z}-\tilde{\f{z}})\big| \,\m\tilde{\f{z}},
\end{align*}
and thus
\begin{align*}
\|(1+|\cdot|^2)^{r/2} D^{\tilde q} \psi^t(\cdot)\|_{L^\infty}
\leq C(r,p) (1+|\Lambda|^{r/2}) \|(1+|\cdot|^2)^{r/2} D^{\tilde q} \psi(\cdot)\|_{L^\infty}.
\end{align*}
Observing that $|\Lambda|\leq C \sup_{t\in[0,T]} \|\f{\phi}^{t\wedge \f{T}}\|_{W^{1,\infty}}^2 T\leq C \|\f{\varphi}\|_{W^{1,\infty}}^2 T$, we conclude that
\begin{align}
\label{EstimateNormPsi}
\|\psi^t\|_{{\tilde{\mathcal{L}}}_\text{pow,r}^{q}}
\leq C(q,r,M,\|\f{\varphi}\|_{W^{1,\infty}}^2, T) \|\psi\|_{{\tilde{\mathcal{L}}}_\text{pow,r}^{q}}.
\end{align}
Arguing similarly, we deduce
\begin{align*}
&\big|(1+|{\f{z}}|^2)^{r/2} D^{\tilde q} \partial_k \psi^t({\f{z}})\big|
\\
&\quad\leq C(r) \int_{\mathbb{R}^M} |\Lambda^{-1} \tilde{\f{z}}| \frac{1}{(\det (2\pi\Lambda))^{1/2}}\exp\big(-\tfrac{1}{2}\Lambda^{-1} \tilde{\f{z}} \cdot \tilde{\f{z}} \big) \big|(1+|{\f{z}}|^2)^{r/2} D^{\tilde q} \psi({\f{z}}-\tilde{\f{z}})\big| \,\m\tilde{\f{z}}
\\
&\quad\leq C(r) \int_{\mathbb{R}^M} (1+|\tilde{\f{z}}|^2)^{r/2} |\Lambda^{-1} \tilde{\f{z}}\cdot e_k| \frac{1}{(\det (2\pi\Lambda))^{1/2}}\exp\big(-\tfrac{1}{2}\Lambda^{-1} \tilde{\f{z}} \cdot \tilde{\f{z}} \big)
\\&~~~~~~~~~~~~~~~~~~~~~\times
\big|(1+|{\f{z}}-\tilde{\f{z}}|^2)^{r/2} D^{\tilde q} \psi({\f{z}}-\tilde{\f{z}})\big| \,d\tilde{\f{z}},
\end{align*}
and therefore
\begin{align*}
\|\partial_k \psi^t\|_{\mathcal{L}_\text{pow,r}^{q}}
\leq C(q,r,M,\|\f{\varphi}\|_{W^{1,\infty}}^2, T) |\Lambda^{-1/2}e_k| \|\psi\|_{\mathcal{L}_\text{pow,r}^{q}}.
\end{align*}
Using the estimate \eqref{LambdaBound}, under the additional assumptions on the $\varphi_k$ stated above we infer
\begin{align}
\label{EstimateNormPsiDecay}
\|\partial_k \psi^t\|_{\mathcal{L}_\text{pow,r}^{q}}
\leq \frac{C(q,r,\|\f{\varphi}\|_{W^{1,\infty}}^2, T,M)}{\overline{m}_{(1/2)T_1}^{1/2} \min\big\{\min_{m:T_m\geq t}(T_m-t)~,~\min_{k,l:T_k\neq T_l}|T_k-T_l|\big\}^{1/2} } \|\psi\|_{\mathcal{L}_\text{pow,r}^{q}}
\end{align}
whenever $T_k>t$. This in particular implies \eqref{EstimateTildePsiSecond}. A similar argument yields
\begin{align}
\label{EstimateNormPsiDecay2}
\|\partial_k \partial_l \psi^t\|_{\mathcal{L}_\text{pow,r}^{q}}
\leq \frac{C(q,r,\|\f{\varphi}\|_{W^{1,\infty}}^2, T,M)}{\overline{m}_{(1/2)T_1}^{1/2} \min\big\{\min_{m:T_m\geq t}(T_m-t)~,~\min_{k,l:T_k\neq T_l}|T_k-T_l|\big\} ^{1/2}} \|\partial_k \psi\|_{\mathcal{L}_\text{pow,r}^{q}}
\end{align}
whenever $T_k,T_l>t$.

Now fix $\eta\in W^{1+\Theta}$. Let $\eta_h$ satisfy the discrete backwards heat equation \eqref{eq:29} subject to $\eta_h^T:=\mathcal{I}_h\eta$. We observe that the moment estimate
\begin{align}\label{MomentBoundsFluctuationsRhoh}
&\mathbb{E}\bigg[\sup_{t\in [0,T]} \bigg|\big(\rho_h-\mathbb{E}[\rho_h],\eta_h\big)_h(t)\bigg|^{2j} \bigg]^{1/2j}
\leq C(j,\rho_{max},\rho_{min},d) N^{-1/2}T^{1/2} \|\eta\|_{W^{1+\Theta}} 
\end{align}
holds for any $j\in \mathbb{N}$. To see this, we use \eqref{e:3} and deduce that for any $t>0$
\begin{align*}
\big(\rho_h-\mathbb{E}[\rho_h],\eta_h\big)_h(t)
=\mathcal{M}_{t}
\end{align*}
where $\mathcal{M}_t$ is a martingale satisfying $\mathbb{E}[\mathcal{M}_t]=0$ and
\begin{align*}
\langle\mathcal{M}_t,\mathcal{M}_t\rangle
&=
\frac{1}{2N}
\mathbb{E}\bigg[
\sum_{(\f{y},\ell)\in(G_{h,d},\{1,\dots,d\})} (\mathcal{F}_\rho(t) \f{e}_{h,\f{y},\ell}^d,\nabla_h \eta_h^t \big)_h (\mathcal{F}_\rho(t) \f{e}_{h,\f{y},\ell}^d,\nabla_h \eta_h^t \big)_h\bigg]
\\&
\stackrel{\eqref{e:7}}{=}
\frac{1}{2N}
\mean{\left(\rho^{+}_h(t),\nabla_h \eta_h^t \cdot \nabla_h \eta_h^t \right)_h}.
\end{align*}
Doob's martingale inequality, the moment bound $\eqref{r:2b-a}$, and the estimate 
\begin{align}\label{r:4}
\sup_{t\in[0,T]}\|\nabla_h\eta_h^t\|_{h,\infty} \leq \|\eta\|_{W^{1+\Theta}},
\end{align}
(which is, depending on $\Theta$, a consequence of either the discrete maximum principle or the Sobolev embedding theorem) yield \eqref{MomentBoundsFluctuationsRhoh}.
It is also straightforward to notice that 
\begin{align}
\label{L2EstimateInterpolation}
\Big(\mathcal{I}_h [\eta],\mathcal{I}_h [\eta]\Big)_h
\leq C \|\eta\|_{W^{1,\infty}}.
\\
\label{L2EstimateInterpolationGrad}
\|\mathcal{I}_h [\nabla \eta_1 \cdot \nabla \eta_2]\|_{L^\infty}
\leq C \|\eta\|_{W^{2,\infty}}^2.
\end{align}
Furthermore, we write
\begin{align*}
& \Big(\mathbb{E}[\rho_h(T)],\mathcal{I}_h\eta\Big)_h - \langle \mean{\mu^{N}_T},\eta\rangle \nonumber\\
& \quad\stackrel{\mathclap{\eqref{b:6a},\eqref{b:6b},\ref{ass:2}}}{=}\,\,\,\,\,\,\quad \Big(\rho_{h}(0),\mathcal{P}^T_h(\mathcal{I}_h\eta)\Big)_h - \langle \mu^{N}_0,\mathcal{P}^T\eta\rangle\nonumber\\
& \quad = \Big(\rho_{h}(0),\mathcal{P}^T_h(\mathcal{I}_h\eta)-\mathcal{I}_h[\mathcal{P}^T(\eta)]\Big)_h + \left\{\Big(\rho_{h}(0),\mathcal{I}_h[\mathcal{P}^T(\eta)]\Big)_h - \langle \mu^{N}_0,\mathcal{P}^T\eta\rangle\right\}\nonumber\\
& \quad =: T_1 + T_2,
\end{align*}
where $\mathcal{P}^{\cdot}$ and $\mathcal{P}^{\cdot}_h$ have been introduced in Subsection \ref{ss:not}.
Term $T_1$ is bounded using \eqref{eq:325a}, while $T_2$ is settled using \eqref{eq:401a} from Assumption \ref{ass:2}. Altogether, this leads to
\begin{align}
\label{ErrorEstimateHeatEquation}
&\bigg|
\Big(\mathbb{E}[\rho_h(T)],\mathcal{I}_h [\eta]\Big)_h
-  \langle \mean{\mu^{N}_T},\eta\rangle
\bigg|
\leq C \|\rho_{h}(0)\|_{h} \|\eta\|_{C^{p+1}} h^{p+1}.
\end{align}

{\bf Step 2: proof of \eqref{InductiveEquationDistributionEmpiricalMeasure}.}
Using It\^o's formula and the fact that
$
\langle \mu_t^N-\mathbb{E}[\mu_t^N], \eta \rangle
=\textstyle N^{-1}\sum_{n=1}^N (\eta(\f{w}_n(t))-\mathbb{E}[\eta(\f{w}_n(t))])
$
holds for all $\eta\in C^0$, we compute
\begin{align*}
&
\m\bigg(
\psi^t \bigg(N^{1/2}\left\langle  \mu_{t\wedge \f{T}}^N-\mathbb{E}[\mu_{t\wedge \f{T}}^N], \f{\phi}^t \right\rangle \bigg)\bigg)
\\&
=(\partial_t \psi^t) \bigg(N^{1/2}\left\langle  \mu_{t\wedge \f{T}}^N-\mathbb{E}[\mu_{t\wedge \f{T}}^N], \f{\phi}^t \right\rangle \bigg) \,\m t
\\&~~~~
+\sum_{k=1}^M
\partial_k \psi^t \bigg(N^{1/2}\left\langle  \mu_{t\wedge \f{T}}^N-\mathbb{E}[\mu_{t\wedge \f{T}}^N], \f{\phi}^t \right\rangle\bigg)
\\&~~~~~~~~~~~~~~~~~~~~~~~~~
\times
N^{-1/2} \sum_{n=1}^N \Big((\partial_t \phi_k^t)(\f{w}_n(t)) - \mathbb{E}\big[(\partial_t \phi_k^t)(\f{w}_n(t))\big] \Big) \,\m t
\\&~~~~
-\sum_{k=1}^M
\partial_k \psi^t \bigg(N^{1/2}\left\langle  \mu_{t\wedge \f{T}}^N-\mathbb{E}[\mu_{t\wedge \f{T}}^N], \f{\phi}^t \right\rangle\bigg)
N^{1/2} \left\langle \partial_t \mathbb{E}[\mu_{t\wedge T_k}^N], \phi^t_k \right\rangle \m t
\\&~~~~
+\sum_{k=1}^M
\partial_k \psi^t \bigg(N^{1/2}\left\langle  \mu_{t\wedge \f{T}}^N-\mathbb{E}[\mu_{t\wedge \f{T}}^N], \f{\phi}^t \right\rangle\bigg) \chi_{t\leq T_k}
N^{-1/2} \sum_{n=1}^N \nabla \phi_k^t(\f{w}_n(t)) \cdot \,\m \f{w}_n
\\&~~~~
+\sum_{k=1}^M
\partial_k \psi^t \bigg(N^{1/2}\left\langle  \mu_{t\wedge \f{T}}^N-\mathbb{E}[\mu_{t\wedge \f{T}}^N], \f{\phi}^t \right\rangle\bigg) \frac{1}{2} \chi_{t\leq T_k}
N^{-1/2} \sum_{n=1}^N \Delta \phi_k^t(\f{w}_n(t))
\,\m t
\\&~~~~
+\frac{1}{2} \sum_{k,l=1}^M
\partial_k \partial_l \psi^t \bigg(N^{1/2}\left\langle  \mu_{t\wedge \f{T}}^N-\mathbb{E}[\mu_{t\wedge \f{T}}^N], \f{\phi}^t \right\rangle\bigg)
\\&~~~~~~~~~~~~~~~~~~~~~~~~~
\times
\chi_{t\leq T_k} \, \chi_{t\leq T_l}
N^{-1} \sum_{n=1}^N \nabla \phi_k^t(\f{w}_n(t)) \cdot \nabla \phi_l^t(\f{w}_n(t))
\,\m t.
\end{align*}
Using the fact that $\partial_t \mathbb{E}[\mu_t^N] = \tfrac{1}{2}\Delta\mathbb{E}[\mu_t^N]$, plugging in the equation \eqref{eq:28} satisfied by $\f{\phi}^t$, and taking the expected value, we obtain
\begin{align*}
&
\m \mathbb{E}\Bigg[
\psi^t \bigg(N^{1/2}\left\langle  \mu_{t\wedge \f{T}}^N-\mathbb{E}[\mu_{t\wedge \f{T}}^N], \f{\phi}^t \right\rangle\bigg)\Bigg]
\\&
=\mathbb{E}\Bigg[(\partial_t \psi^t) \bigg(N^{1/2}\left\langle  \mu_{t\wedge \f{T}}^N-\mathbb{E}[\mu_{t\wedge \f{T}}^N], \f{\phi}^t \right\rangle\bigg)\Bigg] \,\m t
\\&~~~~
+\frac{1}{2} \mathbb{E}\Bigg[\sum_{k,l=1}^M
\chi_{t\leq T_k}\, \chi_{t\leq T_l}
\partial_k \partial_l \psi^t \bigg(N^{1/2}\left\langle  \mu_{t\wedge \f{T}}^N-\mathbb{E}[\mu_{t\wedge \f{T}}^N], \f{\phi}^t \right\rangle\bigg)
\\&~~~~~~~~~~~~~~~~~~~~~~~~~
\times
N^{-1} \sum_{n=1}^N \nabla \phi_k^t(\f{w}_n(t)) \cdot \nabla \phi_l^t(\f{w}_n(t))\Bigg]
\,\m t.
\end{align*}
Integrating in $t$, recalling that $\f{\phi}^{T}=\f{\varphi}$, and plugging in the equation \eqref{EvolutionPsi} satisfied by $\psi^t$, we obtain
\begin{align}
\nonumber
&
\mathbb{E}\Bigg[
\psi \bigg(N^{1/2}\left\langle  \mu_{\f{T}}^N-\mathbb{E}[\mu_{\f{T}}^N], \f{\varphi} \right\rangle\bigg)\Bigg]
\\&
\label{InductiveEquationDistributionEmpiricalMeasureAlmost}
=
\mathbb{E}\Bigg[
\psi^0 \bigg(N^{1/2}\left\langle  \mu_{0}^N-\mathbb{E}[\mu_{0}^N], \f{\phi}^0 \right\rangle\bigg)\Bigg]
\\&~~~~
\nonumber
+\frac{1}{2N^{1/2}} \int_0^T \sum_{k,l=1}^M
\chi_{t\leq T_k}\, \chi_{t\leq T_l}
\mathbb{E}\Bigg[
\partial_k \partial_l \psi^t \bigg(N^{1/2}\left\langle  \mu_{t\wedge \f{T}}^N-\mathbb{E}[\mu_{t\wedge \f{T}}^N], \f{\phi}^t \right\rangle\bigg)
\\&~~~~~~~~~~~~~~~~~~~~~~~~~~~~~~~~~~~~~~~~~~~~~~~~~~~
\nonumber
\times
N^{1/2} \left\langle  \mu_{t}^N-\mathbb{E}[\mu_{t}^N], \nabla\phi_k^t \cdot \nabla\phi_l^t \right\rangle \Bigg]
\,\m t.
\end{align}
\begin{subequations}
We then define $\tilde \psi^t_{kl}:\mathbb{R}^{M+1}\rightarrow \mathbb{R}$ as
\begin{align}
\label{DefTildePsi}
\tilde \psi^t_{kl}(s_1,\ldots,s_{M+1}) := \chi_{t\leq \min\{T_k,T_l\}} \partial_k \partial_l \psi^t(s_1,\ldots,s_M) s_{M+1}
\end{align}
and $\tilde{\f{\phi}}^t_{kl}:\domain \rightarrow \mathbb{R}^{M+1}$ as
\begin{align}
\label{DefTildePhi}
\tilde{\f{\phi}}^t_{kl}(\f{x}):=
\begin{pmatrix}
\phi_1^t(\f{x})
\\
\vdots
\\
\phi_M^t(\f{x})
\\
\nabla \phi_k^t(\f{x}) \cdot \nabla \phi_l^t(\f{x})
\end{pmatrix}.
\end{align}
Moreover, we set $\f{\tilde T}_{kl}:=(T_1,\ldots,T_M,\min\{T_k,T_l\})$. With these definitions, and in view of $\mu_0^N=\mathbb{E}[\mu_0^N]$ (which follows from Assumption \ref{ass:2}), equation \eqref{InductiveEquationDistributionEmpiricalMeasureAlmost} directly implies \eqref{InductiveEquationDistributionEmpiricalMeasure}.
\end{subequations}

Furthermore, the estimate \eqref{EstimateTildePhi} follows immediately from 
$$
\|\f{\phi}^t(\cdot)\|_{W^{q,\infty}}\leq \|\f{\phi}^T(\cdot)\|_{W^{q,\infty}}
$$ 
(which is a consequence of the maximum principle) and the definition of $\tilde{\f{\phi}}^t_{kl}$. Likewise, the estimate \eqref{EstimateTildePsiFirst} is immediate by the definition of $\tilde \psi^t_{kl}$, the estimate \eqref{EstimateNormPsi}, and the definition of the norms $\|\cdot\|_{\mathcal{L}_\text{pow,r}^{q}}$.
Finally, from \eqref{EstimateNormPsiDecay} and the definition of $\tilde \psi^t_{kl}$ we deduce \eqref{EstimateTildePsiSecond}.\\

{\bf Step 3: proof of \eqref{InductiveEquationDistributionDK}.}
Using It\^o's formula and \eqref{e:3}, we infer
\begin{align*}
&
\m \bigg(
\psi^t \Big(N^{1/2} \big(\rho_h-\mathbb{E}[\rho_h],\f{\phi}_h\big)_h (t\wedge \f{T}) \Big)\bigg)
\\&
=(\partial_t \psi^t) \Big(N^{1/2}\big(\rho_h-\mathbb{E}[\rho_h],\f{\phi}_h\big)_h(t\wedge \f{T})\Big) \,\m t
\\&~~~~
+\sum_{k=1}^M
\partial_k \psi^t \Big(N^{1/2}\big(\rho_h-\mathbb{E}[\rho_h],\f{\phi}_h\big)_h(t\wedge \f{T})\Big)
N^{1/2} \big(\rho_h-\mathbb{E}[\rho_h],\partial_t \phi_{h,k}\big)_h(t) \,\m t
\\&~~~~
+\frac{1}{2} \sum_{k=1}^M
\partial_k \psi^t \Big(N^{1/2}\big(\rho_h-\mathbb{E}[\rho_h],\f{\phi}_h\big)_h(t\wedge \f{T})\Big)
\chi_{t\leq T_k}\\
& ~~~~~~~~~~~~~~~
\times
N^{1/2} \big(\Delta_h \rho_h - \mathbb{E}[\Delta_h \rho_h], \phi_{h,k}\big)_h(t)
\,\m t
\\&~~~~
-\sum_{k=1}^M
\partial_k \psi^t \Big(N^{1/2}\big(\rho_h-\mathbb{E}[\rho_h],\f{\phi}_h\big)_h(t\wedge \f{T})\Big)
\\&~~~~~~~~~~~~~~~
\times
\chi_{t\leq T_k}
\sum_{(\f{y},\ell)\in(G_{h,d},\{1,\dots,d\})} (\mathcal{F}_\rho(t) \f{e}_{h,\f{y},\ell}^d,\nabla_h \phi_k^t \big)_h \,\m \beta_{(\f{y},\ell)}
\\&~~~~
+\frac{1}{2} \sum_{k,l=1}^M
\partial_k \partial_l \psi^t \Big(N^{1/2}\big(\rho_h-\mathbb{E}[\rho_h],\f{\phi}_h\big)_h(t\wedge \f{T})\Big)
\\&~~~~~~~~~~~~~~~
\times
\chi_{t\leq T_k\wedge T_l}\!\!\!\!
\sum_{(\f{y},\ell)\in(G_{h,d},\{1,\dots,d\})} (\mathcal{F}_\rho(t) \f{e}_{h,\f{y},\ell}^d,\nabla_h \phi_{h,k}^t \big)_h (\mathcal{F}_\rho(t) \f{e}_{h,\f{y},\ell}^d,\nabla_h \phi_{h,l}^t \big)_h \,\m t.
\end{align*}
Using the fact that $-\partial_t \phi_{h,k} = \chi_{t\leq T_k} \tfrac{1}{2}\Delta_h \phi_{h,k}$ and taking the expected value, we obtain
\begin{align*}
&
\m \mathbb{E}\bigg[
\psi^t \Big(N^{1/2} \big(\rho_h-\mathbb{E}[\rho_h],\f{\phi}_h\big)_h (t\wedge \f{T}) \Big)\bigg]
\\&
=\mathbb{E}\bigg[(\partial_t \psi^t) \Big(N^{1/2}\big(\rho_h-\mathbb{E}[\rho_h],\f{\phi}_h\big)_h(t\wedge \f{T})\Big)\bigg] \,\m t
\\&~~~~
+\frac{1}{2} \sum_{k,l=1}^M
\chi_{t\leq T_k\wedge T_l}
\mathbb{E}\bigg[
\partial_k \partial_l \psi^t \Big(N^{1/2}\big(\rho_h-\mathbb{E}[\rho_h],\f{\phi}_h\big)_h(t\wedge \f{T})\Big)
\\&~~~~~~~~~~~~~~~~~~~
\times
\sum_{(\f{y},\ell)\in(G_{h,d},\{1,\dots,d\})} (\mathcal{F}_\rho(t) \f{e}_{h,\f{y},\ell}^d,\nabla_h \phi_{h,k}^t \big)_h (\mathcal{F}_\rho(t) \f{e}_{h,\f{y},\ell}^d,\nabla_h \phi_{h,l}^t \big)_h\bigg] \,\m t.
\end{align*}
Using the cross-variation identity \eqref{e:7}, we get
\begin{align*}
&
\m\mathbb{E}\bigg[
\psi^t \Big(N^{1/2} \big(\rho_h-\mathbb{E}[\rho_h],\f{\phi}_h\big)_h (t\wedge \f{T}) \Big)\bigg]
\\&
=\mathbb{E}\bigg[(\partial_t \psi^t) \Big(N^{1/2}\big(\rho_h-\mathbb{E}[\rho_h],\f{\phi}_h\big)_h (t\wedge \f{T})\Big)\bigg] \,\m t
\\&~~~~
+\frac{1}{2} \sum_{k,l=1}^M
\chi_{t\leq T_k\wedge T_l}
\mathbb{E}\bigg[
\partial_k \partial_l \psi^t \Big(N^{1/2}\big(\rho_h-\mathbb{E}[\rho_h],\f{\phi}_h\big)_h(t\wedge \f{T})\Big)
\\
&~~~~~~~~~~~~~~~~
\times\big(\rho_h^+(t), \nabla_h \phi_{h,k}^t \cdot \nabla_h \phi_{h,l}^t \big)_h\bigg] \,\m t.
\end{align*}
Switching to integral notation, using \eqref{EvolutionPsi} as well as $\f{\phi}_h^T=\mathcal{I}_h\f{\varphi}$, and adding zero, we obtain
\begin{align*}
&
\mathbb{E}\bigg[
\psi \Big(N^{1/2} \big(\rho_h(\f{T})-\mathbb{E}[\rho_h(\f{T})],\mathcal{I}_h\f{\varphi}\big)_h \Big)\bigg]
\\&
=
\mathbb{E}\bigg[
\psi^0 \Big(N^{1/2} \big(\rho_h-\mathbb{E}[\rho_h],\f{\phi}_h\big)_h (0) \Big)\bigg]
\\&~~~~
-\frac{1}{2} \sum_{k,l=1}^M \int_0^{T_k\wedge T_l} \mathbb{E}\bigg[ \partial_k \partial_l\psi^t \Big(N^{1/2}\big(\rho_h-\mathbb{E}[\rho_h],\f{\phi}_h\big)_h (t\wedge \f{T})\Big)\bigg]
\left\langle \mathbb{E}[\mu^N_t], \nabla\phi_k^t \cdot \nabla \phi_l^t \right\rangle \,\m t
\\&~~~~
+\frac{1}{2} \sum_{k,l=1}^M
\int_0^{T_k\wedge T_l} \mathbb{E}\bigg[
\partial_k \partial_l \psi^t \Big(N^{1/2}\big(\rho_h-\mathbb{E}[\rho_h],\f{\phi}_h\big)_h (t\wedge \f{T})\Big)\\
&~~~~~~~~~~~~~~~~
\times\big(\rho_h(t), \nabla_h \phi_{h,k}^t \cdot \nabla_h \phi_{h,l}^t \big)_h\bigg] \,\m t
\\&~~~~
+\frac{1}{2} \sum_{k,l=1}^M
\int_0^{T_k\wedge T_l} \mathbb{E}\bigg[
\partial_k \partial_l \psi^t \Big(N^{1/2}\big(\rho_h-\mathbb{E}[\rho_h],\f{\phi}_h\big)_h (t\wedge \f{T})\Big)\\
&~~~~~~~~~~~~~~~~
\times
\big(\rho_h^-(t), \nabla_h \phi_{h,k}^t \cdot \nabla_h \phi_{h,k}^t \big)_h\bigg] \,\m t.
\end{align*}
Adding zero once more and using the fact that $\rho_h(0)=\mathbb{E}[\rho_h(0)]$ (which is a consequence of Assumption \ref{ass:2}), we arrive at
\begin{align*}
&
\mathbb{E}\bigg[
\psi \Big(N^{1/2} \big(\rho_h(\f{T})-\mathbb{E}[\rho_h(\f{T})],\mathcal{I}_h\f{\varphi}\big)_h \Big)\bigg]
\\&
=\psi^0(0)
+\frac{1}{2N^{1/2}} \sum_{k,l=1}^M
\int_0^{T_k\wedge T_l} \mathbb{E}\bigg[
\partial_k \partial_l \psi^t \Big(N^{1/2}\big(\rho_h-\mathbb{E}[\rho_h], \mathcal{I}_h\f{\phi}\big)_h(t\wedge \f{T}) \Big)
\\&~~~~~~~~~~~~~~~~~~~~~~~~~~~~~~~~~~~~~~~~~~~~
~~~~~
\times
N^{1/2} \big(\rho_h(t)-\mathbb{E}[\rho_h(t)], \mathcal{I}_h [\nabla \phi_{k}^t \cdot \nabla \phi_{l}^t] \big)_h\bigg] \,\m t
\\&~~~~
+\text{Err}_{num,1}+\text{Err}_{num,2}+\text{Err}_{neg},
\end{align*}
where we have set
\begin{align*}
\text{Err}_{neg}
&
:=
\frac{1}{2} \sum_{k,=1}^M
\int_0^{T_k\wedge T_l} \mathbb{E}\bigg[
\partial_k \partial_l \psi^t \Big(N^{1/2}\big(\rho_h-\mathbb{E}[\rho_h],\f{\phi}_h^t\big)_h (t\wedge \f{T}) \Big)\\
&~~~~~~~~~~~~~~~~
\times
\big(\rho_h^-(t), \nabla_h \phi_{h,k}^t \cdot \nabla_h \phi_{h,l}^t \big)_h\bigg] \,\m t,
\end{align*}
as well as
\begin{align*}
\text{Err}_{num,1}
:=&
\frac{1}{2} \sum_{k,l=1}^M
\int_0^{T_k\wedge T_l} \mathbb{E}\bigg[
\partial_k \partial_l \psi^t \Big(N^{1/2}\big(\rho_h-\mathbb{E}[\rho_h],\f{\phi}_h\big)_h (t\wedge \f{T}) \Big)\bigg]
\\&~~~~~~~~~~~~~~~~~~
\times \bigg(\big(\mathbb{E}[\rho_h], \mathcal{I}_h [\nabla \phi_{k}^t \cdot \nabla \phi_{l}^t] \big)_h-\left\langle \mathbb{E}[\mu^N_t], \nabla \phi_k^t \cdot \nabla \phi_l^t \right\rangle\bigg)
\,\m t
\\&
+\frac{1}{2} \sum_{k,l=1}^M
\int_0^{T_k\wedge T_l} \mathbb{E}\bigg[
\partial_k \partial_l \psi^t \Big(N^{1/2}\big(\rho_h-\mathbb{E}[\rho_h],\f{\phi}_h\big)_h (t\wedge \f{T}) \Big)
\\&~~~~~~~~~~~~~~~~~~
\times
\bigg(
\big(\rho_h(t), \nabla \phi_{h,k}^t \cdot \nabla \phi_{h,l}^t \big)_h
-
\big(\rho_h(t), \mathcal{I}_h[\nabla \phi_{k}^t \cdot \nabla \phi_{l}^t] \big)_h
\bigg)
\bigg]\,\m t,
\end{align*}
and
\begin{align*}
\text{Err}_{num,2}
:=&
\frac{1}{2N^{1/2}} \sum_{k,l=1}^M
\int_0^{T_k\wedge T_l} \mathbb{E}\Bigg[
\bigg(\partial_k \partial_l\psi^t \Big(N^{1/2}\big(\rho_h-\mathbb{E}[\rho_h],\f{\phi}_h\big)_h(t\wedge \f{T})\Big)
\\&~~~~~~~~~~~~~~~~~~~~~~~~~~~~~~~
-\partial_k \partial_l \psi^t \Big(N^{1/2}\big(\rho_h-\mathbb{E}[\rho_h],\mathcal{I}_h\f{\phi}\big)_h(t\wedge \f{T})\Big)
\bigg)
\\&~~~~~~~~~~~~~~~~~~~~~~~~~~~~~
\times
N^{1/2} \big(\rho_h(t)-\mathbb{E}[\rho_h(t)], \mathcal{I}_h [\nabla \phi_{k}^t \cdot \nabla \phi_{l}^t] \big)_h\Bigg] \,\m t.
\end{align*}
Using the definitions \eqref{DefTildePsi} and \eqref{DefTildePhi} and setting $\text{Err}_{num}:=\text{Err}_{num,1}+\text{Err}_{num,2}$, this yields the representation \eqref{InductiveEquationDistributionDK}.\\

{\bf Step 4: estimates for} $\text{Err}_{neg}$ {\bf and} $\text{Err}_{num,i}$.
We begin with $\text{Err}_{neg}$; it is easily seen to be bounded by
\begin{align*}
|\text{Err}_{neg}|
& \stackrel{\mathclap{\eqref{MomentBoundsFluctuationsRhoh}}}{\leq} C(\rho_{max},r,\|\f{\varphi}\|_{C^{1+\Theta}},T) \sum_{k,l=1}^M \int_0^{T_k\wedge T_l} \|\partial_k \partial_l \psi^t\|_{\mathcal{L}^0_{pow,r}}  \mean{\|\rho_h^-(t)\|_h}^{1/2} \,\m t\\
& \stackrel{\mathclap{\eqref{r:3-a},\,\ref{ass:3}}}{\leq} \,\,\,\quad C(\rho_{max},\rho_{min},d,r,\|\f{\varphi}\|_{C^{1+\Theta}},T) \mathcal{E}\!\left(N,h\right)\sum_{k,l=1}^M \int_0^{T_k\wedge T_l} \|\partial_k \partial_l \psi^t\|_{\mathcal{L}^r_{pow,0}}  \,\m t.
\end{align*}
\begin{subequations}
This entails \eqref{EstimateErrNeg}. Furthermore, the analogue of \eqref{EstimateNormPsiDecay} for the second derivative, and the time integrability of the singularity
$
\{\min_{m:T_m\geq t}(T_m-t)\}^{-1/2}
$
entail \eqref{EstimateErrNegSecond}.

We next note that $\mathbb{E}[\rho_h(t)]$ simply solves the discretised heat equation, while $\mathbb{E}[\mu_t^N]$ solves the exact heat equation. Using \eqref{eq:49}, \eqref{ErrorEstimateHeatEquation}, \eqref{r:2b-a}, and 
 \eqref{MomentBoundsFluctuationsRhoh}, we obtain
\begin{align}
\label{EstimateErrNum1}
& |\text{Err}_{num,1}|\nonumber
\\&~~~~
\stackrel{\mathclap{\eqref{eq:49}, \eqref{ErrorEstimateHeatEquation}, \eqref{MomentBoundsFluctuationsRhoh}}}{\leq}\quad\quad C(r,\rho_{max},T)\sum_{k,l=1}^M \int_0^{T_k\wedge T_l}
\bigg(1+T^{(r+1)/2}\|\f{\varphi}\|_{C^{1+\Theta}}^{r+1} \bigg)
\|\partial_k \partial_l \psi^t\|_{\mathcal{L}^0_{pow,r}} \nonumber
\\&~~~~~~~~~~
\times
\Big(
\|\f{\varphi}\|_{C^{p+2}}^2\|\rho_{h}(0)\|_{h}h^{p+1} +\mathbb{E}\big[\|\rho_h(t)\|_{L^2(\domain)}^{r+1}\big]^{1/(r+1)}
\|\f{\varphi}\|_{C^{p+2+\Theta}}^2 h^{p+1}
\Big)\m t\nonumber
\\&~~~~
\stackrel{\mathclap{\eqref{r:2b-a}}}{\leq} C(r,\rho_{max},\rho_{min},d,T) \sum_{k,l=1}^M \int_0^{T_k\wedge T_l}
\Big(
\|\f{\varphi}\|_{C^{p+2+\Theta}}^2 h^{p+1}
\Big)\nonumber
\\&~~~~~~~~~~
\times \bigg(1+T^{(r+1)/2}\|\f{\varphi}\|_{C^{1+\Theta}}^{r+1} \bigg)
\|\partial_k \partial_l \psi^t\|_{\mathcal{L}^0_{pow,r}}
\,\m t.
\end{align}
Finally, we deduce from \eqref{MomentBoundsFluctuationsRhoh}, \eqref{L2EstimateInterpolationGrad}, \eqref{r:2a-a} and \eqref{eq:325a}
\begin{align}
\label{EstimateErrNum2}
|\text{Err}_{num,2}|
&\leq \frac{C(r,\rho_{max},\rho_{min},d,T)}{N^{1/2}} \sum_{k,l=1}^M \|\f{\varphi}\|_{C^{p+1}} h^{p+1}
\bigg(1+ T^{(r+1)2}\|\f{\varphi}\|_{C^{2+\Theta}}^{r+1}\bigg)\nonumber
\\&~~~~~~~~~~~~~~~~\times
\int_0^{T_k\wedge T_l} \|\partial_k \partial_l D\psi^t\|_{\mathcal{L}^0_{pow,r}} \,\m t.
\end{align}
\end{subequations}
Combining \eqref{EstimateErrNum1} and \eqref{EstimateErrNum2} with \eqref{EstimateNormPsi} and \eqref{r:2b-a}, we infer \eqref{EstimateErrNum}. Using in addition \eqref{EstimateNormPsiDecay} and \eqref{EstimateNormPsiDecay2}, we deduce \eqref{EstimateErrNumSecond}.
The proof is complete.
\end{proof}

\begin{lemma}\label{lem:degenerateGreenFunction}
Let $0\leq T_1\leq T_2\leq \ldots\leq T_M\leq T$. Suppose that all $\varphi_m$ have vanishing average and are normalized in the sense $\|\varphi_m\|_{L^2(\domain)}=1$; suppose furthermore that whenever $T_m=T_{\tilde m}$, the corresponding $\varphi_m$ and $\varphi_{\tilde m}$ are orthogonal to each other in $L^2(\mathbb{T}^d)$.
Define $$
\overline{m}_{(1/2)T_1}:=\inf_{x\in \domain,t\geq \tfrac{1}{2}T_1} \mathbb{E}[\mu^N_t](x).
$$

Denoting the pseudo-inverse of the (possibly degenerate) nonnegative symmetric matrix $\Lambda_t$ defined in \eqref{CovMatrix} by $\Lambda_t^{-1}$, we have the estimate
\begin{align}
\label{LambdaBound}
|\Lambda_t^{-1}| \leq \frac{C(M)}{\overline{m}_{(1/2)T_1} \min\Big\{\min_{m:T_m\geq t}(T_m-t)~,~\min_{k,l:T_k\neq T_l}|T_k-T_l|\Big\} }.
\end{align}
\end{lemma}
\begin{proof}
To simplify notation, we define $T_0:=\frac{1}{2}T_1$.
Estimating the matrix in \eqref{CovMatrix}, writing $\varphi_k(\f{x}):=\sum_{\f{n}\in \mathbb{Z}^d} a_{k,\f{n}} \exp(-i \f{n}\cdot \f{x})$, and using the fact that $-\partial_t \phi_k^t = \frac{1}{2}\Delta \phi_k^t$, we get for any $\f{\alpha}\in \mathbb{R}^d$
\begin{align*}
&\frac{2}{\overline{m}_{(1/2)T_1}} \Lambda_t \f{\alpha} \cdot \f{\alpha}
\\
&\quad\geq \frac{1}{\overline{m}_{(1/2)T_1}} \sum_{m=1}^M \int_{T_{m-1}\vee t}^{T_m\vee t} \left\langle \mathbb{E}[\mu^N_t], \nabla \bigg(\sum_{k=m}^M \alpha_k \phi_{k}^t \bigg) \cdot \nabla \bigg(\sum_{l=m}^M \alpha_l \phi_{l}^t \bigg) \right\rangle \,\m t
\\&
\quad\geq
\sum_{m=1}^M \int_{T_{m-1}\vee t}^{T_m\vee t} \int_\domain \nabla \bigg(\sum_{k=m}^M \alpha_k \phi_{k}^t \bigg) \cdot \nabla \bigg(\sum_{l=m}^M \alpha_l \phi_{l}^t \bigg) \,\m \f{x} \,\m t
\\&
\quad=\sum_{m=1}^M \sum_{\f{n}\in \mathbb{Z}^d} \sum_{k,l=m}^M \int_{T_{m-1}\vee t}^{T_m\vee t}  e^{-\tfrac{1}{2}(T_k-t)|\f{n}|^2}e^{-\tfrac{1}{2}(T_l-t)|\f{n}|^2} \,\m t \\
&~~~~~~~~~~~~~~~~~~~~~~~~~
\quad\times 
|\f{n}|^2 a_{k,\f{n}} a_{l,\f{n}} \alpha_k \alpha_l
\\&
\quad=\sum_{m=1}^M \sum_{\f{n}\in \mathbb{Z}^d} \sum_{k,l=m}^M \bigg(e^{-\left(\tfrac{1}{2}(T_l+T_k)-T_m\vee t\right)|\f{n}|^2}-e^{-\left(\tfrac{1}{2}(T_l+T_k)-T_{m-1}\vee t\right)|\f{n}|^2}\bigg)
\\&~~~~~~~~~~~~~~~~~~~~~~~~~
\quad\times
a_{k,\f{n}} a_{l,\f{n}} \alpha_k \alpha_l
\\&
\quad=\sum_{m=1}^M \sum_{\f{n}\in \mathbb{Z}^d} \sum_{k,l=m}^M \bigg(1-e^{-(T_m\vee t-T_{m-1}\vee t)|\f{n}|^2}\bigg)   a_{k,\f{n}} a_{l,\f{n}} \alpha_k \alpha_l
\\&~~~~~~~~~~~~~~~~~~~~~~~~~~
\quad\times
e^{-\tfrac{1}{2}(T_l-T_m\vee t)|\f{n}|^2}e^{-\tfrac{1}{2}(T_k-T_m\vee t)|\f{n}|^2}
\\&
\quad\geq \frac{1}{2}\sum_{m=1}^M \sum_{\f{n}\in \mathbb{Z}^d\setminus \{0\}} \sum_{k,l=m}^M \Big((T_m\vee t-T_{m-1}\vee t)\wedge 1\Big) a_{k,\f{n}} a_{l,\f{n}} \alpha_k \alpha_l
\\&~~~~~~~~~~~~~~~~~~~~~~~~~~
\quad\times
e^{-\tfrac{1}{2}(T_l-T_m\vee t)|\f{n}|^2}e^{-\tfrac{1}{2}(T_k-T_m\vee t)|\f{n}|^2}
\\&
\quad= \frac{1}{2}\sum_{m=1}^M \Big((T_m\vee t-T_{m-1}\vee t)\wedge 1\Big)
\\&~~~
\quad\times
\int_\domain \bigg(\sum_{k=m}^M \alpha_k \Big(\phi_{k}^{T_m\vee t}-\dashint_\domain \phi_k^{T_m\vee t} \,\m \tilde{\f{x}}\Big) \bigg) \bigg(\sum_{l=m}^M \alpha_l \Big(\phi_{l}^{T_m\vee t}-\dashint_\domain \phi_l^{T_m\vee t} \,\m \tilde{\f{x}}\Big) \bigg) \,\m \f{x}
.
\end{align*}
Using the fact that $\phi_k^{T_k}=\varphi_k$, that $\|\phi_k^{t}\|_{L^2(\domain)}\leq 1$ for all $t$, that the $\varphi_k$ have vanishing average, and our assumption on the orthogonality of the $\varphi_k$ with the same $T_k$, we deduce
\begin{align*}
\frac{2}{\overline{m}_{(1/2)T_1}} \Lambda_t \f{\alpha} \cdot \f{\alpha}
\geq c(M) \min\Big\{\min_{m:T_m\geq t}(T_m-t)~,~\min_{k,l:T_k\neq T_l}|T_k-T_l|\Big\} \sum_{1\leq m\leq M:T_m\geq t} |\alpha_m|^2.
\end{align*}
Note that $(\Lambda_t)_{kl}=0$ whenever $T_k<t$ or $T_l<t$. This concludes our proof.
\end{proof}

\subsection{Proof of Theorem \ref{main1}}\label{proof1}\label{ss:p1}

For finite difference discretization schemes, Theorem~\ref{main1} is an easy consequence of Proposition~\ref{prop:recursive}.
\begin{proof}[Proof of Theorem~\ref{main1} in the finite difference case]
Taking the difference of \eqref{InductiveEquationDistributionDK} and \eqref{InductiveEquationDistributionEmpiricalMeasure} and using \eqref{EstimateErrNumSecond} and \eqref{EstimateErrNegSecond}, we see that Proposition~\ref{prop:recursive} implies
\begin{align}
\label{d:1000}
&\Bigg|
\mathbb{E} \Bigg[\psi \bigg(N^{1/2}\big(\rho_h(\f{T})-\mathbb{E}[\rho_h(\f{T})],\mathcal{I}_h\f{\varphi}\big)_h \bigg)\Bigg]
-\mathbb{E} \Bigg[\psi \bigg(N^{1/2}\left\langle \mu_{\f{T}}^N-\mathbb{E}[\mu_{\f{T}}^N], \f{\varphi} \right\rangle\bigg)\Bigg]
\Bigg|
\\&\nonumber
\label{FirstIterationStep}
\leq
\frac{1}{2N^{1/2}}  \sum_{k,l=1}^M \int_0^{T_k\wedge T_l}
\Bigg|
\mathbb{E} \Bigg[\tilde \psi^t_{kl} \bigg(N^{1/2}\big( \rho_h-\mathbb{E}[\rho_h], \mathcal{I}_h\tilde{\f{\phi}}_{kl}\big)_h (t\wedge \f{\tilde T}_{kl})\bigg) \Bigg]
\\&~~~~~~~~~~~~~~~~~~~~~~~~~~~~~~~~~~~
\nonumber
-\mean{\tilde \psi^t_{kl} \bigg(N^{1/2} \left\langle  \mu_{t\wedge \f{\tilde T}}^N-\mathbb{E}[\mu_{t\wedge \f{\tilde T}}^N],\tilde{\f{\phi}}^t_{kl} \right\rangle \bigg) }\Bigg|
\,\m t
\\&~~~~~~~~~~
\nonumber
+C(M,\rho_{max},\rho_{min},d,r,\|\f{\varphi}\|_{C^{p+2+\Theta}}) \big(\|\psi\|_{{{\mathcal{L}}}^{1}_{pow,1}} + N^{-1/2} \|D\psi\|_{{{\mathcal{L}}}^{1}_{pow,1}}\big) \\
& ~~~~~~~~~~~~~~~~~~
\times\frac{1}{\sqrt{  \rho_{min} \min_{k,l:T_k\neq T_l}|T_k-T_l| }}
h^{p+1}
\nonumber
\\&~~~~~~~~~~
\nonumber
+C(M,\rho_{max},\rho_{min},d,r,\|\f{\varphi}\|_{C^{1+\Theta}},T) \|\psi\|_{{{\mathcal{L}}}^{1}_{pow,1}} \mathcal{E}\!\left(N,h\right)
\\
& ~~~~~~~~~~~~~~~~~~
\times\frac{1}{\sqrt{  \rho_{min} \min_{k,l:T_k\neq T_l}|T_k-T_l| }}.
\end{align}
The inequality \eqref{EstimateNormPsiDecay2} implies 
\begin{align}\label{g:1}
\int_0^T \|{\tilde \psi}_{kl}^t\|_{{\tilde {\mathcal{L}}}_{pow,r+1}^{2j-2}} \,\m t
&\leq C(j,M,\|\f{\varphi}\|_{W^{1,\infty}}^2,\overline{m}_{(1/2)T_1},\f{T}) \|{\psi}\|_{{{\mathcal{L}}}_{pow,r}^{2j-1}}.
\end{align}
In case $j=1$, \eqref{FirstIterationStep} entails the desired bound by the estimate on ${\tilde \psi}_{kl}^t$ upon replacing $\psi$ in \eqref{d:1000} by its convolution with a mollifier on the scale $N^{-1/2}$, which we denote by $\eta_{N^{-1/2}}$. This is a straightforward result of the convolutional inequalities 
\begin{align*}
\|D(\eta_{N^{-1/2}}\ast \psi)\|_{{{\mathcal{L}}}^{1}_{pow,1}} & \leq C N^{1/2} \|\psi\|_{{{\mathcal{L}}}^{1}_{pow,1}},\\
|\eta_{N^{-1/2}}\ast \psi-\psi| & \leq CN^{-1/2}\|\psi\|_{{{\mathcal{L}}}^{1}_{pow,0}}.
\end{align*}
For $j>1$, taking the difference of \eqref{InductiveEquationDistributionDK} and \eqref{InductiveEquationDistributionEmpiricalMeasure}, using the bounds \eqref{EstimateTildePhi}, \eqref{EstimateErrNum}, \eqref{EstimateErrNeg}, and iterating this estimate $j-1$ times (i.\,e.\ using in each step again \eqref{InductiveEquationDistributionDK} and \eqref{InductiveEquationDistributionEmpiricalMeasure} to estimate the terms of the form 
\begin{align*}
\mathbb{E} \Bigg[\tilde \psi^t_{kl} \bigg(N^{1/2}\big( \rho_h-\mathbb{E}[\rho_h], \mathcal{I}_h\tilde{\f{\phi}}_{kl}\big)_h (t\wedge \f{\tilde T}_{kl})\bigg) \Bigg]
-\Bigg[\tilde \psi^t_{kl} \bigg(N^{1/2} \left\langle  \mu_{t\wedge \f{\tilde T}}^N-\mathbb{E}[\mu_{t\wedge \f{\tilde T}}^N],\tilde{\f{\phi}}^t_{kl} \right\rangle \bigg) \Bigg],
\end{align*}
only bounding the terms in $\|{\tilde \psi}_{kl}^t\|_{{\tilde {\mathcal{L}}}_{pow,r+1}^{2j-2}}$ using \eqref{g:1} in the last step), we deduce 
\begin{align}\label{b:1000}
&\mathbb{E} \Bigg[\tilde \psi^t_{kl} \bigg(N^{1/2}\big( \rho_h-\mathbb{E}[\rho_h], \mathcal{I}_h\tilde{\f{\phi}}_{kl}\big)_h (t\wedge \f{\tilde T}_{kl})\bigg) \Bigg]
-\Bigg[\tilde \psi^t_{kl} \bigg(N^{1/2} \left\langle  \mu_{t\wedge \f{\tilde T}}^N-\mathbb{E}[\mu_{t\wedge \f{\tilde T}}^N],\tilde{\f{\phi}}^t_{kl} \right\rangle \bigg) \Bigg]\nonumber
\\&
\leq \nonumber
\sum_{\tilde j=1}^{j-1} \left\{N^{-(\tilde j-1)/2} \bigg(C(M,\rho_{max},\rho_{min},d,j,\|\f{\varphi}\|_{C^{p+2+\Theta+\tilde{j}-1}}) h^{p+1} \right.\\\nonumber
& ~~~~~~~~~~~~~~~~~~~~~~\quad \quad + C(M,\rho_{max},\rho_{min},d,j,\|\f{\varphi}\|_{C^{1+\Theta+\tilde{j}-1}},T) \mathcal{E}\!\left(N,h\right) \bigg) \|\psi\|_{{{\mathcal{L}}}^{2j-1}_{pow,\tilde j}}
\\&~~~~~~~~~~~~~~~~\nonumber
+C(M,\rho_{max},\rho_{min},d,j,\|\f{\varphi}\|_{C^{p+2+\Theta+\tilde{j}-1}}) N^{-\tilde j/2} h^{p+1}
\|D\psi\|_{{\mathcal{L}}^{2j-1}_{pow,\tilde j}}
\bigg\}
\\&~~~
+C(M,\rho_{max},\rho_{min},d,j,\|\f{\varphi}\|_{W^{j-1,\infty}}^2 T) \|\psi\|_{{ {\mathcal{L}}}^{2j-1}_{pow,1}} N^{-j/2}.
\end{align}
We use estimate \eqref{b:1000} in \eqref{FirstIterationStep} to bound the terms of the form
\begin{align*}
&\int_0^{T_k\wedge T_l} \Bigg|\mathbb{E} \Bigg[\tilde \psi^t_{kl} \bigg(N^{1/2}\big( \rho_h-\mathbb{E}[\rho_h], \mathcal{I}_h\tilde{\f{\phi}}_{kl}\big)_h (t\wedge \f{\tilde T}_{kl})\bigg) \Bigg]\\
&~~~~~~~~~~~~~~~~~~~~
-\Bigg[\tilde \psi^t_{kl} \bigg(N^{1/2} \left\langle  \mu_{t\wedge \f{\tilde T}}^N-\mathbb{E}[\mu_{t\wedge \f{\tilde T}}^N],\tilde{\f{\phi}}^t_{kl} \right\rangle \bigg) \Bigg]\Bigg| \,\m t,
\end{align*}
Therefore, estimate \eqref{b:1000} turns into 
\begin{align*}
&\Bigg|
\mathbb{E} \Bigg[\psi \bigg(N^{1/2}\big((\rho_h-\mathbb{E}[\rho_h]),\mathcal{I}_h\varphi\big)_h(\f{T}) \bigg)\Bigg]
-\mathbb{E} \Bigg[\psi \bigg(N^{1/2}\left\langle \mu_{\f{T}}^N-\mathbb{E}[\mu_{\f{T}}^N], \f{\varphi} \right\rangle\bigg)\Bigg]
\Bigg|
\\&
\leq C(M,\rho_{max},\rho_{min},d,j,\f{T},\|\f{\varphi}\|_{C^{p+2+\Theta+j-1}},\overline{m}_{(1/2)T_1}) h^{p+1}
\|\psi\|_{{{\mathcal{L}}}^{2j-1}_{pow,0}}
\\&~~~
+C(M,\rho_{max},\rho_{min},d,j,\f{T},\|\f{\varphi}\|_{C^{p+2+\Theta+j-1}},\overline{m}_{(1/2)T_1}) N^{-j/2} h^{p+1}
\|D\psi\|_{{{\mathcal{L}}}^{2j-1}_{pow,0}}
\\&~~~
+ C(M,\rho_{max},\rho_{min},d,j,\f{T},\|\f{\varphi}\|_{C^{1+\Theta+j-1}},\overline{m}_{(1/2)T_1}) \mathcal{E}\!\left(N,h\right)
\|\psi\|_{{{\mathcal{L}}}^{2j-1}_{pow,0}} 
\\&~~~
+C(M,\rho_{max},\rho_{min},d,j,\f{T},\|\f{\varphi}\|_{W^{j,\infty}}^2,\overline{m}_{(1/2)T_1}) \|\psi\|_{{{\mathcal{L}}}^{2j-1}_{pow,0}} N^{-j/2}.
\end{align*}
Finally, we replace $\psi$ by $\eta_{N^{-j/2}} \ast \psi$ in \eqref{d:1000} (note that we have $|\psi-\eta_{N^{-j/2}} \ast \psi| \leq C\|D\psi\|_{L^\infty} N^{-j/2}$ and $\|D(\eta_{N^{-j/2}}\ast \psi)\|_{\mathcal{L}_{pow,r}^m}\leq C N^{j/2} \|\psi\|_{\mathcal{L}_{pow,r}^m}$).
This, together with the fact that $\overline{m}_{(1/2)T_1}$ is controlled by $\rho_{min}$, proves Theorem~\ref{main1} in the case of finite difference discretisations.
\end{proof}

\subsection{Recursive step for Theorem \ref{main2}}\label{ss:rec}

In Theorem \ref{main1}, one is forced to distinguish between the different final times $T_1,\dots,T_M$ due to the singular nature of the evolution equation for $\psi$ \eqref{EvolutionPsi}. In contrast, $\psi$ is static in Theorem \ref{main2}: therefore, its proof can be detailed in the (notationally much more convenient) case of equal final times $T_1=\dots=T_m=T$ without losing in generality.

We first introduce some handy notation. For $t\leq T$, we abbreviate 
\begin{align*}
\mathcal{T}_N(\varphi,T,t) & :=\langle \mu^{N}_t,\phi^t\rangle - \langle \mu^{N}_0,\phi^0\rangle = \frac{1}{N}\sum_{k=1}^{N}{\phi^t(\f{w}_k(t))}-\frac{1}{N}\sum_{k=1}^{N}{\phi^0(\f{w}_k(0))}
\end{align*}
and
\begin{align*}
\mathcal{S}_N(\mathcal{I}_h\varphi,T,t) & := (\rho_h(t),\phi_{h}^t)_h-(\rho_h(0),\phi_{h}^0)_h,
\end{align*}
where $\phi^t$ (respectively, $\phi^t_h$) solves the backwards heat equation \eqref{eq:28} (respectively, the backwards discrete heat equation \eqref{eq:29}) with datum $\varphi$ (respectively, $\mathcal{I}_h\varphi$) at time $T$. 
Given a multi-index $\boldsymbol{j}=(j_1,\dots,j_M)$ such that $|\boldsymbol{j}|_1=j\in\mathbb{N}$ and a set of smooth test functions $\boldsymbol{\varphi}=(\varphi_1,\dots,\varphi_M)$, we abbreviate 
\begin{align}\label{e:25}
& \mathcal{S}^{\boldsymbol{j}}_N(\mathcal{I}_h\boldsymbol{\varphi},T,t):=\prod_{m=1}^{M}{\mathcal{S}^{j_m}_N(\mathcal{I}_h\varphi_{m},T,t)},\qquad \mathcal{T}^{\boldsymbol{j}}_N(\boldsymbol{\varphi},T,t):=\prod_{m=1}^{M}{\mathcal{T}^{j_m}_N(\varphi_{m},T,t)}
\end{align}
and we set  
\begin{align}\label{eq:402}
\mathcal{D}(\boldsymbol{j},\boldsymbol{\varphi},T):=\left|\mean{\mathcal{S}^{\boldsymbol{j}}_N(\mathcal{I}_h\boldsymbol{\varphi},T,T)}-\mean{\mathcal{T}^{\boldsymbol{j}}_N(\boldsymbol{\varphi},T,T)}\right|.
\end{align}

In order to show Theorem \ref{main2}, we first provide a series of preliminary results.

\begin{lemma}[First moments]\label{lem:12}
The first moments of the Dean--Kawasaki model in Definition \ref{def:1} agree with those of the Brownian particle system. Namely, for $\varphi\in C^1$, we have $\mean{\mathcal{S}_N(\mathcal{I}_h\varphi,T,T)}=\mean{\mathcal{T}_N(\varphi,T,T)}=0$, where $\mathcal{S}_N$ and $\mathcal{T}_N$ have been defined in \eqref{e:25}.
\end{lemma}
\begin{proof}
This follows promptly from Lemma \ref{lem:10}, as neither $\mathcal{S}_N(\mathcal{I}_h\varphi,T,t)$ nor $\mathcal{T}_N(\varphi,T,t)$ admits drift.
\end{proof}

\begin{lemma}[Second moments]\label{thm:1}
Let $\Theta$ be as in \eqref{b:7}.
 Assume the validity of Assumptions  \ref{ass:1}, \ref{ass:2}, \ref{ass:4}, \ref{ass:3}. Fix $\varphi_1,\varphi_2\in C^{3+p+\Theta}$. Let $\rho_h$ be as given in Definition \ref{def:1}. Then 
\begin{align}
& \left|\mean{\mathcal{S}_N(\mathcal{I}_h\varphi_1,T,T)\mathcal{S}_N(\mathcal{I}_h\varphi_2,T,T)}-\mean{\mathcal{T}_N(\varphi_1,T,T)\mathcal{T}_N(\varphi_2,T,T)}\right| \nonumber\\
& \quad \leq N^{-1}T\|\varphi_1\|_{C^{1+\Theta}}\|\varphi_{2}\|_{C^{1+\Theta}}\,\mathcal{E}\!\left(N,h\right)\nonumber\\
& \quad \quad + h^{p+1}N^{-1}\max\{T^{1/2};T\}C(d,\rho_{max},\rho_{min})\|\varphi_1\|_{C^{p+3+\Theta}}\|\varphi_2\|_{C^{p+3+\Theta}},
\end{align}
where $\mathcal{S}_N$ and $\mathcal{T}_N$ have been defined in \eqref{e:25}, and $\mathcal{E}\!\left(N,h\right)$ has been  introduced in \eqref{b:8}.
\end{lemma}
\begin{proof}
Set $r_h^t:=\nabla_h\phi_{1,h}^t\cdot\nabla_h\phi_{2,h}^t-\mathcal{I}_h\left\{\nabla\phi_{1}^t\cdot\nabla\phi_{2}^t\right\}$. The It\^o differential formula for $\mathcal{S}_N(\varphi_1,T,t)\mathcal{S}_N(\varphi_2,T,t)$ stated in Lemma \ref{lem:10} gives
\begin{align*}
& \m \mean{\mathcal{S}_N(\mathcal{I}_h\varphi_1,T,t)\mathcal{S}_N(\mathcal{I}_h\varphi_2,T,t)} \\
& \quad = N^{-1}\mean{(\rho^{+}_h(t),\nabla_h\phi_{1,h}^t\cdot\nabla_h\phi^t_{2,h})_h}\m t\\
 & \quad =  N^{-1}\mean{(\rho_h(t),\nabla_h\phi_{1,h}^t\cdot\nabla_h\phi_{2,h}^t)_h}\m t+N^{-1}\mean{(\rho^-_h(t),\nabla_h\phi_{1,h}^t\cdot\nabla_h\phi_{2,h}^t)_h}\m t\\
 & \quad = N^{-1}\mean{\left(\rho_h(t),\mathcal{I}_h\left\{\nabla\phi_{1}^t\cdot\nabla\phi_{2}^t\right\}\right)_h}\m t + N^{-1}\mean{(\rho^-_h(t),\nabla_h\phi_{1,h}^t\cdot\nabla_h\phi_{2,h}^t)_h}\m t \\
 & \quad \quad + N^{-1}\mean{(\rho_h(t),r_{h}^t)_h}\m t\\
 & \quad  = N^{-1}\mean{\left\{\left(\rho_h(t),\mathcal{I}_h\left\{\nabla\phi_{1}^t\cdot\nabla\phi_{2}^t\right\}\right)_h-\left(\rho_h(0),\mathcal{P}^t_h\left(\mathcal{I}_h\left\{\nabla\phi_{1}^t\cdot\nabla\phi_{2}^t\right\}\right)\right)_h\right\}}\m t \\
 &\quad \quad + N^{-1}\mean{\left(\rho_h(0),\mathcal{P}^t_h\left(\mathcal{I}_h\left\{\nabla\phi_{1}^t\cdot\nabla\phi_{2}^t\right\}\right)\right)_h}\m t \\
 & \quad \quad +N^{-1}\mean{(\rho^-_h(t),\nabla_h\phi_{1,h}^t\cdot\nabla_h\phi_{2,h}^t)_h}\m t + N^{-1}\mean{(\rho_h(t),r_{h}^t)_h}\m t =: \sum_{i=1}^{4}{A_i}\,\m t,
 \end{align*}
where $\mathcal{P}_h^{\cdot}$ is the solution operator for the discrete backwards heat equation, see Subsection \ref{ss:not}. On the other hand Lemma \ref{lem:10} also implies
\begin{align*}
& \m\mean{\mathcal{T}_N(\varphi_1,T,t)\mathcal{T}_N(\varphi_2,T,t)} \\
& \quad = N^{-1}\mean{\frac{1}{N}\sum_{k=1}^{N}{\nabla\phi_1^t(\f{w}_k(t))\cdot\nabla\phi_2^t(\f{w}_k(t))}}\m t\\
& \quad = N^{-1}\mean{\left(\frac{1}{N}\sum_{k=1}^{N}{\nabla\phi_1^t(\f{w}_k(t))\cdot\nabla\phi_2^t(\f{w}_k(t))}-\frac{1}{N}\sum_{k=1}^{N}{\mathcal{P}^t(\nabla\phi_1^t\cdot\nabla\phi_2^t)(\f{w}_k(0))}\right)}\m t\\
& \quad \quad + N^{-1}\mean{\frac{1}{N}\sum_{k=1}^{N}{\mathcal{P}^t(\nabla\phi_1^t\cdot\nabla\phi_2^t)(\f{w}_k(0))}}\m t =: \sum_{i=1}^{2}{B_i\m t},
\end{align*}
where $\mathcal{P}^{\cdot}$ is the solution operator for the backwards heat equation, see Subsection \ref{ss:not}.
We get $A_1=B_1$ since the first (centred) moments agree (see Lemma \ref{lem:12}). Furthermore, \eqref{r:3-a} and \eqref{eq:202} grant
\begin{align*}
|A_3|\m t & \leq N^{-1}\|\nabla_h\phi_{1,h}^t\|_{\infty}\|\nabla_h\phi_{2,h}^t\|_{\infty}\mathcal{E}\!\left(N,h\right)\m t\\
&  \leq N^{-1}\|\nabla_h\phi_{1,h}^t\|_{\infty}\|\nabla_h\phi_{2,h}^t\|_{\infty}\mathcal{E}\!\left(N,h\right)\m t \leq N^{-1}\|\varphi_1\|_{C^{1+\Theta}}\|\varphi_{2}\|_{C^{1+\Theta}}\mathcal{E}\!\left(N,h\right)\m t.
\end{align*}
The bounds \eqref{eq:49} and \eqref{eq:202} promptly give 
\begin{align*}
\int_{0}^{T}{|A_4|\m t} & \leq N^{-1}\sup_{t\in[0,T]}{\|r_h^t\|_h} \int_{0}^{T}{\mean{\|\rho_h(t)\|_{h}}\m t}\\
& \leq CN^{-1}\sup_{t\in[0,T]}{\|r^t_h\|_h}\, T^{1/2}\left(\int_{0}^{T}{\mean{\|\rho_h(t)\|_h^2}\m t}\right)^{1/2}\\
& \stackrel{\mathclap{\eqref{eq:49}}}{\leq} Ch^{p+1}N^{-1}\|\varphi_1\|_{C^{p+2+\Theta}}\|\varphi_2\|_{C^{p+2+\Theta}}T^{1/2}\left(\int_{0}^{T}{\mean{\|\rho_h(t)\|_h^2}\m t}\right)^{1/2}\\
& \stackrel{\mathclap{\eqref{eq:202}\eqref{r:2b-a}}}{\leq} \,\,\,\quad h^{p+1}N^{-1}\max\{T^{1/2};T\}C(d,\rho_{max},\rho_{min})\|\varphi_1\|_{C^{p+2+\Theta}}\|\varphi_2\|_{C^{p+2+\Theta}}.
\end{align*}
We decompose $A_2-B_2$ as follows
\begin{align}\label{eq:808}
A_2-B_2& = N^{-1}\mean{\left(\rho_h(0),\mathcal{P}^t_h\left(\mathcal{I}_h\left\{\nabla\phi_{1}^t\cdot\nabla\phi_{2}^t\right\}\right)\right)_h}\m t \nonumber\\
& \quad -N^{-1}\mean{\frac{1}{N}\sum_{k=1}^{N}{\mathcal{P}^t(\nabla\phi_1^t\cdot\nabla\phi_2^t)(\f{w}_k(0))}}\m t \nonumber\\
& = -N^{-1}\left\{\mean{\frac{1}{N}\sum_{k=1}^{N}{\mathcal{P}^t(\nabla\phi_1^t\cdot\nabla\phi_2^t)(\f{w}_k(0))}}\m t\right.\nonumber\\
& \quad \quad \left. - \left(\rho_h(0),\mathcal{I}_h[\mathcal{P}^t(\nabla\phi_1^t\cdot\nabla\phi_2^t)]\right)_h \m t\right\}\nonumber\\
& \quad + N^{-1}\left(\rho_h(0),\mathcal{P}^t_h\left(\mathcal{I}_h\left\{\nabla\phi_{1}^t\cdot\nabla\phi_{2}^t\right\}\right)-\mathcal{I}_h[\mathcal{P}^t(\nabla\phi_1^t\cdot\nabla\phi_2^t)]\right)_h\m t \nonumber\\
& =:C_1+C_2,
\end{align}
where we have also used that $\rho_h(0)$ is deterministic.
The term $C_1$ is bounded using \eqref{eq:401a} applied to the function $\eta:=\mathcal{P}^t(\nabla\phi_1^t\cdot\nabla\phi_2^t)$, while the term $C_2$ is dealt with using \eqref{eq:325a} with choice $\varphi:=\nabla\phi_1^t\cdot\nabla\phi_2^t$. All together, we obtain the bound
\begin{align}\label{eq:51}
|A_2-B_2| 
& \leq Ch^{p+1}N^{-1}(C+\rho_{max})\|\varphi_1\|_{C^{p+3}}\|\varphi_2\|_{C^{p+3}}\m t.
\end{align}
The proof is complete.
\end{proof}

\begin{prop}[Recursive formula for higher moments]\label{thm:2}
Let $\Theta$ be as in \eqref{b:7}. 
Fix $\boldsymbol{\varphi}=(\varphi_1,\dots,\varphi_K)\in\left[C^{p+3+\Theta}\right]^M$, a vector $\boldsymbol{j}=(j_1,\dots,j_M)$ such that $|\boldsymbol{j}|_1=j$. For each pair $i,j\in\{1,\dots,M\}$, let $\boldsymbol{j}^{ij}$ be as defined in Lemma \ref{lem:10}.
Let $\mathcal{E}\!\left(N,h\right)$ be as defined in \eqref{b:8}. Assume the validity of Assumptions  \ref{ass:1}, \ref{ass:2}, \ref{ass:4}, \ref{ass:3}. 
We recall the abbreviation for the difference of moments (see \eqref{e:25}, \eqref{eq:402})
\begin{align*}
\mathcal{D}(\boldsymbol{j},\boldsymbol{\varphi},T) & := \left|\mean{\mathcal{S}^{\boldsymbol{j}}_N(\mathcal{I}_h\boldsymbol{\varphi},T,T)}-\mean{\mathcal{T}^{\boldsymbol{j}}_N(\boldsymbol{\varphi},T,T)}\right|\\
& = \left|\mean{\prod_{m=1}^{M}{\left\{(\rho_h(T),\phi_{m,h}^T)_h-(\rho_h(0),\phi_{m,h}^0)_h\right\}^{j_m}}}\right.\\
& \quad \left.-\mean{\prod_{m=1}^{M}{\left\{\frac{1}{N}\sum_{k=1}^{N}{\phi_{m}^T(\f{w}_k(T))}-\frac{1}{N}\sum_{k=1}^{N}{\phi_{m}^0(\f{w}_k(0))}\right\}^{j_m}}}\right|.
\end{align*}
Then we have the recursive formula
\begin{align}\label{eq:54}
\mathcal{D}(\boldsymbol{j},\boldsymbol{\varphi},T) & \leq N^{-1}\sum_{k,l=1}^{M}{\frac{(j_k-\delta_{kl})j_l}{2}}\int_{0}^{T}{\mathcal{D}(\{\boldsymbol{j}^{kl};1\},\{\boldsymbol{\phi}^t;\nabla\phi_k^t\cdot \nabla\phi_l^t\},t)\m t} \nonumber\\
& \quad\quad + N^{-1}\rho_{max}\sum_{k,l=1}^{M}{\frac{(j_k-\delta_{kl})j_l}{2}}\int_{0}^{T}{\mathcal{D}(\boldsymbol{j}^{kl},\boldsymbol{\phi}^t,t)\|\varphi_k\|_{C^{1+\Theta}}\|\varphi_l\|_{C^{1+\Theta}}\m t}\nonumber\\
& \quad\quad + \left\{CN^{-1} TC(d,\rho_{max},\rho_{min}) \right\}^{j/2}(2j)^{3(j-2)}\mathcal{E}\!\left(N,h\right)\nonumber\\
& \quad \quad \quad \quad\quad \times\left(\sum_{k,l=1}^{M}{\frac{(j_k-\delta_{kl})j_l}{2}}\right)\left(\prod_{m=1}^{M}{\|\varphi_m\|_{C^{1+\Theta}}^{j_m}}\right)\nonumber\\
& \quad\quad + h^{p+1}\left\{CN^{-1} \max\{T^{1/2};T\}C(d,\rho_{max},\rho_{min}) \right\}^{j/2}(2j)^{3(j-2)}\nonumber\\
& \quad\quad\quad\quad\quad \times \left(\sum_{k,l=1}^{M}{\frac{(j_k-\delta_{kl})j_l}{2}}\right)\left(\prod_{m=1}^{M}{\|\varphi_m\|_{C^{p+3+\Theta}}^{j_m}}\right)\nonumber\\
& \quad =: A^{j-1}_{recursion}+A^{j-2}_{recursion} +\mbox{Err}_{neg} + \mbox{Err}_{num}.
\end{align}
\end{prop}
\begin{proof}
We use Lemma \ref{lem:10} to deduce
\begin{align*} 
& \m\mean{\mathcal{S}^{\boldsymbol{j}}_N(\mathcal{I}_h\boldsymbol{\varphi},T,t)} \\
& \quad = N^{-1}\mean{\sum_{k,l=1}^{M}{\frac{(j_k-\delta_{kl})j_l}{2}\mathcal{S}^{\boldsymbol{j}^{kl}}_N(\mathcal{I}_h\boldsymbol{\varphi},T,t)\left(\rho^{+}_h(t),\nabla_h\phi_{k,h}^t\cdot\nabla_h\phi_{l,h}^t\right)_h}}\m t\\
& \quad = N^{-1}\mean{\sum_{k,l=1}^{M}{\frac{(j_k-\delta_{kl})j_l}{2}\mathcal{S}^{\boldsymbol{j}^{kl}}_N(\mathcal{I}_h\boldsymbol{\varphi},T,t)\left(\rho_h(t),\nabla_h\phi_{k,h}^t\cdot\nabla_h\phi_{l,h}^t\right)_h}}\m t\\
& \quad\quad + N^{-1}\mean{\sum_{k,l=1}^{M}{\frac{(j_k-\delta_{kl})j_l}{2}\mathcal{S}^{\boldsymbol{j}^{kl}}_N(\mathcal{I}_h\boldsymbol{\varphi},T,t)\left(\rho^{-}_h(t),\nabla_h\phi_{k,h}^t\cdot\nabla_h\phi_{l,h}^t\right)_h}}\m t.
\end{align*}
In analogy to the notation of Lemma \ref{thm:1}, we define 
$$
r_{k,l,h}^t:=\nabla_h\phi_{k,h}^t\cdot\nabla_h\phi_{l,h}^t-\mathcal{I}_h\left\{\nabla\phi_{k}^t\cdot\nabla\phi_{l}^t\right\}.
$$
Let $\mathcal{P}^{\cdot}$ (respectively, $\mathcal{P}_h^{\cdot}$) be the solution operator for the backwards heat equation (respectively, for the discrete backwards heat equation), see Subsection \ref{ss:not}. 
We then proceed above as
\begin{align*}
& \m\mean{\mathcal{S}^{\boldsymbol{j}}_N(\mathcal{I}_h\boldsymbol{\varphi},T,t)} \\
& \quad = N^{-1} \mathbb{E}\left[\sum_{k,l=1}^{M}{\frac{(j_k-\delta_{kl})j_l}{2}\mathcal{S}^{\boldsymbol{j}^{kl}}_N(\mathcal{I}_h\boldsymbol{\varphi},T,t)\left(\rho_h(t),\mathcal{I}_h\left\{\nabla\phi_{k}^t\cdot\nabla\phi_{l}^t\right\}\right)_h}\right.\\
& \quad \quad \quad \left.-\left(\rho_h(0),\mathcal{P}_h^t\left(\mathcal{I}_h\left\{\nabla\phi_{k}^t\cdot\nabla\phi_{l}^t\right\}\right)\right)_h\right]\m t\\
&\quad\quad + N^{-1} \mean{\sum_{k,l=1}^{M}{\frac{(j_k-\delta_{kl})j_l}{2}\mathcal{S}^{\boldsymbol{j}^{kl}}_N(\mathcal{I}_h\boldsymbol{\varphi},T,t)\left(\rho_h(0),\mathcal{P}_h^t\left(\mathcal{I}_h\left\{\nabla\phi_{k}^t\cdot\nabla\phi_{l}^t\right\}\right)\right)_h}}\m t\\
&\quad \quad + N^{-1}\mean{\sum_{k,l=1}^{M}{\frac{(j_k-\delta_{kl})j_l}{2}\mathcal{S}^{\boldsymbol{j}^{kl}}_N(\mathcal{I}_h\boldsymbol{\varphi},T,t)\left(\rho^{-}_h(t),\nabla_h\phi_{k,h}^t\cdot\nabla_h\phi_{l,h}^t\right)_h}}\m t\\
& \quad\quad + N^{-1} \mean{\sum_{k,l=1}^{M}{\frac{(j_k-\delta_{kl})j_l}{2}\mathcal{S}^{\boldsymbol{j}^{kl}}_N(\mathcal{I}_h\boldsymbol{\varphi},T,t)(\rho_h(t),r_{k,l,h}^t})_h}\m t =: \sum_{i=1}^{4}{A_i\m t}.
\end{align*}
On the other hand
\begin{align*}
& \m\mean{\mathcal{T}^{\boldsymbol{j}}_N(\boldsymbol{\varphi},T,t)}\\ 
& \quad = N^{-1}\mean{\sum_{k,l=1}^{M}{\frac{(j_k-\delta_{kl})j_l}{2}\mathcal{T}^{\boldsymbol{j}^{kl}}_N(\boldsymbol{\varphi},T,t)\left(\frac{1}{N}\sum_{r=1}^{N}{\nabla\phi_{k}^t(\f{w}_r(t))\cdot\nabla\phi_{l}^t(\f{w}_r(t))}\right)}}\m t\\
&  \quad  = N^{-1}\mathbb{E}\!\left[{\sum_{k,l=1}^{M}{\frac{(j_k-\delta_{kl})j_l}{2}\mathcal{T}^{\boldsymbol{j}^{kl}}_N(\boldsymbol{\varphi},T,t)}}\right.\\
&  \quad\quad \quad \quad \left.\times\left(\frac{1}{N}\sum_{r=1}^{N}{\nabla\phi_k^t(\f{w}_r(t))\cdot\nabla\phi_{l}^t(\f{w}_r(t))}-\frac{1}{N}\sum_{r=1}^{N}{\mathcal{P}^t\left\{\nabla\phi_k^t\cdot\nabla\phi_l^t\right\}(\f{w}_r(0))}\right)\right]\m t\\
& \quad  \quad + N^{-1}\mean{\sum_{k,l=1}^{M}{\frac{(j_k-\delta_{kl})j_l}{2}\mathcal{T}^{\boldsymbol{j}^{kl}}_N(\boldsymbol{\varphi},T,t)\left(\frac{1}{N}\sum_{r=1}^{N}{\mathcal{P}^t\left\{\nabla\phi_{k}^t\cdot\nabla\phi_{l}^t\right\}(\f{w}_r(0))}\right)}}\m t\\
&  \quad=: \sum_{i=1}^{2}{B_i\m t}.
\end{align*}
It is straightforward to notice that $A_1-B_1$ can be settled using the estimates for the moments of order $j-1$, as (for each pair $k,l$) the exponent vector $\f{j}$ is decreased by two units to $\f{j}^{kl}$, while the additional test function $\nabla\phi^t_k\cdot\nabla\phi^t_l$ is picked up. The bound for $A_3$ relies on the Cauchy-Schwartz inequality, Corollary \ref{cor:1}, \eqref{eq:202}, and \eqref{r:3-a}. It reads
\begin{align*}
|A_3| 
&\stackrel{\mathclap{\eqref{eq:40}}}{\leq}  \sum_{k,l=1}^{M}{\frac{(j_k-\delta_{kl})j_l}{2}N^{-1}\|\nabla_h\phi_{k,h}^t\|_{\infty}\|\nabla_h\phi_{l,h}^t\|_{\infty}C\mathcal{E}\!\left(N,h\right)}\\
& \quad\quad\quad\quad\times \left[ \left\{2N^{-1} TC\left(d,\rho_{max},\rho_{min}\right) \right\}^{(2j-4)/2}(2j-4)^{3(2j-4)} \right]^{1/2}\\
& \quad\quad\quad \quad \times\left[\prod_{m=1}^{M}{\|\varphi_m\|_{C^{1+\Theta}}^{j_m-\delta_{km}-\delta_{lm}}}\right]\m t\\
& \quad\leq \left(\prod_{m=1}^{M}{\|\varphi_m\|_{C^{1+\Theta}}^{j_m}}\right) \left(\sum_{k,l=1}^{M}{\frac{(j_k-\delta_{kl})j_l}{2}}\right)N^{-1}C\mathcal{E}\!\left(N,h\right)\\
& \quad\quad\quad\quad\times \left\{N^{-1} TC\left(d,\rho_{max},\rho_{min}\right) \right\}^{(j-2)/2}(2j)^{3(j-2)}\m t\\
& \quad\stackrel{\mathclap{\eqref{eq:202}}}{\leq}  T^{j/2-1}\left\{N^{-1} C\left(d,\rho_{max},\rho_{min}\right) \right\}^{j/2}(2j)^{3(j-2)}\mathcal{E}\!\left(N,h\right)\\
& \quad\quad\quad\quad\times \left(\prod_{m=1}^{M}{\|\varphi_m\|_{C^{1+\Theta}}^{j_m}}\right) \left(\sum_{k,l=1}^{M}{\frac{(j_k-\delta_{kl})j_l}{2}}\right)\m t.
\end{align*}
The term $A_4$ may be bounded as follows 
\begin{align*}
 \int_{0}^{T}{|A_4|\m t} & \leq N^{-1}\sum_{k,l=1}^{M}{\frac{(j_k-\delta_{kl})j_l}{2}\int_{0}^{T}{\mean{\left|\mathcal{S}^{\boldsymbol{j}^{kl}}_N(\mathcal{I}_h\boldsymbol{\varphi},T,t)\right|\|\rho_h(t)\|_{h}\|r_{k,l,h}^t\|_{h}}}\m t}\\
&  \stackrel{\mathclap{\eqref{eq:49}}}{\leq} \,\, Ch^{p+1}N^{-1}\sum_{k,l=1}^{M}{\|\varphi_k\|_{C^{2+p+\Theta}}\|\varphi_l\|_{C^{2+p+\Theta}}\frac{(j_k-\delta_{kl})j_l}{2}}\\
&  \quad\quad\times \left(\max_{t\in[0,T]}{\mean{\left|\mathcal{S}^{\boldsymbol{j}^{kl}}_N(\mathcal{I}_h\boldsymbol{\varphi},T,t)\right|^2}^{1/2}}\right)T^{1/2}\left(\int_{0}^{T}{\mean{\|\rho_h(t)\|_h^2}\m t}\right)^{1/2}\\
&  \stackrel{\mathclap{\eqref{eq:202}\eqref{r:2b-a}\eqref{eq:40}}}{\leq} \,\,\,\,\,\quad Ch^{p+1}N^{-1}\sum_{k,l=1}^{M}{\|\varphi_k\|_{C^{2+p+\Theta}}\|\varphi_l\|_{C^{2+p+\Theta}}\frac{(j_k-\delta_{kl})j_l}{2}}\\
& \quad\quad \times C(d,\rho_{max},\rho_{min})\max\{T^{1/2};T\}\\
&  \quad\quad \times \left\{CN^{-1} TC(d,\rho_{max},\rho_{min}) \right\}^{(j-2)/2}(2j)^{3(j-2)}\left(\prod_{m=1}^{M}{\|\varphi_m\|_{C^{1+\Theta}}^{j_m-\delta_{km}-\delta_{lm}}}\right)\\
&  \leq Ch^{p+1}\left\{CN^{-1}\max\{T^{1/2};T\}C(d,\rho_{max},\rho_{min})\right\}^{j/2}\\
&  \quad\quad \times (2j)^{3(j-2)}\left(\sum_{k,l=1}^{M}{\frac{(j_k-\delta_{kl})j_l}{2}}\right)\left(\prod_{m=1}^{M}{\|\varphi_m\|_{C^{2+p+\Theta}}^{j_m}}\right).
\end{align*}
The difference $A_2-B_2$ is rewritten as 
\begin{align}
& A_2-B_2 \nonumber\\
&  \quad = N^{-1} \mathbb{E}\left[\sum_{k,l=1}^{M}{\frac{(j_k-\delta_{kl})j_l}{2}\left[\mathcal{S}^{\boldsymbol{j}^{kl}}_N(\mathcal{I}_h\boldsymbol{\varphi},T,t)-\mathcal{T}^{\boldsymbol{j}^{kl}}_N(\boldsymbol{\varphi},T,t)\right]}\right.\nonumber\\
& \quad\quad\quad\quad \times\left.\left(\rho_h(0),\mathcal{P}_h^t\left(\mathcal{I}_h\left\{\nabla\phi_{k}^t\cdot\nabla\phi_{l}^t\right\}\right)\right)_h\right]\m t\nonumber\\
&  \quad\quad - N^{-1}\mathbb{E}\left[\sum_{i,j=1}^{K}{\frac{(j_i-\delta_{ij})j_j}{2}T^{\boldsymbol{j}^{ij}}_N(\boldsymbol{\varphi},t,s)}\right.\nonumber\\
&  \quad\quad\quad \quad\times\left.\left(\frac{1}{N}\sum_{r=1}^{N}{\mathcal{P}^t\left\{\nabla\phi_k^t\cdot\nabla\phi_{l}^t\right\}(\f{w}_r(0))}-\left(\rho_h(0),\mathcal{P}^t_h\left(\mathcal{I}_h\left\{\nabla\phi_{k}^t\cdot\nabla\phi_{l}^t\right\}\right)\right)_h\right)\right]\m t\nonumber \\
&  \quad= N^{-1} \sum_{k,l=1}^{M}{\frac{(j_k-\delta_{kl})j_l}{2}\left(\mean{\mathcal{S}^{\boldsymbol{j}^{kl}}_N(\mathcal{I}_h\boldsymbol{\varphi},T,t)}-\mean{\mathcal{T}^{\boldsymbol{j}^{kl}}_N(\boldsymbol{\varphi},T,t)}\right)}\nonumber\\
& \quad\quad\quad\quad \times\left(\rho_h(0),\mathcal{P}_h^t\left(\mathcal{I}_h\left\{\nabla\phi_{k}^t\cdot\nabla\phi_{l}^t\right\}\right)\right)_h\m t\nonumber\\
& \quad \quad - N^{-1}\mathbb{E}\left[\sum_{k,l=1}^{M}{\frac{(j_k-\delta_{kl})j_l}{2}\mathcal{T}^{\boldsymbol{j}^{kl}}_N(\boldsymbol{\varphi},T,t)}\right.\nonumber\\
&  \quad\quad\quad \quad\times\left.\left(\frac{1}{N}\sum_{r=1}^{N}{\mathcal{P}^t\left\{\nabla\phi_{k}^t\cdot\nabla\phi_{l}^t\right\}(\f{w}_r(0))}-\left(\rho_h(0),\mathcal{P}^t_h\left(\mathcal{I}_h\left\{\nabla\phi_{k}^t\cdot\nabla\phi_{l}^t\right\}\right)\right)_h\right)\right]\m t\label{e:35}\\
& =:T_1+T_2\nonumber,
\end{align}
where equality \eqref{e:35} is valid because the term 
$$
\left(\rho_h(0),\mathcal{P}_h^t\left(\mathcal{I}_h\left\{\nabla\phi_{k}^t\cdot\nabla\phi_{l}^t\right\}\right)\right)_h
$$ 
is deterministic. The term $T_1$ is dealt with using the estimates of order $j-2$ (as, for each $k,l$, the exponent vector is decreased by two units to $\f{j}^{kl}$). The term $T_2$ is settled with the same arguments as for term $C_2$ in \eqref{eq:808}, with the additional use of the H\"older inequality and of \eqref{eq:41}. We obtain 
\begin{align*}
|T_2| & \leq N^{-1}\sum_{k,l=1}^{M}{\left[\frac{(j_k-\delta_{kl})j_l}{2}\right.} \left\{CN^{-1} T \right\}^{(j-2)/2}j^{j-2}\prod_{m=1}^{M}{\|\nabla\varphi_m\|_{\infty}^{j_m-\delta_{km}-\delta_{lm}}}\\
& \quad\quad \left.\times \left\{h^{p+1}C(d,\rho_{max},\rho_{min})\|\varphi_i\|_{C^{p+3}}\|\varphi_j\|_{C^{p+3}}\right\}\right]\m t\\
& \leq h^{p+1}T^{j/2-1}\left\{CN^{-1} C(d,\rho_{max},\rho_{min}) \right\}^{j/2}j^{j-2}\\
& \quad\quad \times\left(\prod_{m=1}^{M}{\|\varphi_m\|_{C^{p+3}}^{j_m}}\right)\left(\sum_{k,l=1}^{M}{\frac{(j_k-\delta_{kl})j_l}{2}}\right)\m t.
\end{align*}
Putting together all the estimates and integrating in time gives \eqref{eq:54}.
\end{proof}
\begin{rem}
The finite-difference error in \eqref{eq:54} accounts for two different errors:
\begin{itemize}[leftmargin=0.9 cm]
\item the difference between the initial conditions $\rho_{h,0}$ and the empirical density $\mu^N_0$, as well as the difference between the solutions to continuous and discrete backwards heat equations. This is captured in the term $A_2-B_2$ for the second order moment, and in the term $T_2$ for higher moments.
\item the difference between $\mathcal{I}_h(\nabla\phi_k^t\cdot\nabla\phi_l^t)$ and $\nabla_{h}\phi_{k,h}^t\cdot\nabla_h\phi_{l,h}^t$. As anticipated in Subsection \ref{ideas}, \emph{Block 3}, the high-order accuracy of the difference between the solutions to continuous and discrete backwards heat equations relies on the discrete final datum to be the interpolant of the continuous final datum. Since $\nabla_{h}\phi_{k,h}^t\cdot\nabla_h\phi_{l,h}^t$ does not interpolate $\nabla\phi_k^t\cdot\nabla\phi_l^t$ in general, we quantify $\mathcal{I}_h(\nabla\phi_k^t\cdot\nabla\phi_l^t)-\nabla_{h}\phi_{k,h}^t\cdot\nabla_h\phi_{l,h}^t$.
\end{itemize}
\end{rem}

\subsection{Proof of Theorem \ref{main2}}\label{ss:end}

\emph{Step 1: Interpreting \eqref{eq:54}}. The recursive relation \eqref{eq:54} may be visualised in the following way:
\begin{enumerate}[leftmargin=0.9 cm]
\item[i)] Each moment of order $j$ produces residuals $\mbox{Err}_{neg}$ and $\mbox{Err}_{num}$.
\item[ii)] Each moment of order $j$ is linked recursively to a collection of moments of order $j-1$ ($A^{j-1}_{recursion}$) and a collection of moments of order $j-2$ ($A^{j-2}_{recursion}$).
\item[iii)] The overall bound for $\mathcal{D}(\boldsymbol{j},\boldsymbol{\varphi},T)$ is given by summing all the residuals for all moments found by exhausting the recursive relation. More specifically, it holds 
$$
\mathcal{D}(\boldsymbol{j},\boldsymbol{\varphi},T) \leq \sum_{K=0}^{j-2}\mathcal{R}_K,
$$ 
where $\mathcal{R}_K$ is the sum of all residuals associated with the moments explored after \emph{exactly} $K$ steps. Therefore, we only need to suitably control $\mathcal{R}_K$ for $K=0,\dots,j-2$. In order to do this, we need the following auxiliary bound.
\end{enumerate} 
\vspace{1 pc}
\emph{Step 2: Auxiliary bound}. At every step of the recursive relation, the sets of test functions which are fed into the lower order terms $A^{j-1}_{recursion}$ and $A^{j-2}_{recursion}$ are modifications of the current set of test functions, specifically:
\begin{itemize}[leftmargin=0.9 cm]
\item 
in the case of $A^{j-1}_{recursion}$, one instance for each of two functions $\varphi_k,\varphi_l$ is replaced by the product $\nabla\varphi_k \cdot \nabla\varphi_l$;
\item 
in the case of $A^{j-2}_{recursion}$, one instance for each of two functions $\varphi_k,\varphi_l$ is removed from the set of test functions, and a pre-factor $\|\varphi_k\|_{C^{1+\Theta}}\|\varphi_l\|_{C^{1+\Theta}}$ is gained.
\end{itemize}
It is thus natural to define the object
\begin{align*}
\left\{(\f{\psi}_{K,r},\f{j}_{K,r}),Y_{K,r}\right\},
\end{align*}
where $r$ is a given way of exhausting the recursive relation for $K$ steps (i.e., a sequence of $K$ moves dictating whether moments of type $A^{j-1}_{recursion}$ or $A^{j-2}_{recursion}$ are explored at each step), where $\f{\psi}_{K,r}$ is the set of test functions after $K$ steps with sequence $r$, where $\f{j}_{K,r}$ is the corresponding set of powers, and where $Y_{K,r}$ is the overall pre-factor cumulated from all the moments of type $A^{j-2}_{recursion}$ for the sequence $r$.

For each $\gamma\in\mathbb{N}_0$, we have the bound
\begin{align}\label{eq:60}
& \left(\prod_{m=1}^{M_{K,r}}{\|\psi_{K,r,m}\|_{C^{\gamma}}^{j_{K,r,m}}}\right) \times |Y_{K,r}|\leq 
j^{2K}j^{j(\max\{\gamma;1+\Theta\}+1)}\cdot\prod_{m=1}^{M}{\|\varphi_m\|_{C^{\max\{\gamma;1+\Theta\}+K}}^{j_m}}, 
\end{align}
which is justified by the following observations:
\begin{itemize}[leftmargin=0.9 cm]
\item The number of occurrences of the original functions $\f{\varphi}$ (i.e., $\f{j}$) is preserved, regardless of the path $r$. This is straightforward to verify by direct inspection of how the recursive terms $A^{j-1}_{recursion}$ and $A^{j-2}_{recursion}$ handle the test functions.
\item The factor $j^{2K}$ provides a bound on the product of the number of individual addends making up the functions $\{\psi_{K,r,m}\}_m$ and of the number of individual addends making up the functions of type $\psi_{\tilde{K},r,m}$ (where $\tilde{K}< K$) found in the term $Y_{K,r}$. This is a simple consequence of the fact that, whenever a step of type $A^{j-1}_{recursion}$ is performed, such product can be multiplied by at most $K\cdot K = K^2$ (i.e., by the product of the maximum lengths of the addends making up the two functions $\phi_k$ and $\phi_l$ which give rise to the new test function $\nabla\phi_k\cdot \nabla\phi_l$). When a step of type $A^{j-2}_{recursion}$ is performed, such product does not increase.
\item The factor 
$
\prod_{m=1}^{M}{\|\varphi_m\|_{C^{\max\{\gamma;1+\Theta\}+K}}^{j_m}}
$
takes into account the evaluation of the norms for all functions (both $\{\psi_{K,r,m}\}_m$ and those associated with $Y_{K,r}$) by using the most restrictive exponent between $1+\Theta$ (needed in any step of type $A^{j-2}_{recursion}$) and $\gamma$ (which is the exponent we are interested in), and adding $K$ (to reflect the unitary increment of differentiation entailed by each step of type $A^{j-1}_{recursion}$). 
\item
The term
$
j^{(\max\{\gamma;1+\Theta\}+1)}
$
is associated with the pre-factor of the inequality 
$$
\left\|\prod_{i=1}^{\ell}{f_i}\right\|_{C^{\beta}}\leq \ell^{\beta+1}\prod_{i=1}^{\ell}{\|f_i\|_{C^{\beta}}}
$$ 
applied with $\ell \leq j $ ($j$ is the maximum number of factors in the addends of type $\prod_{i=1}^{\ell}{f_i}$ making up any function $\psi_{K,r,m}$ and any function associated with $Y_{K,r}$), and with $\beta = \max\{\gamma;1+\Theta\}$. The overall pre-factor $j^{j(\max\{\gamma;1+\Theta\}+1)}$ results from multiplying $j^{(\max\{\gamma;1+\Theta\}+1)}$ by itself $j$ times ($j$ being an upper bound for the total number of functions $\psi_{K,r,m}$ together with all functions associated with $Y_{K,r}$).
\end{itemize}
Crucially, \eqref{eq:60} only depends on $K$ and $j$, and not on the specific path $r$.\\ \\
\emph{Step 3: Bounding $\mathcal{R}_K$}. 
The quantity $2^{K}j^{4(K+1)}=2^K\times j^{4K} \times j^4$ is a bound for both the number of residuals of type $\mbox{Err}_{neg}$ and $\mbox{Err}_{num}$ associated with the moments explored after exactly $K$ steps: Such a quantity is the product of $2^{K}$ (accounting for the recursive splitting of \eqref{eq:54} into two families of moments of lower order), of $j^{4K}$ (accounting for a bound of the pre-factor $\sum_{k,l=1}^{M}{(j_k-\delta_{kl})j_l/2}$ multypling each of the two families of moments), and of $j^4$ (accounting for a bound of the pre-factor $\sum_{k,l=1}^{M}{(j_k-\delta_{kl})j_l/2}$ multypling the residual terms).
Using \eqref{eq:54} and \eqref{eq:60}, we obtain
\begin{align}\label{e:36}
\mathcal{R}_K\leq\left(2^{K}j^{4(K+1)}\right) & \times \left[\underbrace{\left\{N^{-1} TC(d,\rho_{max},\rho_{min})\right\}^{j/2}(2j)^{3(j-2)}\mathcal{E}\!\left(N,h\right)}_{\mbox{Err}_{neg}\mbox{ contributions, see }\eqref{eq:54}}\right.\nonumber\\
& \quad\quad\quad\quad\quad\quad\quad\quad\times\underbrace{j^{2K+(2+\Theta)j}\left(\prod_{m=1}^{M}{\|\varphi_m\|_{C^{1+\Theta+K}}^{j_m}}\right)}_{\mbox{see }\eqref{eq:60}}\nonumber\\
& \quad\quad + \underbrace{h^{p+1}\left\{N^{-1}\max\{T^{1/2};T\}C(d,\rho_{max},\rho_{min})\right\}^{j/2}(2j)^{3(j-2)}}_{\mbox{Err}_{num}\mbox{ contributions, see }\eqref{eq:54}}\nonumber\\
& \quad\quad\quad\quad\quad\quad\quad\quad\left.\times\underbrace{j^{2K+(3+p+\Theta)j}\left(\prod_{m=1}^{M}{\|\varphi_m\|_{C^{p+3+\Theta+K}}^{j_m}}\right) }_{\mbox{see }\eqref{eq:60}}\right].
\end{align}
\emph{Step 4: Concluding the argument}. Since $\mathcal{D}(\boldsymbol{j},\boldsymbol{\varphi},T)\leq \sum_{K=0}^{j-2}{\mathcal{R}_K}$, we obtain
\begin{align*}
\mathcal{D}(\boldsymbol{j},\boldsymbol{\varphi},T) & \stackrel{\mathclap{\eqref{e:36}}}{\leq} \sum_{K=0}^{j-2}{\left[\left(2^{K}j^{4(K+1)}\right)\left\{N^{-1} TC(d,\rho_{max},\rho_{min}) \right\}^{j/2}(2j)^{3(j-2)}\right.}\\
& \quad\quad\quad \times\mathcal{E}\!\left(N,h\right)j^{2K}j^{j(\max\{1+\Theta;1+\Theta\}+1)}\left(\prod_{m=1}^{M}{\|\varphi_m\|_{C^{1+\Theta+K}}^{j_m}}\right) \nonumber\\
& \quad +\left(2^{K}j^{4(K+1)}\right)h^{p+1}\left\{N^{-1}\max\{T^{1/2};T\}C(d,\rho_{max},\rho_{min})\right\}^{j/2}(2j)^{3(j-2)}\\
&\quad\quad\quad \left.\times j^{2K}j^{j(\max\{p+3+\Theta;1+\Theta\}+1)}\left(\prod_{m=1}^{M}{\|\varphi_m\|_{C^{p+3+\Theta+K}}^{j_m}}\right)\right]\\
& \leq \left\{N^{-1} T C(d,\rho_{max},\rho_{min})\right\}^{j/2}j^{C_1j+C_2}\mathcal{E}\!\left(N,h\right)\left(\prod_{m=1}^{M}{\|\varphi_m\|_{C^{j-1+\Theta}}^{j_m}}\right)\\
& \quad + h^{p+1}\left\{N^{-1} \max\{T^{1/2};T\} C(d,\rho_{max},\rho_{min})\right\}^{j/2}\\
& \quad \quad \quad \times j^{C_3j+C_4}\left(\prod_{m=1}^{M}{\|\varphi_m\|_{C^{p+j+1+\Theta}}^{j_m}}\right),
\end{align*}
which -- up to trivial rescaling in $N^{1/2}$ -- is precisely \eqref{eq:63c}.

\subsection{Exponentially decaying estimate for $\mean{\|\rho_h^{-}\|^2_h}$ and moment bounds for $\rho_h$}\label{ss:exp}

\begin{proposition}
Let the assumptions and notation of the finite difference case of Theorem~\ref{main1} be in place; in particular, let $\rho_h$ be a solution to the Dean--Kawasaki equation discretised using finite elements in the sense of \eqref{e:3}. Assuming in addition the scaling \eqref{eq:202}, namely $h\geq C(d,\rho_{min},\rho_{max})N^{-1/d} |\log N|^{2/d}(T+1)$, we then have the estimate
\begin{align}\label{r:1-a}
&\mathbb{P}\Bigg[\sup_{\f{x}\in G_{h,d}, t\in [0,T]} |\rho_h-\mathbb{E}[\rho_h]|(\f{x},t) \geq B \frac{\rho_{min}}{4} \Bigg]
\nonumber\\&
\quad \leq C \exp\bigg(-\frac{\rho_{min}B^{1/2} N^{1/2}h^{d/2}}{C\rho_{max}^{1/2}}\bigg)
+C \exp\big(-cB^{1/4} h^{-1}\big)
\end{align}
for any $B\geq 1$.
In particular, we can deduce
\begin{align}
\mathbb{E}\Bigg[\sup_{\f{x}\in G_{h,d},t\in [0,T]} |\rho_h-\mean{\rho_h}|^j(\f{x},t) \Bigg]^{1/j}
& \leq C(d,\rho_{max},\rho_{min}) j^4 \label{r:2a-a}\\
\mathbb{E}\Bigg[\sup_{\f{x}\in G_{h,d},t\in [0,T]} |\rho_h(\f{x},t)|^j \Bigg]^{1/j} 
& \leq C(d,\rho_{max},\rho_{min}) j^4 \label{r:2b-a}
\end{align}
for any $j\geq 1$, as well as
\begin{align}\label{r:3-a}
\mean{\sup_{t\in [0,T]} \|\rho_h^{-}(t)\|^{2}_{h}} \leq C(d,\rho_{min},\rho_{max}) \left\{\exp\bigg(-\frac{\rho_{min} N^{1/2}h^{d/2}}{C\rho_{max}^{1/2}}\bigg)
+ \exp\big(-ch^{-1}\big)\right\}.
\end{align}
\end{proposition}

\begin{proof}
We split the proof into several steps.

{\bf Step 1: energy estimates for test functions}.
In order to evaluate $\rho_h(\f{x}_0,T)$ at a given point $\f{x}_0$, we choose $\phi_h(\cdot,T)\in \Lh$ as the function satisfying $( \phi_h(\cdot,T), \eta_h )_h = \eta_h(\f{x}_0)$ for all $\eta_h\in \Lh$ and evolve $\phi_h$ in time by the backward heat equation
\begin{align}
\label{BackwardHeat}
\partial_t \phi_h = -\frac{1}{2} \Delta_h \phi_h.
\end{align}
By the standard energy estimate for the discrete heat equation 
we get
\begin{align}
\label{EstimateEnergyDiscreteDiracHeatEquation}
\int_0^T  \|\nabla \phi_h \|_h^2 \,\m t
\leq 2 \|\phi_h(T) \|_h^2 
\leq C h^{-d}.
\end{align}

{\bf Step 2: exponentially decaying bounds for $|\rho_h-\mean{\rho_h}|(\f{x}_0)$ for chosen point $\f{x}_0$}.
Using \eqref{eq:702}, \eqref{b:5}, and \eqref{BackwardHeat}, we obtain by the It\^o formula for any positive integer $j$
\begin{align*}
\m \big(\rho_h-\mathbb{E}[\rho_h], \phi_h\big)_h^j & = j \big(\rho_h-\mathbb{E}[\rho_h], \phi_h\big)_h^{j-1} 
N^{-1/2}\!\!\!\!\!\!\!\sum_{(\f{y},\ell)\in(G_{h,d},\{1,\dots,d\})}{\!\!\left(\mathcal{F}_\rho\f{e}^d_{h,\f{y},\ell},\nabla_h \phi_h \right)_h\m\beta_{(\f{y},\ell)}}
\\&~~~~~~
+\frac{j(j-1)}{2} \big(\rho_h-\mathbb{E}[\rho_h], \phi_h\big)_h^{j-2} N^{-1} (\rho_h^+ , |\nabla \phi_h|^2)_h \,\m t.
\end{align*}
In particular, $(\rho_h-\mathbb{E}[\rho_h], \phi_h )_h$ is a martingale.
Integrating in time up to a stopping time $T_s$ and taking the expected value, we obtain
\begin{align*}
&\mean{\left(\big(\rho_h-\mathbb{E}[\rho_h]\big)(\cdot,T\wedge T_s), \phi_h(\cdot,T\wedge T_s) \,\right)_h^j}
\\&~~
=\frac{j(j-1)}{2} \mean{\int_0^{T\wedge T_s}
\big(\rho_h-\mathbb{E}[\rho_h], \phi_h \big)_h^{j-2} N^{-1} (\rho_h^+, |\nabla \phi_h|^2)_h \m t }.
\end{align*}
Choosing $T_s$ for arbitrary but fixed $B\geq 1$ as
\begin{align*}
T_s:=\inf\Big\{t>0:\sup_{\f{x}\in \Ghd} |\rho_h-\mathbb{E}[\rho_h]|(t,\f{x}) \geq B\frac{\rho_{min}}{2}\Big\},
\end{align*}
we get using $\rho_{max}\geq \rho_{min}$ and the assumption $|\mathbb{E}[\rho_h]|\leq \rho_{max}$
\begin{align*}
& \mean{\left(\big(\rho_h-\mathbb{E}[\rho_h]\big)(\cdot,T\wedge T_s), \phi_h(\cdot,T\wedge T_s) \,\right)_h^j}
\\&~~
\leq j(j-1) N^{-1} B \rho_{max} \mean{ \int_0^{T\wedge T_s}
\big(\rho_h-\mathbb{E}[\rho_h], \phi_h \big)_h^{j-2} \|\nabla \phi_h \|^2 \,\m t }
\\&~~
\leq j^2 N^{-1} B \rho_{max} \mean{
\sup_{t\in [0,T\wedge T_s]} \big(\rho_h-\mathbb{E}[\rho_h], \phi_h\big)_h^{j}}^{(j-2)/j} \int_0^{T} \|\nabla \phi_h \|^2 \,\m t
\\&~~
\stackrel{\eqref{EstimateEnergyDiscreteDiracHeatEquation}}{\leq} Cj^2 N^{-1} B \rho_{max} \mathbb{E}\Bigg[
\sup_{t\in [0,T\wedge T_s]} \big(\rho_h-\mathbb{E}[\rho_h], \phi_h \big)_h^{j}\Bigg]^{(j-2)/j} h^{-d}.
\end{align*}
Using Doob's martingale inequality, we deduce for nonnegative even integers $j$
\begin{align*}
&\mean{
\sup_{t\in [0,T\wedge T_s]} \big(\rho_h-\mathbb{E}[\rho_h], \phi_h  \big)_h^{j}}^{2/j}
\leq C j^2 \frac{B \rho_{max}}{N h^{d}}.
\end{align*}
Raising both sides to the power $j/2$ and using Chebyshev's inequality, we get after optimizing in $j$
\begin{align*}
\mathbb{P}\left[
\sup_{t\in [0,T\wedge T_s]} \left| \big(\rho_h-\mathbb{E}[\rho_h], \phi_h \big)_h \right| \geq B \frac{\rho_{min}}{8} \right]
\leq 2 \exp\left(-\frac{\rho_{min} B N^{1/2}h^{d/2}}{C B^{1/2} \rho_{max}^{1/2}}\right).
\end{align*}
In particular, we deduce by the definition of $\phi_h(\cdot,T)$
\begin{align*}
\mathbb{P}\left[T\leq T_S \text{ and } |\rho_h-\mathbb{E}[\rho_h]|(\f{x}_0,T) \geq B\frac{\rho_{min}}{8} \right]
\leq 2 \exp\left(-\frac{\rho_{min} B^{1/2} N^{1/2}h^{d/2}}{C\rho_{max}^{1/2}}\right).
\end{align*}

{\bf Step 3: extending the estimate to finitely many time points in $[0,T\wedge T_s]$}.
Applying the previous estimate for all $\f{x}_0 \in \Ghd$ (there are $\propto h^{-d}$ of such points), and for all times $h^\beta$, $2h^\beta$, $3h^\beta$, $\ldots$, for some $\beta>0$ to be chosen, we obtain
\begin{align}
\label{CoveringLinftyEstimate}
&\mathbb{P}\left[
||(\rho_h-\mathbb{E}[\rho_h])(\cdot,ih^\beta)||_{L^\infty} \geq B\frac{\rho_{min}}{8} \text{ for some }i\in \mathbb{N}\text{ with }ih^\beta\leq T\wedge T_S  \right]
\nonumber \\ 
& \quad \leq C h^{-d} \frac{T}{h^\beta} \exp\left(-\frac{\rho_{min} B^{1/2} N^{1/2}h^{d/2}}{C\rho_{max}^{1/2}}\right).
\end{align}

{\bf Step 4: extending the estimate to all times in $[0,T]$}.
It only remains to pass from the discrete times $ih^\beta$ to all times $t$ and to remove the restriction to times $t \leq T_S$. Let $e_k\in \Lh$ be nodal function satisfying $e_k(\f{x}_j) = \delta_{kj}$. Then the differential
\begin{align*}
\m (\rho_h, e_k )_h
=\frac{1}{2} \left( \Delta \rho_h, e_k \right)_h 
-N^{-1/2} \sum_{(\f{y},\ell)\in(G_{h,d},\{1,\dots,d\})}{\!\!\left(\mathcal{F}_\rho\f{e}^d_{h,\f{y},\ell},\nabla_h e_k \right)_h\m\beta_{(\f{y},\ell)}},
\end{align*}
entails, using in a second step also Doob's maximal inequality and abbreviating $\mathcal{W}(\rho_h^+,e_k) := \sum_{(\f{y},\ell)\in(G_{h,d},\{1,\dots,d\})}{\!\!\left(\mathcal{F}_\rho\f{e}^d_{h,\f{y},\ell},\nabla_h e_k \right)_h\m\beta_{(\f{y},\ell)}}$
\begin{align*}
&\sum_k \mean{\chi_{ih^\beta\leq T_S} \sup_{t\in [ih^\beta,(i+1)h^\beta]} 
\left| (\rho_h(\cdot,t), e_k)_h -(\rho_h(\cdot,ih^\beta), e_k)_h \right|^j }^{1/j}
\\&~~
\leq Ch^{-2} \sum_k\mean{\chi_{ih^\beta\leq T_S} \bigg(\int_{ih^\beta}^{(i+1)h^\beta} \left|  (\rho_h(\cdot,t), e_k)_h \right| \,\m t\bigg)^j}^{1/j}
\\&~~~~~
+C N^{-1/2}\sum_k \mean{\chi_{ih^\beta\leq T_S} \sup_{t\in [ih^\beta,(i+1)h^\beta]} \big| \mathcal{W}(\rho_h^+,e_k)(t)-\mathcal{W}(\rho_h^+,e_k)(ih^\beta)\big|^{j}}^{1/j}
\\&~~
\leq Ch^{-2} \sum_k \mean{\chi_{ih^\beta\leq T_S} \bigg(\int_{ih^\beta}^{(i+1)h^\beta} \left|  (\rho_h(\cdot,t), e_k)_h \right| \,\m t\bigg)^j }^{1/j}
\\&~~~~~
+CN^{-1/2} \sum_k \mathbb{E}\Bigg[\chi_{ih^\beta\leq T_S} \big| \mathcal{W}(\rho_h^+,e_k)((i+1)h^\beta)-\mathcal{W}(\rho_h^+,e_k)(ih^\beta)\big|^{j}\Bigg]^{1/j}.
\end{align*}
Using the triangle inequality for the first term on the right-hand side and a (straightforward but rather pessimistic) estimate on the quadratic variation of $\mathcal{W}$, we obtain
\begin{align*}
&\sum_k \mean{\chi_{ih^\beta\leq T_S} \sup_{t\in [ih^\beta,(i+1)h^\beta]} \left| (\rho_h(\cdot,t), e_k)_h - (\rho_h(\cdot,ih^\beta), e_k )_h \right|^j }^{1/j}
\\&~~
\leq Ch^{\beta-2} \sum_k \mean{\chi_{ih^\beta\leq T_S} \sup_{t\in [ih^\beta,(i+1)h^\beta]} \left|(\rho_h(\cdot,t), e_k)_h - (\rho_h(\cdot,ih^\beta), e_k)_h \right|^j }^{1/j}
\\&~~~~~
+Ch^{\beta-2} \sum_k \mean{\chi_{ih^\beta\leq T_S} \left| (\rho_h(\cdot,ih^\beta), e_k)_h \right|^j}^{1/j}
\\&~~~~~
+Cj h^{-2} N^{-1/2} \sum_k \mean{\chi_{ih^\beta\leq T_S} \bigg(\int_{ih^\beta}^{(i+1)h^\beta} (\rho_h^+,1)_h \,\m t\bigg)^{j/2}}^{1/j}.
\end{align*}
By absorption, the triangle inequality, the fact that $\sum_k 1 \leq C h^{-d}$, this implies for $h\leq c(\beta)$
\begin{align*}
&\sum_k \mean{\chi_{ih^\beta\leq T_S} \sup_{t\in [ih^\beta,(i+1)h^\beta]} \left| (\rho_h(\cdot,t), e_k )_h - (\rho_h(\cdot,ih^\beta), e_k)_h \right|^j }^{1/j}
\\&~~
\leq Ch^{\beta-2} \sum_k \mean{\chi_{ih^\beta\leq T_S} \left| (\rho_h(\cdot,ih^\beta), e_k)_h \right|^j}^{1/j}
\\&~~~~~
+Cj h^{\beta/2-d-2} N^{-1/2} \sum_l \mean{\chi_{ih^\beta\leq T_S} \left| ( \rho_h(\cdot,ih^\beta), e_l)_h \right|^{j/2}}^{1/j}
\\&~~~~~
+Cj h^{\beta/2-d-2} N^{-1/2} 
\\&~~~~~~~~~~~~~
\times\sum_l \mean{\chi_{ih^\beta\leq T_S} \sup_{t\in [ih^\beta,(i+1)h^\beta]} \left|  (\rho_h(\cdot,t) ,e_l)_h - (\rho_h(\cdot,ih^\beta), e_l)_h \right|^{j/2} }^{1/j}.
\end{align*}
Using Young's inequality and absorbing as well as using the fact that for $ih^\beta \leq T_S$ we have $|\rho_h|\leq (B+1) \rho_{max}$, we obtain
\begin{align*}
&\sum_k \mean{\chi_{ih^\beta\leq T_S} \sup_{t\in [ih^\beta,(i+1)h^\beta]} \left| ( \rho_h(\cdot,t), e_k)_h  - ( \rho_h(\cdot,ih^\beta), e_k)_h \right|^j }^{1/j}
\\&~~
\leq Ch^{\beta-d-2} (B+1) \rho_{max}
+Cj h^{\beta/2-d-2} (B+1)^{1/2} N^{-1/2} \rho_{max}^{1/2} + C j^2 h^{\beta-2d-4} N^{-1}.
\end{align*}
For $\beta\geq 6d+8$ and for all $h\leq c(\rho_{min},\rho_{max})$, we obtain
\begin{align}
\label{IntermediateLinftyEstimate}
& \mathbb{P}\left[ih^\beta\leq T_S,\sup_{t\in [ih^\beta,(i+1)h^\beta]} \|\rho_h(\cdot,t)-\rho_h(\cdot,jh^\beta)\|_{L^\infty} \geq B\frac{\rho_{min}}{10} \right] \leq C \exp(-B^{1/4} h^{-\beta/8}).
\end{align}
{\bf Step 5: obtaining \eqref{r:1-a}}.
Overall, from \eqref{CoveringLinftyEstimate} and \eqref{IntermediateLinftyEstimate} we conclude
\begin{align*}
&\mathbb{P}\Bigg[\sup_{t\in [0,T]} \big(\rho_h-\mathbb{E}[\rho_h]\big)(\f{x}_0,t) \geq B \frac{\rho_{min}}{4} \Bigg]
\\&
\quad \leq C T h^{-\beta-d} \exp\bigg(-\frac{\rho_{min}B^{1/2} N^{1/2}h^{d/2}}{C\rho_{max}^{1/2}}\bigg)
+C \exp\big(-cB^{1/4} h^{-1}\big).
\end{align*}
Upon choosing $h\geq C(d,\rho_{min},\rho_{max}) N^{-1/d} |\log N|^{2/d}(1+T)$, this implies \eqref{r:1-a}.

{\bf Step 6: obtaining \eqref{r:2a-a}--\eqref{r:2b-a}}.
For any $z\geq 0$, we use \eqref{r:1-a} to write
\begin{align*} 
& \mathbb{P}\left[\sup_{x\in\domain,t\in [0,T]} |\rho_h-\mean{\rho_h}|^j(x,t) > z\right] \\
&~~ \stackrel{\mathclap{\eqref{r:1-a}}}{\leq} \mathbf{1}_{\{z^{1/j} < \rho_{min}/4\}} \\
&~~ \quad + \mathbf{1}_{\{z^{1/j} \geq \rho_{min}/4\}} \left(C T \exp\bigg(-\frac{\rho_{min}^{1/2}z^{1/2j} N^{1/2}h^{d/2}}{C\rho_{max}^{1/2}}\bigg)
+C \exp\big(-c\rho_{min}^{-1/4}z^{1/4j} h^{-1}\big)\right)
\end{align*}
For a non-negative random variable $Z$, we know that $\mean{Z}=\int_{0}^{\infty}{\mathbb{P}(Z> z)\m z}$. We set $Z:=\sup_{x\in\domain,t\in [0,T]} |\rho_h-\mean{\rho_h}|^j(x,t)$ and deduce 
\begin{align*}
& \mean{\sup_{x\in\domain,t\in [0,T]} |\rho_h-\mean{\rho_h}|^j(x,t)} \\
& \leq \int_{0}^{(\rho_{min}/4)^j}{\m z} \\
& \quad + \int_{(\rho_{min}/4)^j}^{\infty}{\left(C T \exp\bigg(-\frac{\rho_{min}^{1/2}z^{1/2j} N^{1/2}h^{d/2}}{C\rho_{max}^{1/2}}\bigg) + C \exp\big(-c\rho_{min}^{-1/4}z^{1/4j} h^{-1}\big)\right)\m z}\\
& \leq C^{j}(\rho_{min},\rho_{max})(1+T)j^{4j}\left\{(N^{-1}h^{-d})^{Cj}+1\right\},
\end{align*}
where we have used the Gaussian moments estimates in the last inequality, and \eqref{r:2a-a} is proved. Inequality \eqref{r:2b-a} follows from the triangle inequality, the assumption $\mean{|\rho_h|}\leq \rho_{max}$ and \eqref{r:2a-a}.

{\bf Step 7: obtaining \eqref{r:3-a}}. We use the H\"older inequality and the lower bound $\mean{\rho_h}\geq \rho_{min}$ and obtain \eqref{r:3-a} via the estimate
\begin{align*}
 \mean{\sup_{t\in [0,T]} \|\rho_h^{-}(t)\|_h^{2}}
& \leq C\mean{\sup_{t\in [0,T]} \|\mathbf{1}_{\{|(\rho_h-\mean{\rho_h})(t)| \geq \rho_{\min}\}}\cdot(\rho_h-\mean{\rho_h})(t)\|^{2}_{L^\infty}} \\
&   \leq C\mean{\sup_{t\in [0,T]} \|\mathbf{1}_{\{|(\rho_h-\mean{\rho_h})(t)| \geq \rho_{\min}\}}\|^4_{L^\infty}}^{1/2} \\
& \quad \quad \times \mean{\sup_{t\in [0,T]}\|(\rho_h-\mean{\rho_h})(t)\|^{4}_{L^\infty}}^{1/2}\\
&   \stackrel{\mathclap{\eqref{r:1-a}\eqref{r:2a-a}}}{\leq} \,\,\, C(d,\rho_{min},\rho_{max}) \left[\exp\left(-\frac{\rho_{min} (Nh^{d})^{1/2}}{C\rho_{max}^{1/2}}\right)
+ \exp(-ch^{-1})\right]\!.\qedhere
\end{align*}
\end{proof}

\section{Numerical examples}\label{num}

In this section, we give numerical examples that illustrate that the Dean--Kawasaki equation correctly captures the fluctuations of diffusing non-interacting particles\footnote{The datasets generated and analysed during the current study are available from the corresponding author on reasonable request.}. 
We limit our attention to the case $d=1$.

To compute the motion of $N$ Brownian particles, we perform a direct simulation based on the transition probabilities; this is feasible as our numerical experiments only concern empirical measures $\mu_t^N$ at two different times $T_1$ and $T_2$ (see below).
Our discretisation of the Dean--Kawasaki equation is obtained as follows:
\begin{itemize}[leftmargin=0.9 cm]
\item For the spatial discretisation of the Dean--Kawasaki equation \eqref{DeanKawasaki}, we use the finite difference scheme from Definition~\ref{def:1} with order $p=1$.
\item To discretise the spatially semi-discrete equation in time, we use the (two-step) BDF2 scheme (see, e.g., \cite{SDELab}). The first timestep is performed using an explicit treatment for the noise and a mixed implicit-explicit Euler scheme for the deterministic diffusion.
\item Overall, our discrete scheme for the Dean--Kawasaki equation \eqref{DeanKawasaki} reads for the first timestep
\begin{align}\label{eq:4000}
\rho_{h}^{\Delta t}
&=\rho_{h}^{0}+\Big(\tfrac{1}{4} \Delta_h \rho_h^{\Delta t}+\tfrac{1}{4} \Delta_h \rho_h^{0}\Big)\Delta t\nonumber
\\&~~~~
+\sum_{{\f{y}\in G_{h,1}}} \nabla_h \cdot \Big(\sqrt{(\rho_{h}^{0})_+} e^1_{\f{y}} \Big)\big(\tilde \beta_{\f{y}}(\Delta t) - \tilde \beta_{\f{y}}(0)\big),
\end{align}
and for the $(m+1)$-th timestep, $m\geq 1$,
\begin{align}\label{eq:4001}
~~~~~~~~
\rho_{h}^{(m+1)\Delta t}
&=
\tfrac{4}{3}\rho_{h}^{m\Delta t}-\tfrac{1}{3}\rho_{h}^{(m-1)\Delta t}
+\tfrac{2}{3} \Delta_h \rho_h^{(m+1)\Delta t} \Delta t \nonumber \\
& \quad -\tfrac{1}{3}\sum_{{\f{y}\in G_{h,1}}} \nabla_h \cdot \Big(\sqrt{(\rho_{h}^{(m-1)\Delta t})_+} {e}^1_{\f{y}} \Big)\big(\tilde \beta_{\f{y}}(m\Delta t) - \tilde \beta_{\f{y}}((m-1) \Delta t)\big) \nonumber\\
& \quad+ \sum_{{\f{y}\in G_{h,1}}} \nabla_h \cdot \Big(\sqrt{(\rho_{h}^{m\Delta t})_+} e^1_{\f{y}} \Big)\big(\tilde \beta_{\f{y}}((m+1)\Delta t) - \tilde \beta_{\f{y}}(m \Delta t)\big),
\end{align}
where $(\beta_{\f{y}})$ are independent Brownian motions.
\item 
We place the initial positions $\{\f{w}_k(0)\}_{k=1}^N$ of the Brownian particles only at grid points of $G_{h,1}$. Consequently, we define the initial condition $\rho_h(0)$ by requiring that the equality $\langle \mu^N_0, \eta \rangle = (\rho_h(0),\mathcal{I}_h\eta)_h$ holds for any test function $\eta$. This way, we avoid any error caused by deviating initial conditions. 
\item As we are primarily interested in scaling in $h$ and $N$, we make the following choices: 
\begin{itemize}[leftmargin=0.9 cm]
\item we set the time-step $\Delta t:=0.001$, which, according to our numerical convergence tests, is small enough for the spatial discretisation error to dominate over the time error, and
\item we keep the discretisation parameter $h$ above or equal to the threshold $2\pi\cdot2^{-7}\approx 0.05$, so that the finite difference error dominates over the error associated with the negative part of $\rho_h$.
\end{itemize}

\end{itemize}

\begin{figure}
\includegraphics[scale=0.23]{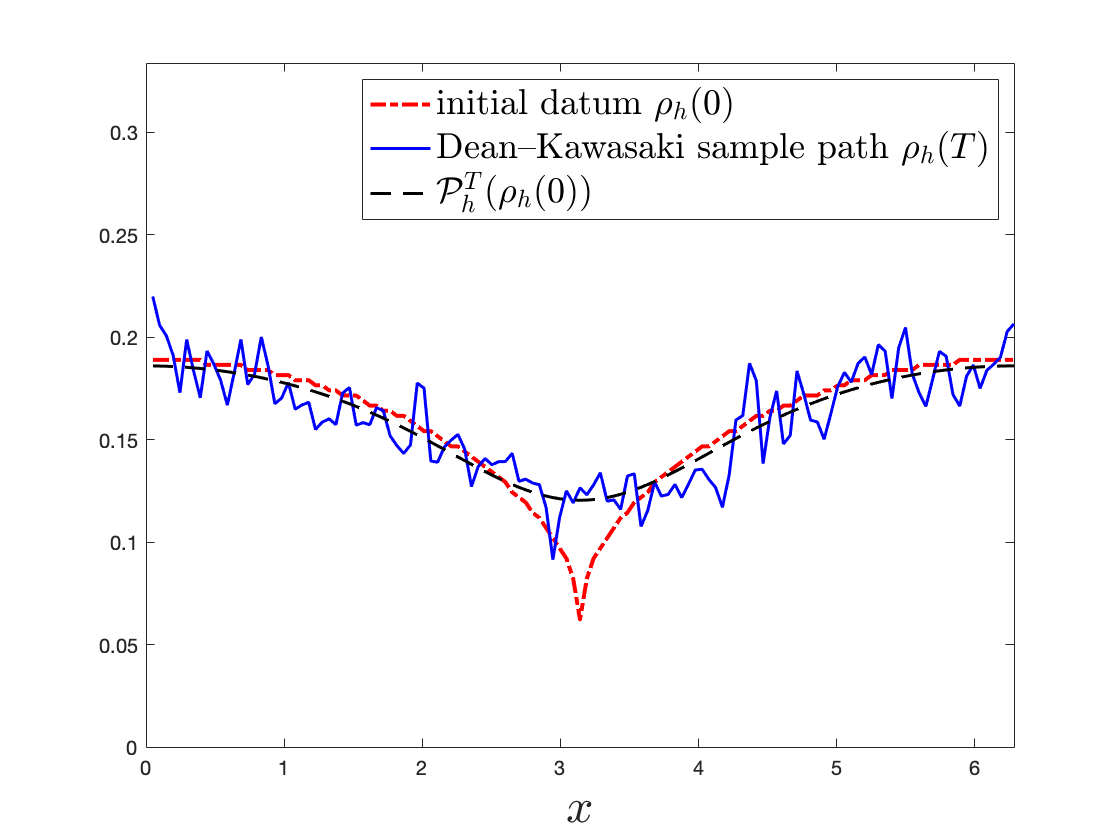}
\\
\includegraphics[scale=0.45]{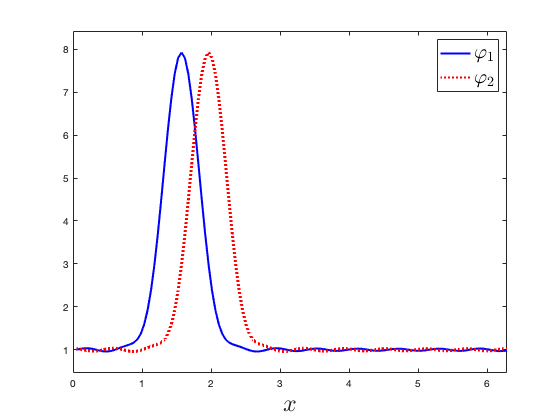}
\caption{\label{FigurePlotSolutionAndTestFunctions2}\emph{Top}: A plot of the initial datum $\rho_0(x):=1/2+|\sin(\tfrac{x-\pi}{2})|^{1/2}$ (dashed red line), its deterministic evolution by the heat equation at time $T_1:=0.4$ (dashed black line), and a sample path from the Dean--Kawasaki equation at time $T_1:=0.4$ for $N:=8137$ particles (blue solid line). \emph{Bottom}: The test functions $\varphi_1$, $\varphi_2$ used for the moment computations (blue solid line, red dotted line).}
\end{figure}

\begin{figure}
\includegraphics[scale=0.14]{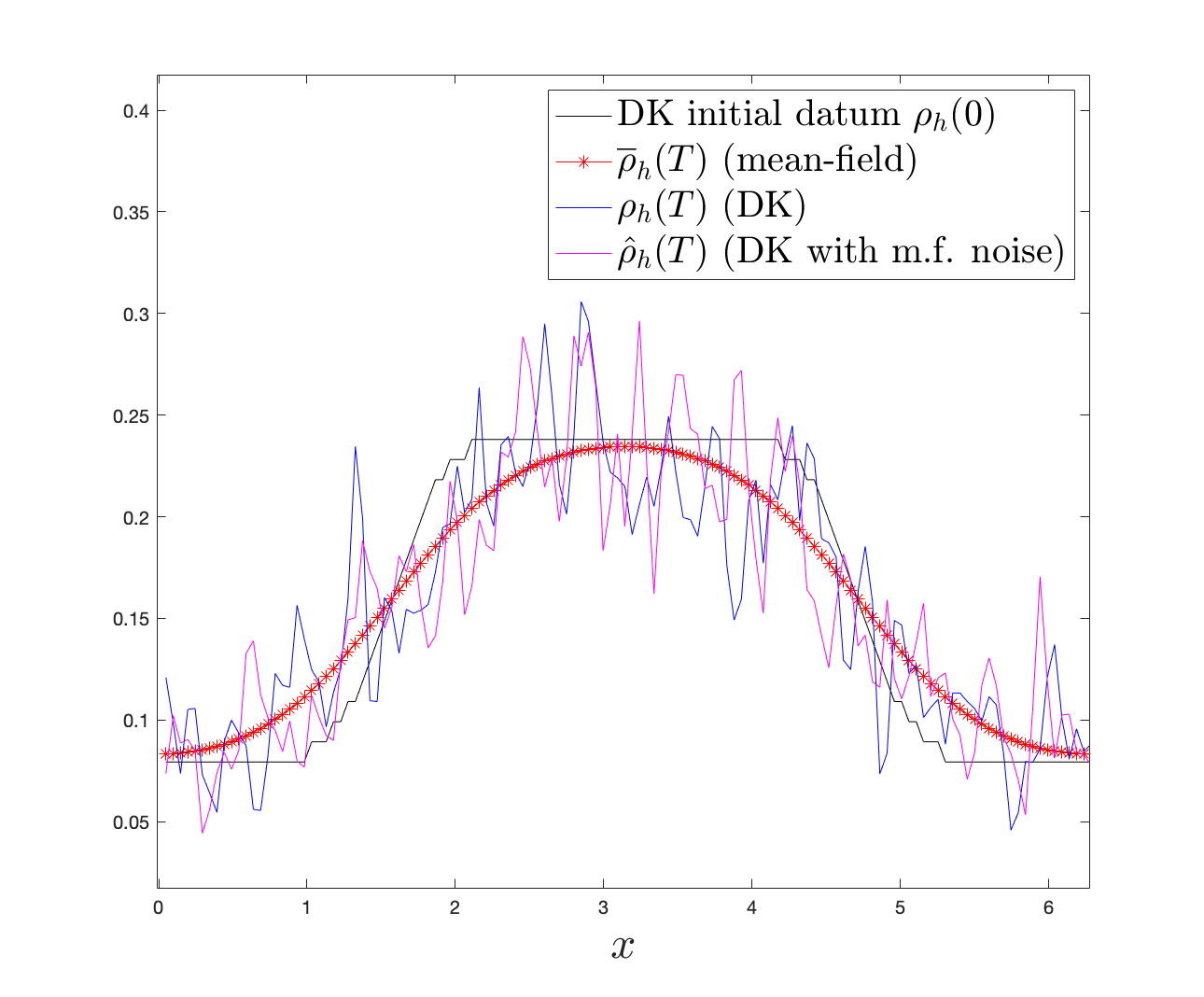}
\includegraphics[scale=0.16]{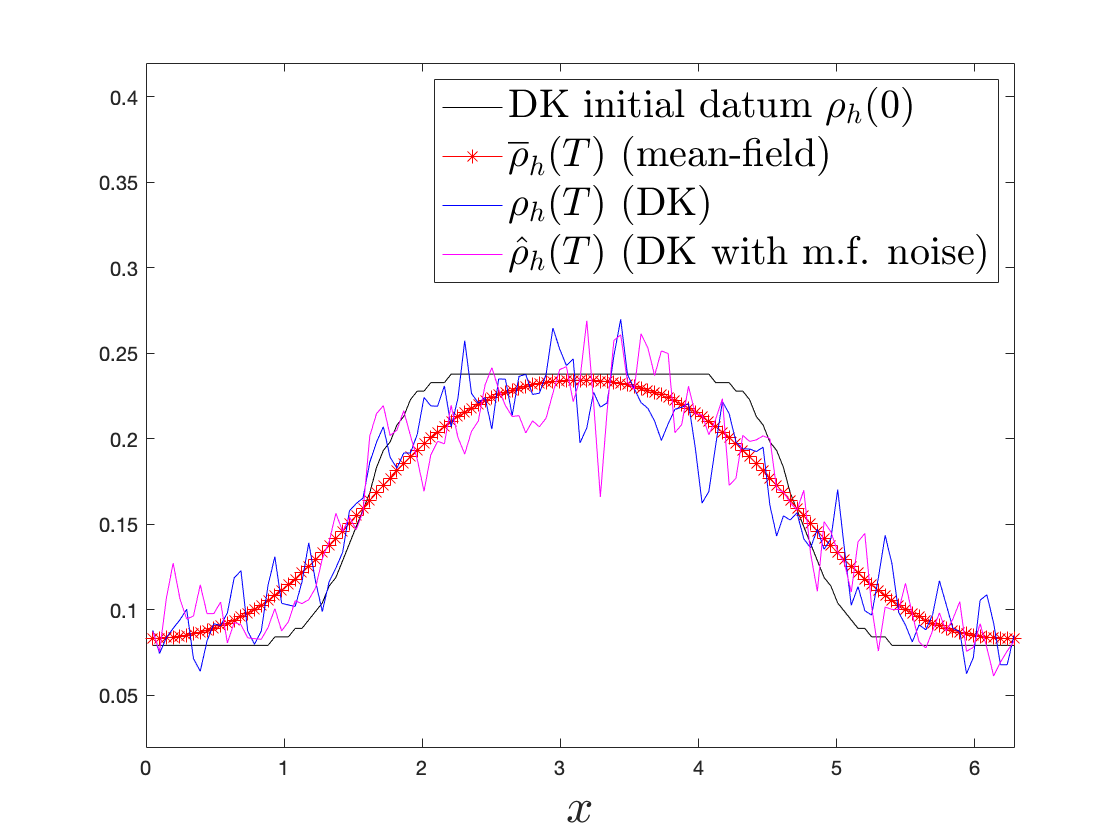}
\\
\includegraphics[scale=0.45]{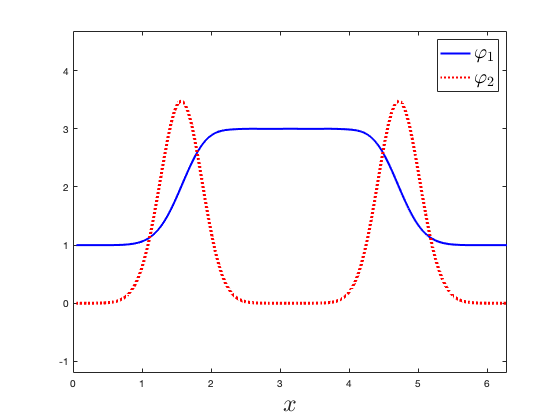}
\caption{\label{FigurePlotSolutionAndTestFunctions3}\emph{Top Left}: A plot of the initial datum $\rho_0(x):=3-2e^{-\sin^8(x/2)/0.03}$ (black solid line), its deterministic evolution by the heat equation at time $T_1:=0.4$ (dashed red line), a sample path from the Dean--Kawasaki equation at time $T_1:=0.4$ for $N:=2011$ particles (blue solid line), and a sample path from the linearised Dean--Kawasaki equation at time $T_1:=0.4$ (pink solid line). \emph{Top Right}: same as \emph{Top Left}, but with $N=4096$. \emph{Bottom}: The test functions $\varphi_1$, $\varphi_2$ used for the moment computations (blue solid line, red dotted line). More specifically, $\varphi_1=\rho_0(x)$ while $\varphi_2(\cdot)\approx |\nabla\varphi_1(\cdot,T/4)|^2$.}
\end{figure}

Using a Monte-Carlo approach with $M\gg 1$ realizations, we next computed the centered stochastic moments
\begin{align*}
M^{DK}_{j_1,j_2} := &
\mean{
\left(\rho_h(T_1) - \mathbb{E}[\rho_h(T_1)], \mathcal{I}_h\varphi_{1}\right)_h^{j_1}
\left(\rho_h(T_2) - \mathbb{E}[\rho_h(T_2)], \mathcal{I}_h\varphi_{2}\right)_h^{j_2}
},
\end{align*}
for test functions $\varphi_1$, $\varphi_2$, times $T_1$, $T_2$, and integer exponents $j_1,j_2$ specified below. We then compared these stochastic moments to the corresponding centered stochastic moments of the empirical density $\mu^N$
\begin{align*}
M^{Brownian}_{j_1,j_2} := &
\mean{
\langle \mu_{T_1}^N-\mathbb{E}[\mu_{T_1}^N], \varphi_1\rangle^{j_1}
\langle \mu_{T_2}^N-\mathbb{E}[\mu_{T_2}^N], \varphi_2\rangle^{j_2}
},
\end{align*}
the latter being also computed by a Monte Carlo approximation with $M$ realizations.

We have performed various simulations in order to assess the convergence of the moments with respect to $h$, $N$, and to compare the discretisations to the linearised Dean--Kawasaki model \eqref{b:3} and to the Dean--Kawasaki model \eqref{DeanKawasaki}.

\subsection{Moment error decay (with respect to $h$)}
For two different choices of initial data $\rho_0(x)$, test functions $\varphi_i(x)$, and times $T_i$, the resulting errors 
$$
|M_{j_1,j_2}^{DK}-M_{j_1,j_2}^{Brownian}|
$$ 
have been plotted in Figure~\ref{FigureConvergenceInH} as a function of the discretisation parameter $h$. We clearly observe a convergence rate $O(h^2)$ for the accuracy of the computed moments. 
\begin{figure}[ht!] 
\includegraphics[scale=0.23]{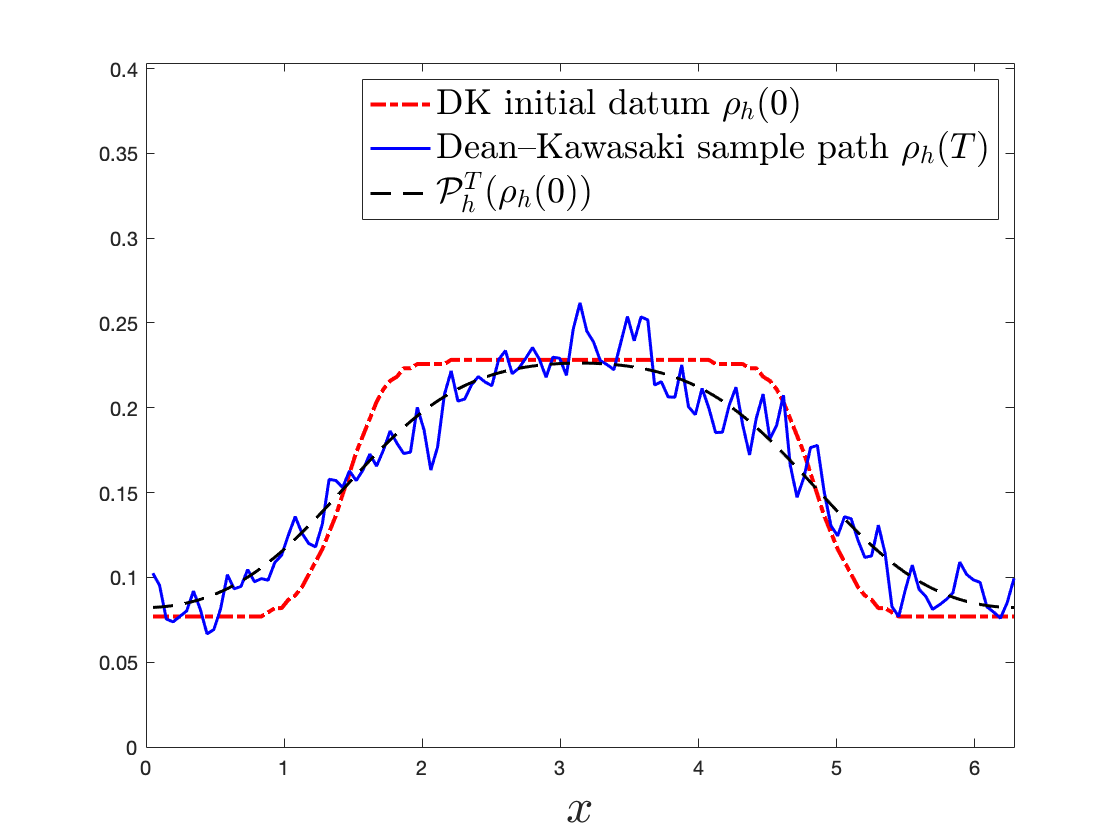}
\\
\includegraphics[scale=0.45]{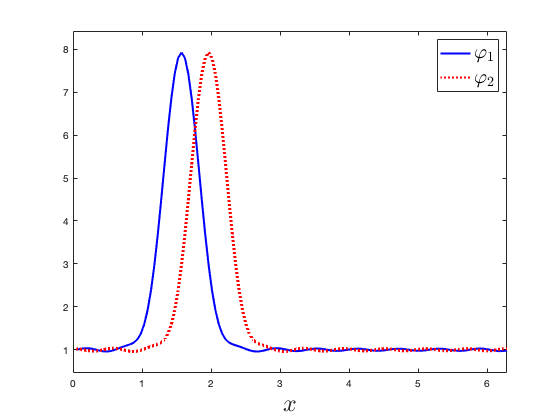}
\caption{\label{FigurePlotSolutionAndTestFunctions1}\emph{Top}: A plot of the initial datum $\rho_0(x):=3-2e^{-\sin^6(x/2)/0.05}$ (dashed red line), its deterministic evolution by the heat equation at time $T_1:=0.4$ (dashed black line), and a sample path from the Dean--Kawasaki equation at time $T_1:=0.4$ for $N:=8211$ particles (blue solid line). \emph{Bottom}: The test functions $\varphi_1$, $\varphi_2$ used for the moment computations (blue solid line, red dotted line).}
\end{figure}

\begin{figure}[ht!] 
\includegraphics[scale=0.230]{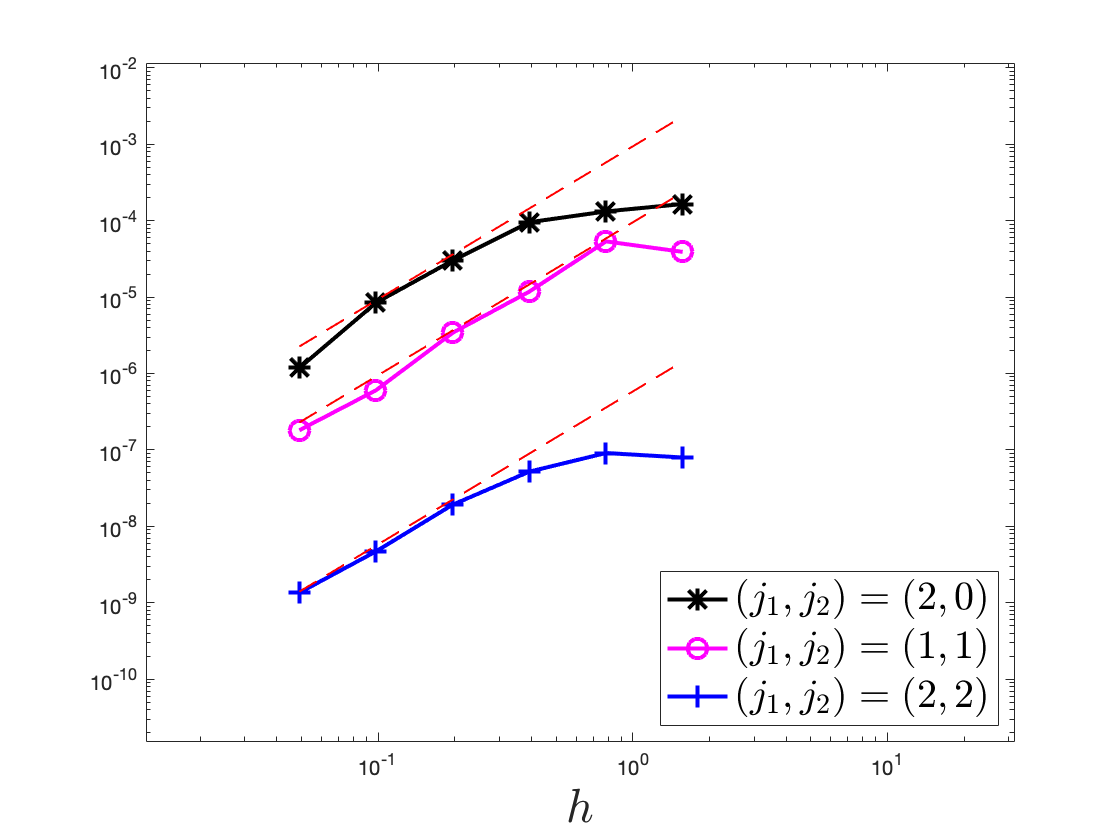}
\\
\includegraphics[scale=0.230]{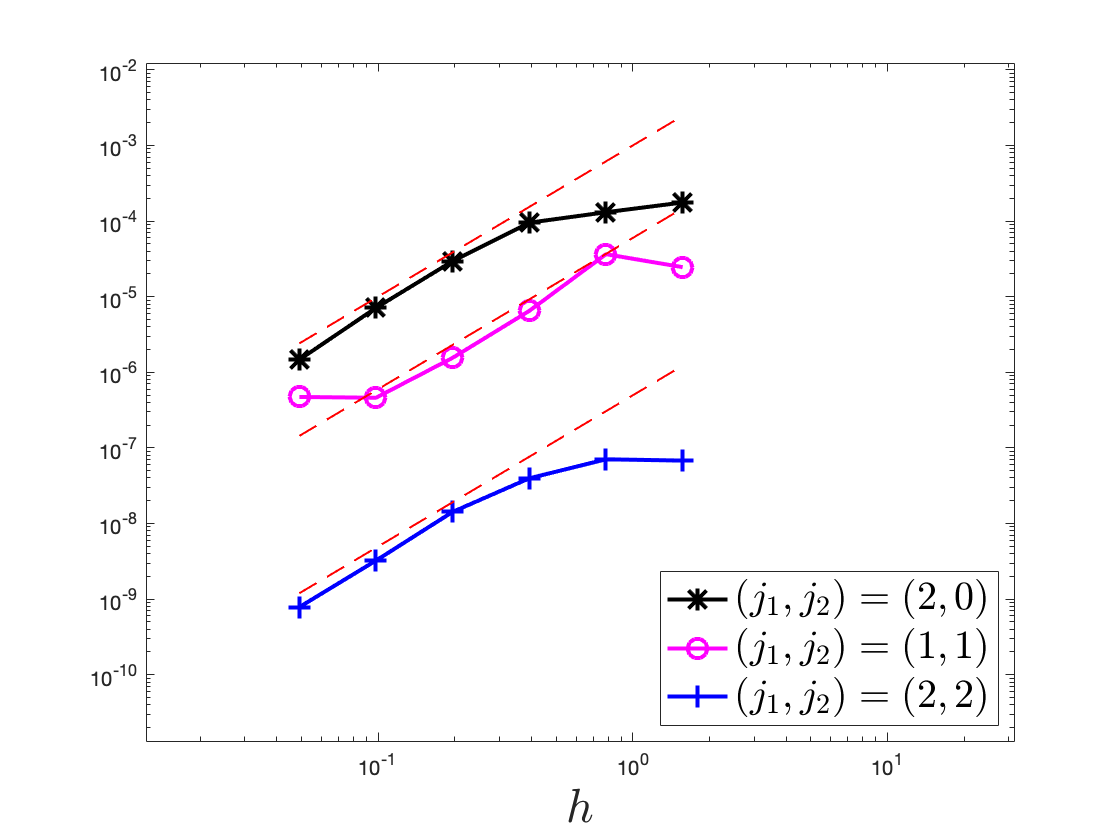}
\caption{\label{FigureConvergenceInH}A log-log plot of the error $|M_{j_1,j_2}^{DK}-M_{j_1,j_2}^{Brownian}|$ in the numerical examples illustrated in Figure~\ref{FigurePlotSolutionAndTestFunctions1} (\emph{top}, with $T_1:=0.4$, $T_2:=0.32$, and particle number $N=8211$) respectively for the numerical examples illustrated in Figure~\ref{FigurePlotSolutionAndTestFunctions2} (\emph{bottom}, with $T_1:=0.4$, $T_2:=0.32$, and particle number $N=524291$). It is clearly visible that (after an initial preasymptotic region) a second-order convergence rate $O(h^2)$ is achieved for all computed moments.}
\end{figure}

\subsection{Moment error decay (with respect to $N$)}
In Figure~\ref{FigureConvergenceInN}, we have plotted the error $|M_{j_1,j_2}^{DK}-M_{j_1,j_2}^{Brownian}|$ as a function of the particle number $N$. We observe that the absolute error decays with the same rate $N^{-(j_1+j_2)/2}$ as the centered moments $M_{j_1,j_2}^{Brownian}$, i.\,e.\ our relative error is basically independent of the particle number $N$ and only depends on the grid size $h$.

\subsection{Comparison with linearised Dean--Kawasaki model \eqref{b:3}}

For the same choice of initial data $\rho_0(x)$, test functions $\varphi_i(x)$, and times $T_i$, and two different choices of $N$, 
we investigate the difference of performance between the time-discretised Dean--Kawasaki model \eqref{eq:4000}-\eqref{eq:4001} and the equivalent scheme associated with the linearised Dean--Kawasaki model \eqref{b:3} (whose discretisation is obtained as a straightforward adaptation of \eqref{eq:4000}-\eqref{eq:4001}). 

More precisely, we have plotted both 
$$
|M_{j_1,j_2}^{DK}-M_{j_1,j_2}^{Brownian}|
$$ 
and  
$$
|M_{j_1,j_2}^{DK,linearised}-M_{j_1,j_2}^{Brownian}|,
$$ 
where $M_{j_1,j_2}^{DK,linearised}$ is the natural counterpart to $M_{j_1,j_2}^{DK}$,
in Figure~\ref{LinearisedCase3_4} as a function of the discretisation parameter $h$. 

We observe that the two models show the same behaviour for the second moment associated with the exponents $(j_1,j_2)=(2,0)$. This is expected, as both models share the same quadratic variation structure of the noise (more explicitly, one can readapt Lemma \ref{thm:1} to the linearised case). On the contrary, the nonlinear model visibly outperforms the linearised model for the higher moment associated with $(j_1,j_2)=(2,1)$. The reason for this is that one can not readapt Proposition \ref{thm:2} to the linearised case, as doing so would result in lower order moments comprising both the Dean--Kawasaki solution and its mean-field limit, thus breaking the very recursive structure of the Proposition.

We have chosen a relatively low number of particles $N$ a particular couple of test functions (with $\varphi_2$ approximately matching the quadratic variation associated with the second test function after some time, i.e., $\varphi_2\approx \nabla|\varphi_1(T/4)|^2$, thus giving non-trivial correlation between $\varphi_1$ and $\varphi_2$) in order to make the difference between the two models more pronounced. Such difference is not completely clear cut though, as one can see for the lowest values of $h$ in the bottom figure. This behaviour is likely caused by:
\begin{itemize}[leftmargin=0.9 cm]
\item the reduced accuracy of the BDF2 integration method for low $h$;
\item in the case of Figure~\ref{LinearisedCase3_4} (\emph{Bottom}), an accuracy saturation.
\end{itemize}

\begin{figure}
\includegraphics[scale=0.2]{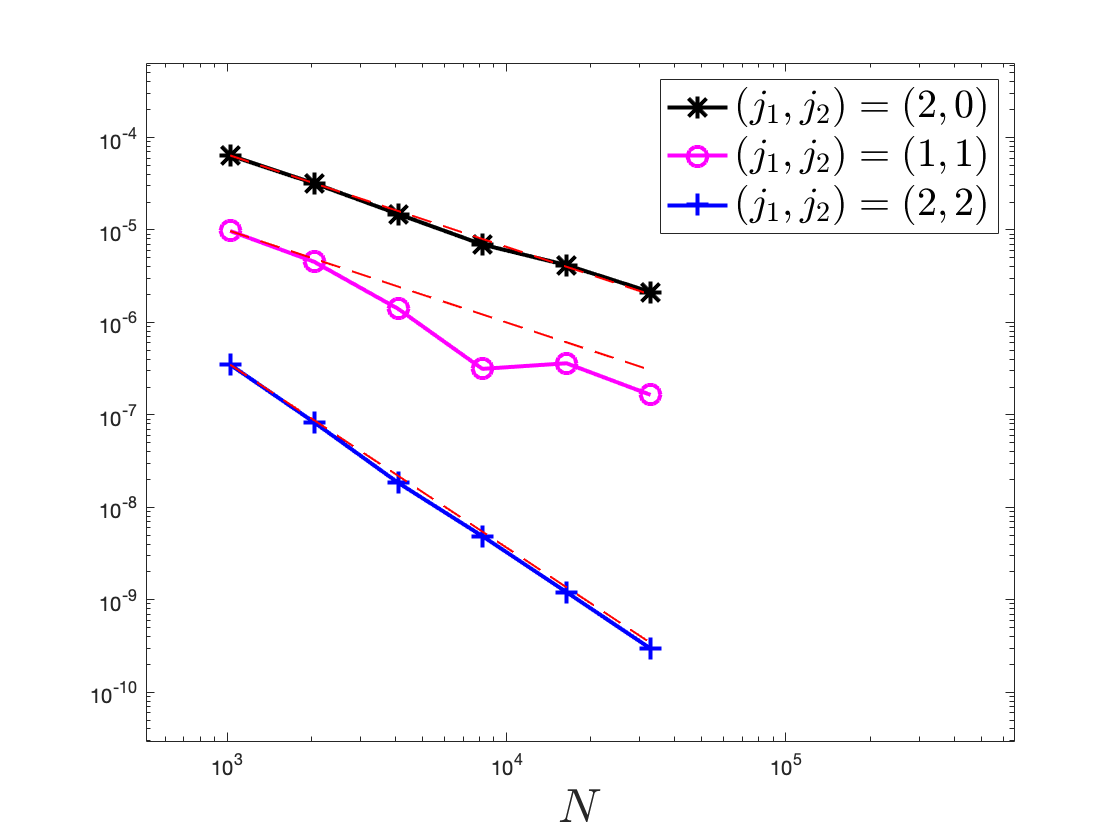}
\caption{\label{FigureConvergenceInN} A log-log plot of the error $|M_{j_1,j_2}^{DK}-M_{j_1,j_2}^{Brownian}|$ for the numerical examples in Figure~\ref{FigurePlotSolutionAndTestFunctions1} for varying values of $N$ (with $T_1:=0.4$, $T_2:=0.32$, and $h=0.098175$). Note that the relative error in the computation of the moments $M_{j_1,j_2}^{Brownian}$ that is achieved by the discretized Dean--Kawasaki equation is basically independent of the particle number $N$: The errors decay essentially uniformly according to the rate $N^{-\frac{j_1+j_2}{2}}$, which coincides with the rate of decay of the moments $M_{j_1,j_2}^{Brownian}$.}
\end{figure}

\begin{figure}[ht!] 
\includegraphics[scale=0.23]{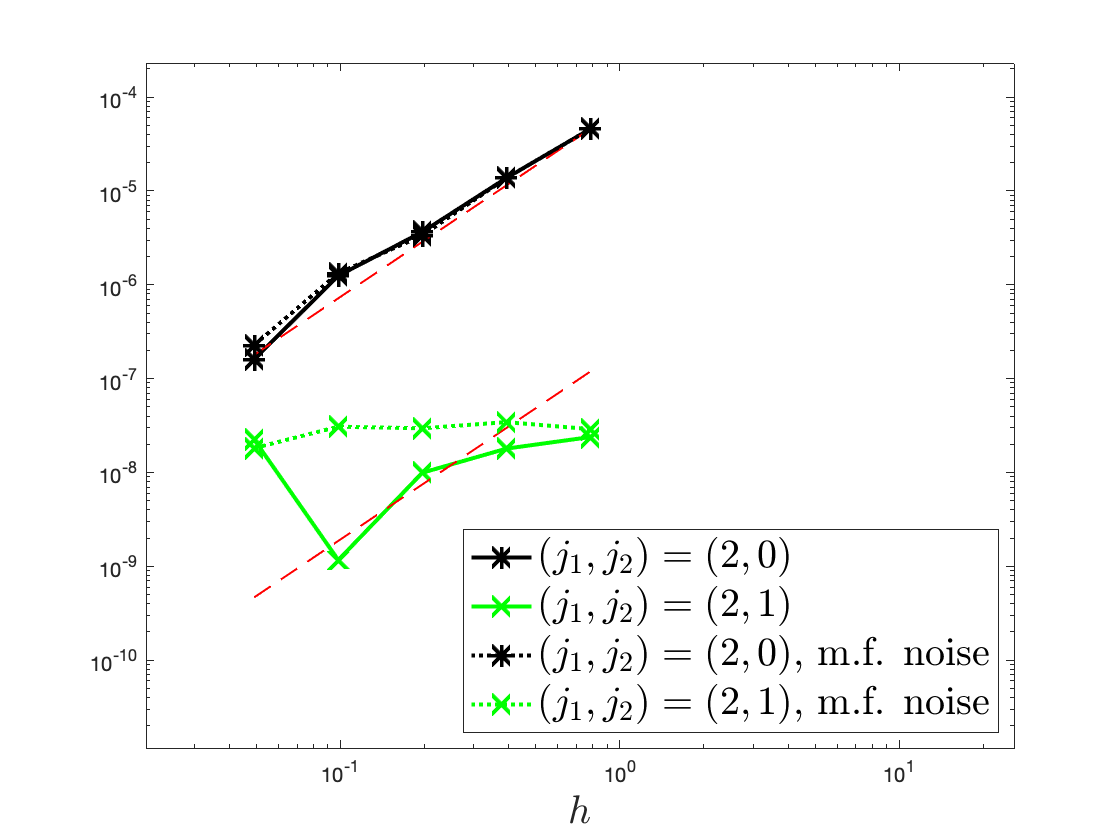}
\\
\includegraphics[scale=0.22]{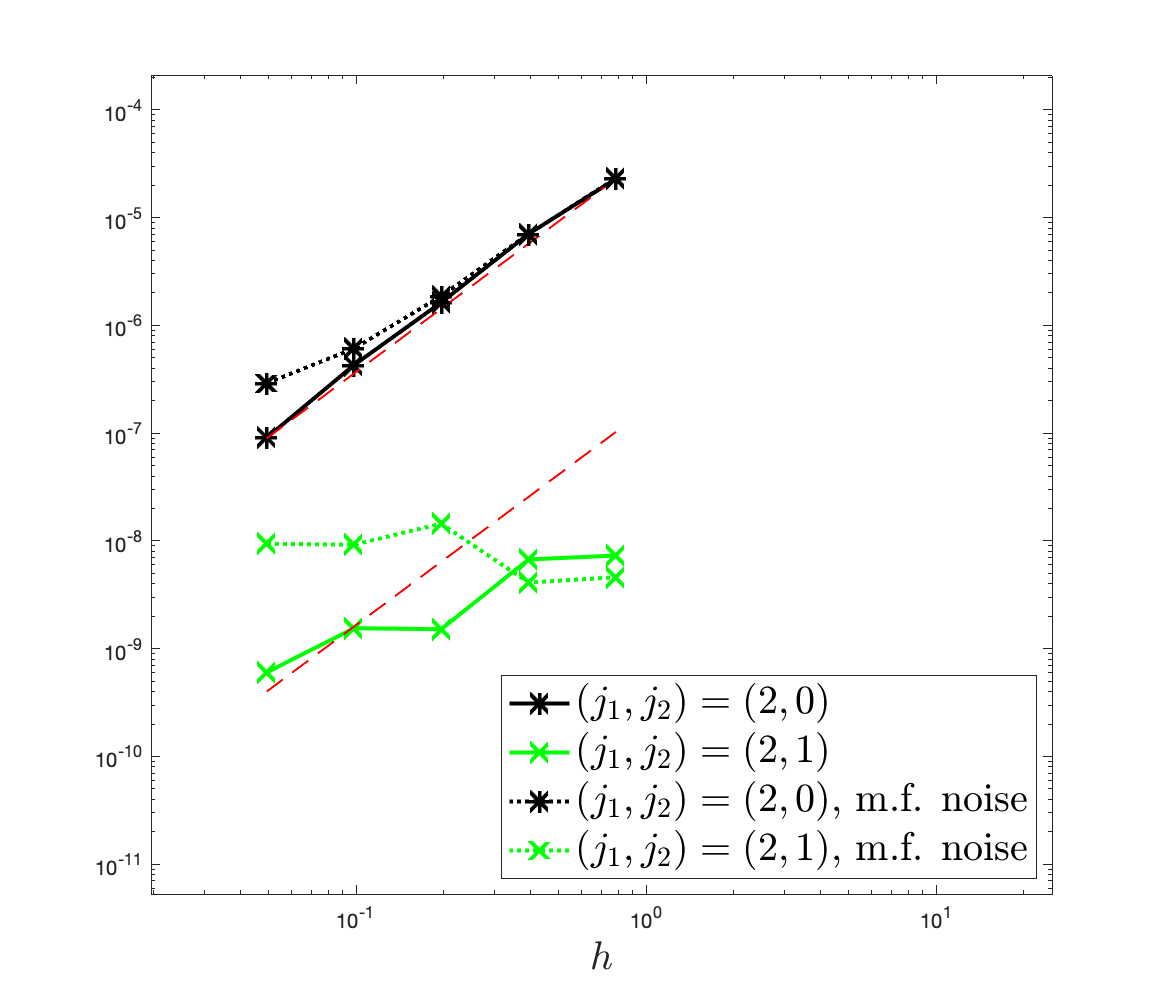}
\caption{\label{LinearisedCase3_4} A log-log plot comparing the error $|M_{j_1,j_2}^{DK}-M_{j_1,j_2}^{Brownian}|$ for the Dean--Kawasaki model (continuous lines) and the linearised Dean--Kawasaki model (dotted lines), in the context of the numerical examples illustrated in Figure~\ref{FigurePlotSolutionAndTestFunctions3} for varying values of $h$ (with $T_1:=0.4$, $T_2:=0.2$, and either $N=2011$ (\emph{Top}) or $N=4096$ (\emph{Bottom})).
We observe that the discretised Dean--Kawasaki model outperforms -- to a good extent -- the linearised version for the moment associated with $(j_1,j_2)=(2,1)$. }
\end{figure}

\appendix

\section{Standard estimates for finite difference discretisation}\label{app:a}

\subsection{Error bounds for continuous and discretised heat flows}\label{ss:approx}

In order to prove the following lemma, we introduce a minimal amount of tools related to Fourier analysis for functions belonging to $[L^2(G_{h,d})]^m$. This is an adaptation of the contents of \cite[Section 2.3]{hackbusch1981}.
Set 
$
I_h:=h^{-1}G_{h,d}=\{-L/2,-L/2+1,\dots,L/2-1,L/2-1\}^d.
$ 
The discrete Fourier transform of $\f{v}_h\in[L^2(G_{h,d})]^m$ is the periodic function 
\begin{align*}
\hat{\f{v}}(\f{\xi}):=h^d\sum_{\f{x}\in G_{h,d}}{\f{v}_h(\f{x})e^{-i\f{x}\cdot\f{\xi}}},\qquad \f{\xi}\in I_h.
\end{align*}
Furthermore, the function $\f{v}_h$ may be reconstructed as
\begin{align*}
\f{v}(\f{x})=\sum_{\f{\xi}\in I_h}{\hat{\f{v}}(\f{\xi})e^{i\f{x}\cdot\f{\xi}}},\qquad \f{x}\in G_{h,d}.
\end{align*}

\begin{lemma}\label{lem:101}
Let $\phi_1$ (respectively, $\phi_2$) be the solution to \eqref{eq:28} with final datum $\phi_1^T=\varphi_1$ (respectively, with final datum $\phi_2^T=\varphi_2$), for some $\varphi_1,\varphi_2\in C^{p+2+\Theta}$, where $\Theta$ is given in \eqref{b:7}. Let $\phi_{1,h}$ (respectively, $\phi_{2,h}$) be the solution of \eqref{eq:29} with final datum $\phi_{1,h}^T=\mathcal{I}_h\varphi_1$ (respectively, with final datum $\phi_{2,h}^T=\mathcal{I}_h\varphi_2$). Assume the validity of Assumption \ref{ass:1}. Then, for $t\leq T$, we have
\begin{align}
\|\phi_i^t-\phi_{i,h}^t\|_h & \leq C\|\varphi_i\|_{C^{p+1}}h^{p+1},\qquad i=1,2,\label{eq:325a}\\
\|\nabla\phi_i^t-\nabla_{h}\phi_{i,h}^t\| & \leq C\|\varphi_i\|_{C^{p+2}}h^{p+1},\qquad i=1,2,\label{eq:325b}\\
\|\nabla_h\phi_{1,h}^t\cdot\nabla_h\phi_{2,h}^t-\nabla\phi_{1}^t\cdot\nabla\phi_{2}^t\| & \leq C\|\varphi_1\|_{C^{p+2+\Theta}}\|\varphi_2\|_{C^{p+2+\Theta}}h^{p+1},\label{eq:49}
\end{align} 
where $C$ is independent of $T$.
\end{lemma}
\begin{proof}
We recall the relation $\mathbb{N}\ni L=2\pi/h$ and definition $I_h=h^{-1}G_{h,d}=\{-L/2,-L/2+1,\dots,L/2-1,L/2-1\}^d$. It is easy to use the continuous and discrete backwards heat equations \eqref{eq:28}--\eqref{eq:29} to deduce that the Fourier coefficients of $\phi_i^t$ and $\phi_{i,h}^t$, $i\in\{1,2\}$, are
\begin{align}
\hat{\phi}_i^t(\f{\xi})  & = \frac{1}{\pi^{d/2}}\int_{\mathbb{T}^d}{\phi_i^t(\f{y})e^{-i\f{y}\cdot\f{\xi}}\m \f{y}} = e^{-(|\f{\xi}|^2/2)(T-t)}\hat{\varphi}_{i}(\f{\xi}), \qquad \f{\xi}\in \mathbb{Z}^d,\label{e:2a}\\
\hat{\phi}_{i,h}^t(\f{\xi}) & = h^d\sum_{\f{x}\in G_{h,d}}{\phi_{i,h}^t(\f{x})e^{-i\f{x}\cdot\f{\xi}}} = e^{-P(h,\f{\xi})(T-t)}\widehat{\mathcal{I}_h\varphi_i}({\f{\xi}}),\qquad \f{\xi}\in I_h\label{e:2b},
\end{align}
for some functional $P(h,\f{\xi})$. As the discrete Laplacian $\Delta_h$ is a $(p+1)$-th order approximation of the true Laplacian with order $p+1$, it is easy to see that 
\begin{align}\label{eq:811}
\left\|\f{\xi}|^2/2-P(h,\f{\xi})\right| \leq |\f{\xi}|^{p+3}h^{p+1}.
\end{align}
Furthermore, since $\Delta_h$ is a symmetric finite difference operator, it is easy to see that $P(h,d)$ is nonnegative. This fact, together with the convexity of the exponential function (which in turn implies the monotonicity of the ratio $(e^{x}-e^{y})/(x-y)$ in either one of the two variables, provided the other one is kept fixed) gives
\begin{align*}
\frac{\left|e^{-(|\f{\xi}|^2/2)(T-t)}-e^{-P(h,\f{\xi})(T-t)}\right|}{\left|\left(-(|\f{\xi}|^2/2)+P(h,\f{\xi})\right)(T-t)\right|} & \leq \frac{\left|e^{-(|\f{\xi}|^2/2)(T-t)}-1\right|}{(|\f{\xi}|^2/2)(T-t)}\leq \frac{1}{(|\f{\xi}|^2/2)(T-t)},
\end{align*}
and therefore
\begin{align}\label{eq:324}
&\left|e^{-(|\f{\xi}|^2/2)(T-t)}-e^{-P(h,\f{\xi})(T-t)}\right| \leq \frac{\left|\left(-\frac{|\f{\xi}|^2}{2}+P(h,\f{\xi})\right)(T-t)\right|}{(|\f{\xi}|^2/2)(T-t)} \stackrel{\eqref{eq:811}}{\leq }C|\f{\xi}|^{p+1}h^{p+1}.
\end{align}
We can deduce that the discrete Fourier expansion of $\mathcal{I}_h\phi_i^t$ from \eqref{e:2a} is
\begin{align}
\mathcal{I}_h\phi_i^t(\f{x})=\sum_{\f{\xi}\in I_h}{\left(\sum_{\f{z}\in\mathbb{Z}^d}{\hat{\phi}_i^t(\f{\xi}+L\f{z})}\right)e^{i\f{x}\cdot\f{\xi}}}=\sum_{\f{\xi}\in I_h}{\widehat{\mathcal{I}_h\phi_i^t}(\f{\xi})e^{i\f{x}\cdot\f{\xi}}},
\end{align}
where we have also used  the fact that $\mathbb{Z}^d = I_h + L\mathbb{Z}^d$.
We deduce
\begin{align*}
\|\mathcal{I}_h\phi_i^t-\phi_{i,h}^t\|^2_h &  = \sum_{\f{\xi}\in I_h}{\left| \hat{\phi}_{i,h}^t(\f{\xi}) - \widehat{\mathcal{I}_h\phi_i^t}(\f{\xi})\right|^2}\nonumber\\
& \stackrel{\mathclap{\eqref{e:2a}\eqref{e:2b}}}{\leq}\quad C\sum_{\f{\xi}\in I_h}{\left|e^{-(|\f{\xi}|^2/2)(T-t)}\left(\sum_{\f{z}\in\mathbb{Z}^d}{\hat{\varphi}_i(\f{\xi}+L\f{z})}\right)-e^{-P(h,\f{\xi})(T-t)}\widehat{\phi_{i,h}^T}({\f{\xi}})\right|^2}\nonumber\\
& \quad + C\sum_{\f{\xi}\in I_h}{\left|\sum_{\f{z}\in\mathbb{Z}^d\setminus \f{0}}{\left(e^{-(|\f{\xi}|^2/2)(T-t)}-e^{-(|\f{\xi}+L\f{z}|^2/2)(T-t)}\right)\hat{\varphi}_i(\f{\xi}+L\f{z})}\right|^2}.
\end{align*}
Since $\mathcal{I}_h\varphi_i=\phi_{i,h}^T$, we carry on in and write
\begin{align}\label{e:911}
 \|\phi_i^t-\phi_{i,h}^t\|^2_h & \leq C\sum_{\f{\xi}\in I_h}{\left|e^{-(|\f{\xi}|^2/2)(T-t)}-e^{P(h,\f{\xi})(T-t)}\right|^2|\widehat{\mathcal{I}_h\varphi}_{i}(\f{\xi})|^2} \nonumber\\
& \quad + C\sum_{\f{\xi}\in I_h}{\left|\sum_{\f{z}\in\mathbb{Z}^d\setminus \f{0}}{\hat{\varphi}_{i}(\f{\xi}+L\f{z})}\right|^2}\nonumber\\
& \stackrel{\mathclap{\eqref{eq:324}}}{\leq} Ch^{2(p+1)} \sum_{\f{\xi}\in I_h}{(1+|\f{\xi}|^2)^{p+1}|\widehat{\mathcal{I}_h\varphi}_{i}(\f{\xi})|^2} \nonumber\\
& \quad + C\sum_{\f{\xi}\in I_h}{\left(\sum_{\f{z}\in\mathbb{Z}^d\setminus \f{0}}{|\hat{\varphi}_{i}(\f{\xi}+L\f{z})|^2(1+|\f{\xi}+L\f{z}|^2)^{p+1}}\right)}\nonumber\\
& \quad \quad \quad \quad \times\left(\sum_{\f{z}\in\mathbb{Z}^d\setminus \f{0}}{(1+|\f{\xi}+L\f{z}|^2)^{-(p+1)}}\right).
\end{align}
We estimate 
\begin{align*}
\sum_{\f{z}\in\mathbb{Z}^d\setminus \f{0}}{(1+|\f{\xi}+L\f{z}|^2)^{-(p+1)}} & \leq \sum_{\f{z}\in\mathbb{Z}^d\setminus \f{0}}{(2L^2|\f{z}|^2)^{-(p+1)}} \leq Ch^{2(p+1)}\sum_{\f{z}\in\mathbb{Z}^d\setminus \f{0}}{|\f{z}|^{-2(p+1)}}\\
& \leq Ch^{2(p+1)},
\end{align*}
where the final step is justified by 
\begin{align}\label{e:913}
\int_{\mathbb{R}^d\setminus \{\f{z}\in\mathbb{R}^d\colon|\f{z}|\geq 1\}}{\!\!|\f{z}|^{-2(p+1)}\m \f{z}}=C(d)\int_{1}^{\infty}{|r|^{-2(p+1)+d-1}\m r}\propto \left.r^{-2(p+1)+d}\right|^{\infty}_{1}< \infty,
\end{align}
which is valid since $2(p+1)>d$, as $d\leq 3$ and $p\geq 1$.
We continue in \eqref{e:911} as 
\begin{align*}
\|\phi_i^t-\phi_{i,h}^t\|^2_h & \leq Ch^{2(p+1)} \sum_{\f{\xi}\in I_h}{(1+|\f{\xi}|^2)^{p+1}|\widehat{\mathcal{I}_h\varphi_i}(\f{\xi})|^2} \nonumber\\
& \quad + Ch^{2(p+1)}\sum_{\f{\xi}\in \mathbb{Z}^d}{|\hat{\varphi}_{i}(\f{\xi})|^2(1+|\f{\xi}|^2)^{p+1}}\nonumber\\
& \leq Ch^{2(p+1)}\{\|\mathcal{I}_h\varphi_i\|_{p+1,h}+\|\varphi_i\|_{C^{p+1}}\} \leq Ch^{2(p+1)}\|\varphi_i\|_{C^{p+1}},
\end{align*}
and \eqref{eq:325a} is proved.

We adapt the arguments carried out so far and write 
\begin{align}\label{e:914}
& \|\nabla\phi_i^t-\nabla_h\phi_{i,h}^t\|^2 \nonumber\\ 
& \leq  C\sum_{\f{\xi}\in I_h}{\left| \widehat{\nabla_h\mathcal{I}_h\phi_i^t}(\f{\xi}) - \widehat{\nabla\phi^t}_{\!\!i,h}(\f{\xi})\right|^2}\nonumber\\
& \stackrel{\mathclap{\eqref{e:2a}\eqref{e:2b}}}{\leq} \quad C\sum_{\f{\xi}\in I_h}{\left|e^{-(|\f{\xi}|^2/2)(T-t)}\left(\sum_{\f{z}\in\mathbb{Z}^d}{\widehat{\nabla\varphi}_i(\f{\xi}+L\f{z})}\right)-e^{P(h,\f{\xi})(T-t)}\widehat{\nabla_h\mathcal{I}_h\varphi_i}(\f{\xi})\right|^2}\nonumber\\
& \quad\quad + C\sum_{\f{\xi}\in I_h}{\left|\sum_{\f{z}\in\mathbb{Z}^d\setminus \f{0}}{\left(e^{-(|\f{\xi}|^2/2)((T-t)}-e^{-(|\f{\xi}+L\f{z}|^2/2)(T-t)}\right)\widehat{\nabla\varphi}_{i}(\f{\xi}+L\f{z})}\right|^2},
\end{align}
After simple algebraic rearrangements in \eqref{e:914}, we get
\begin{align}\label{e:912}
&  \|\nabla\phi_i^t-\nabla_h\phi_{i,h}^t\|^2 \nonumber\\
  & \quad\leq C\sum_{\f{\xi}\in I_h}{\left|e^{-(|\f{\xi}|^2/2)(T-t)}-e^{P(h,\f{\xi})(T-t)}\right|^2\left|\widehat{\nabla_h\mathcal{I}_h\varphi_i}(\f{\xi})\right|^2}\nonumber\\
& \quad\quad + C\sum_{\f{\xi}\in I_h}{\left|e^{-(|\f{\xi}|^2/2)(T-t)}\right|^2\left|\left(\sum_{\f{z}\in\mathbb{Z}^d}{\widehat{\nabla\varphi}_i(\f{\xi}+L\f{z})}\right)-\widehat{\nabla_h\mathcal{I}_h\varphi_i}(\f{\xi})\right|^2}\nonumber\\
& \quad\quad + C\sum_{\f{\xi}\in I_h}{\left|\sum_{\f{z}\in\mathbb{Z}^d\setminus \f{0}}{\left|\widehat{\nabla\varphi}_{i}(\f{\xi}+L\f{z})\right|}\right|^2}=: T_1+T_2+T_3.
\end{align}
The term $T_1$ is estimated using \eqref{eq:324}, giving $T_1\leq Ch^{2(p+1)}\|\varphi_i\|^2_{C^{p+2}}$. The term $T_2$ is estimated by relying on \eqref{e:2}, giving $T_2\leq Ch^{2(p+1)}\|\varphi_i\|^2_{C^{p+1}}$. As for $T_3$, we rely on the fact that $\widehat{\nabla\varphi}_{i}(\f{\xi}+L\f{z})=-i\f{\xi}\hat{\varphi}_{i}(\f{\xi}+L\f{z})$ and write
\begin{align*}
T_3 & \leq C\sum_{\f{\xi}\in I_h}{\left(\sum_{\f{z}\in\mathbb{Z}^d\setminus \f{0}}{|\hat{\varphi}_{i}(\f{\xi}+L\f{z})|^2(1+|\f{\xi}+L\f{z}|^2)^{p+2}}\right)}\\
& \quad \quad \times \left(\sum_{\f{z}\in\mathbb{Z}^d\setminus \f{0}}{(1+|\f{\xi}+L\f{z}|^2)^{-(p+1)}}\right) \\
& \stackrel{\mathclap{\eqref{e:913}}}{\leq} Ch^{2(p+1)}\sum_{\f{\xi}\in I_h}{\left(\sum_{\f{z}\in\mathbb{Z}^d\setminus \f{0}}{|\hat{\varphi}_{i}(\f{\xi}+L\f{z})|^2(1+|\f{\xi}+L\f{z}|^2)^{p+2}}\right)}\\
& \leq Ch^{2(p+1)}\sum_{\f{\xi}\in\mathbb{Z}^d}{|\hat{\varphi}_{i}(\f{\xi})|^2(1+|\f{\xi}|^2)^{p+2}} \leq Ch^{2(p+1)}\|\varphi_i\|_{C^{p+2}},
\end{align*}
and \eqref{eq:325b} is proved.
To prove \eqref{eq:49}, we write 
\begin{align*}
& \left\|\nabla_h\phi_{1,h}^t\cdot\nabla_h\phi_{2,h}^t-\nabla\phi_1^t\cdot\nabla\phi_2^t\right\|\\
& \leq \left\|\left(\nabla_h\phi_{1,h}^t-\nabla\phi_1^t\right)\cdot\nabla_h\phi_{2,h}^t\right\| + \left\|\left(\nabla_h\phi_{2,h}^t-\nabla\phi_2^t\right)\cdot\nabla\phi_{1}^t\right\| =: T_4+T_5.
\end{align*}
Now let $\phi$ be the solution to \eqref{eq:29} with final datum $\phi^T=\varphi\in C^{1+\Theta}$.  
If \eqref{eq:29} admits a discrete maximum principle, then  
$$
\max_{t\in[0,T]}\{\|\nabla_h\phi_{h}^t\|_{\infty}\} \leq \|\nabla_h\mathcal{I}_h\varphi\|_{\infty} \leq \|\varphi\|_{C^{1}}. 
$$
If \eqref{eq:29} does not admit a discrete maximum principle, then we rely on the Sobolev embedding $H^{s}\subset C^0$, where $s=d/2+1$, and get
$$
\max_{t\in[0,T]}\{\|\nabla_h\phi_{h}^t\|_{\infty}\} \leq \|\varphi\|_{C^{2+d/2}}.
$$ 
and the last two expressions can be summarised as 
\begin{align}\label{5005}
\max_{t\in[0,T]}\{\|\nabla_h\phi_{h}^t\|_{\infty}\} \leq \|\varphi\|_{C^{1+\Theta}}
\end{align}
We focus on $T_4$. It is easy to see that
\begin{align*}
T_4\leq \left\|\nabla_h\phi_{1,h}^t-\nabla\phi_1^t\right\|\|\nabla_h\phi_{2,h}^t\|_{\infty}\stackrel{\eqref{eq:325b}\eqref{5005}}{\leq} C\|\varphi_1\|_{C^{p+2}}\|\varphi_2\|_{C^{1+\Theta}}h^{p+1}.
\end{align*}
The estimate for $T_5$ is even more straightforward, and it reads 
\begin{align*}
T_5\leq \left\|\nabla_h\phi_{2,h}^t-\nabla\phi_2^t\right\|\|\nabla\phi_{1}^t\|_{\infty}\stackrel{\eqref{eq:325b}}{\leq} C\|\varphi_2\|_{C^{p+2}}\|\varphi_1\|_{C^{1}}h^{p+1},
\end{align*}
and \eqref{eq:49} is proved. 
\end{proof}

\subsection{Stretched exponential moment bounds for the Dean--Kawasaki solution and the particle system}\label{ss:mom}

We compute the It\^o differential of the quantities in \eqref{e:25}.
\begin{lemma}[It\^o differential for moments $\mathcal{S}_N^{\boldsymbol{j}}$ (Dean--Kawasaki model) and $\mathcal{T}_N^{\boldsymbol{j}}$ (Brownian particles)]\label{lem:10}
Fix $M\in\mathbb{N}$, a multi-index $\boldsymbol{j}=(j_1,\dots,j_M)$, and a set of test functions $\boldsymbol{\varphi}=(\varphi_1,\dots,\varphi_M)\in\left[C^{2}\right]^M$. For any $(k,l)\in\{1,\dots,M\}$, denote by $\boldsymbol{j}^{kl}$ the vector $\boldsymbol{j}$ with both $j_k$ and $j_l$ decreased by one unit (if $k=l$, then $j_k$ is understood to be reduced by two units). For any $k\in\{1,\dots,M\}$, denote by $\boldsymbol{j}^k$ the vector $\boldsymbol{j}$ with $j_k$ decreased by one unit. Let $\rho_h$ be as given in Definition \ref{def:1}. We recall the following definitions 
\begin{align*}
\mathcal{S}^{\boldsymbol{j}}_N(\mathcal{I}_h\boldsymbol{\varphi},T,t) & := \prod_{m=1}^{M}{\mathcal{S}^{j_m}_N(\mathcal{I}_h\varphi_{m},T,t)}= \prod_{m=1}^{M}{\left\{(\rho_h(t),\phi_{m,h}^t)_h-(\rho_h(0),\phi_{m,h}^0)_h\right\}^{j_m}},\\
 \mathcal{T}^{\boldsymbol{j}}_N(\boldsymbol{\varphi},T,t) & := \prod_{m=1}^{M}{\mathcal{T}^{j_m}_N(\varphi_{m},T,t)} = \prod_{m=1}^{M}{\left\{\langle \mu^N_t,\phi_{m}^t\rangle - \langle \mu^N_0,\phi_{m}^0\rangle \right\}^{j_m}}
\end{align*}
from Subsection \ref{ss:not}. Then 
\begin{subequations}
\begin{align}\label{e:26aa}
& \emph{\m} \mathcal{S}^{\boldsymbol{j}}_N(\mathcal{I}_h\boldsymbol{\varphi},T,t) \nonumber \\
&\quad  = -N^{-1/2}\sum_{m=1}^{m}{j_m \mathcal{S}^{\boldsymbol{j}^m}_N(\mathcal{I}_h\boldsymbol{\varphi},T,t)\sum_{(\f{y},\ell)}{(\mathcal{F}_\rho(t) \f{e}^d_{h,\f{y},\ell},\nabla_h \phi_{i,h}^t)_h\emph{\m} \beta_{\f{y},\ell}}}\nonumber\\
&\quad\quad + N^{-1}\sum_{k,l=1}^{M}{\frac{(j_k-\delta_{kl})j_l}{2}\mathcal{S}^{\boldsymbol{j}^{kl}}_N(\mathcal{I}_h\boldsymbol{\varphi},T,t)(\rho_h^+(t),\nabla_h \phi_{k,h}^t\cdot\nabla_h \phi_{l,h}^t)_h}\emph{\m t},
\intertext{and}
\label{e:26bb}
& \emph{\m} \mathcal{T}^{\boldsymbol{j}}_N(\boldsymbol{\varphi},T,t) \nonumber\\
 & \quad = -\sum_{m=1}^{M}{j_m\mathcal{T}^{\boldsymbol{j}^m}(\boldsymbol{\varphi},T,t)\left[N^{-1}\sum_{k=1}^{N}{\nabla\phi_m^t(\f{w}_k(t))\cdot \emph{\m} \f{w}_k(t)}\right]}\nonumber\\
& \quad\quad + N^{-1}\sum_{k,l=1}^{M}{\frac{(j_k-\delta_{kl})j_l}{2}\mathcal{T}^{\boldsymbol{j}^{kl}}_N(\boldsymbol{\varphi},T,t)\left[N^{-1}\sum_{r=1}^{N}{\nabla\phi_k^t(\f{w}_r(t))\cdot\nabla\phi_l^t(\f{w}_r(t))}\right]\emph{\m} t}.
\end{align}
\end{subequations}
\end{lemma}
\begin{proof} All differentials in this proof are with respect to the variable $t$.  
We prove \eqref{e:26aa} in three steps.

\emph{Step 1: Case $M=|\boldsymbol{j}|_1=1$}. We just need to compute the differential for $(\rho_h(t),\phi_{1,h}^t)_h-(\rho_h(0),\phi_{1,h}^0)_h$. We use an $L^2(G_{h,d})$-expansion and the It\^o formula to deduce
\begin{align}\label{e:30}
& \m \left\{(\rho_h(t),\phi_{1,h}^t)_h-(\rho_h(0),\phi_{1,h}^0)_h\right\}\nonumber\\
& \quad = \m \sum_{\f{x}}{(\rho_h(t),e_{\f{x}})_h(\phi_{1,h}^t,e_{\f{x}})_h} \nonumber\\
& \quad = \sum_{\f{x}}{\m \left[(\rho_h(t),e_{\f{x}})_h\right](\phi_{1,h}^t,e_{\f{x}})_h} + \sum_{\f{x}}{(\rho_h(t),e_{\f{x}})_h\m (\phi_{1,h}^t,e_{\f{x}})_h}\nonumber\\
& \quad\stackrel{\mathclap{\eqref{eq:29}\eqref{e:3}}}{=} \quad \sum_{\f{x}}{\left(\frac{1}{2}\Delta \rho_h(t),e_{\f{x}}\right)_h(\phi_{1,h}^t,e_{\f{x}})_h\m t} + \sum_{\f{x}}{(\rho_h(t),e_{\f{x}})_h \left(-\frac{1}{2}\phi_{1,h}^t,e_{\f{x}}\right)_h\m t}\nonumber\\
& \quad\quad - N^{-1/2}\sum_{(\f{y},\ell)}{(\mathcal{F}_\rho(s) \f{e}^d_{h,\f{y},\ell},\nabla_h \phi_{1,h}^t)\m \beta_{(\f{y},\ell)}}\nonumber\\
&\quad \stackrel{\mathclap{\eqref{e:1}}}{=} - N^{-1/2}\sum_{(\f{y},\ell)}{(\mathcal{F}_\rho(t) \f{e}^d_{h,\f{y},\ell},\nabla_h \phi_{1,h}^t)\m \beta_{(\f{y},\ell)}}.
\end{align}
\emph{Step 2: Case $M=1$, $|\boldsymbol{j}|_1=j_1>1$}. The It\^o formula applied to composition of the function $z\mapsto z^{j}$ with the process $(\rho_h,\phi_{1,h})_h$ gives
\begin{align}\label{e:31}
&\m \left\{(\rho_h(t),\phi_{1,h}^t)_h-(\rho_h(0),\phi_{1,h}^0)_h\right\}^{j_1}\nonumber\\
&\quad \stackrel{\mathclap{\eqref{e:30}}}{=} -N^{-1/2}j_1\mathcal{S}^{j_1-1}(\varphi_1,T,t)\sum_{(\f{y},\ell)}{(\mathcal{F}_\rho(t) \f{e}^d_{h,\f{y},\ell},\nabla_h \phi_{1,h}^t)\m \beta_{(\f{y},\ell)}}\nonumber\\
&\quad \quad + N^{-1}\frac{j_1(j_1-1)}{2}\mathcal{S}^{j_1-2}(\varphi_1,T,t)\sum_{(\f{y},\ell)}{(\mathcal{F}_\rho(t) \f{e}^d_{h,\f{y},\ell},\nabla_h \phi_{1,h}^t)^2\m t}\nonumber\\
&\quad = -N^{-1/2}j_1\mathcal{S}^{j_1-1}(\varphi_1,T,t)\sum_{(\f{y},\ell)}{(\mathcal{F}_\rho(t) \f{e}^d_{h,\f{y},\ell},\nabla_h \phi_{1,h}^t)\m \beta_{(\f{y},\ell)}}\nonumber\\
&\quad \quad + N^{-1}\frac{j_1(j_1-1)}{2}\mathcal{S}^{j_1-2}(\varphi_1,T,t)(\rho^+_h(t),\nabla_h\phi_{1,h}^t\cdot \nabla_h\phi_{1,h}^t)_h\m t.
\end{align}
\emph{Step 3: Inductive step in the index $M$}. Assume the validity of \eqref{e:26aa} for some $M$, some vector of exponents $\boldsymbol{\tilde{j}}$, and some vector of functions $\boldsymbol{\tilde{\varphi}}=(\varphi_1,\dots,\varphi_M)$. For an additional test function $\varphi_{M+1}$ with associated cardinality $j_{M+1}$, define $\boldsymbol{j}:=(\boldsymbol{\tilde{j}},j_{M+1})$ and $\boldsymbol{\varphi}:=(\varphi_1,\dots,\varphi_{M+1})$. We use the It\^o formula for the product of the two real-valued processes 
$\textstyle
\mathcal{S}^{\boldsymbol{\tilde{j}}}_N(\mathcal{I}_h\boldsymbol{\tilde{\varphi}},T,t)=\prod_{m=1}^{M}{\mathcal{S}^{\tilde{j}_m}(\mathcal{I}_h\varphi_m,T,t)}
$ 
and $\mathcal{S}^{j_{M+1}}(\mathcal{I}_h\varphi_{M+1},T,t)$, namely
\begin{align*}
\m \mathcal{S}^{\boldsymbol{j}}(\mathcal{I}_h\boldsymbol{\varphi},T,t) & = \m \left\{\mathcal{S}^{\boldsymbol{\tilde{j}}}_N(\mathcal{I}_h\boldsymbol{\tilde{\varphi}},T,t)\cdot \mathcal{S}^{j_{M+1}}(\varphi_{M+1},T,t)\right\}\\
& = \m \left\{\mathcal{S}^{\boldsymbol{\tilde{j}}}_N(\mathcal{I}_h\boldsymbol{\tilde{\varphi}},T,t)\right\} \mathcal{S}^{j_{M+1}}(\mathcal{I}_h\varphi_{M+1},T,t) \\
& \quad +\mathcal{S}^{\boldsymbol{\tilde{j}}}_N(\mathcal{I}_h\boldsymbol{\tilde{\varphi}},T,t) \m \mathcal{S}^{j_{M+1}}(\mathcal{I}_h\varphi_{M+1},T,t) \\
& \quad + \left\langle \mathcal{S}^{\boldsymbol{\tilde{j}}}_N(\mathcal{I}_h\boldsymbol{\tilde{\varphi}},T,t), \mathcal{S}^{j_{M+1}}(\mathcal{I}_h\varphi_{M+1},T,t) \right\rangle,
\end{align*}
 and take \eqref{e:31} and the inductive hypothesis into account, thus getting
\begin{align*}
& \m \mathcal{S}^{\boldsymbol{j}}(\mathcal{I}_h\boldsymbol{\varphi},T,t) \\
& = -N^{-1/2}\sum_{m=1}^{M}{\tilde{j}_m \mathcal{S}^{\boldsymbol{\tilde{j}}^m}_N(\boldsymbol{\mathcal{I}_h\tilde{\varphi}},T,t)\mathcal{S}^{j_{M+1}}(\mathcal{I}_h\varphi_{M+1},T,t)\sum_{(\f{y},\ell)}{(\mathcal{F}_\rho(t) \f{e}^d_{h,\f{y},\ell},\nabla_h \phi_{m,h}^t)_h\m \beta_{\f{y},\ell}}}\nonumber\\
& \quad + N^{-1}\sum_{k,l=1}^{M}{\frac{(\tilde{j}_k-\delta_{kl})\tilde{j}_l}{2}\mathcal{S}^{\boldsymbol{\tilde{j}}^{kl}}_N(\boldsymbol{\mathcal{I}_h\tilde{\varphi}},T,t)\mathcal{S}^{j_{M+1}}(\mathcal{I}_h\varphi_{M+1},T,t)(\rho_h^+(t),\nabla_h \phi_{k,h}^t\cdot\nabla_h \phi_{l,h}^t)_h}\m t\\
& \quad -N^{-1/2}j_{M+1}\mathcal{S}^{\boldsymbol{\tilde{j}}}_N(\mathcal{I}_h\boldsymbol{\tilde{\varphi}},T,t)\mathcal{S}^{j_{M+1}-1}(\mathcal{I}_h\varphi_{M+1},T,t)\sum_{(\f{y},\ell)}{(\mathcal{F}_\rho(t) \f{e}^d_{h,\f{y},\ell},\nabla_h \phi_{M+1,h}^t)\m \beta_{(\f{y},\ell)}}\nonumber\\
&\quad + N^{-1}\frac{j_{M+1}(j_{M+1}-1)}{2}\mathcal{S}^{\boldsymbol{\tilde{j}}}_N(\mathcal{I}_h\boldsymbol{\tilde{\varphi}},T,t)\mathcal{S}^{j_{M+1}-2}(\mathcal{I}_h\varphi_{M+1},T,t)\\
& \quad \quad \times (\rho^+_h(t),\nabla_h\phi_{M+1,h}^t\cdot \nabla_h\phi_{M+1,h}^t)_h\m t.\\
&\quad + N^{-1}\sum_{m=1}^{M}{\frac{j_{M+1}\tilde{j_m}}{2}\mathcal{S}^{\boldsymbol{\tilde{j}}^m}_N(\mathcal{I}_h\boldsymbol{\tilde{\varphi}},T,t)\mathcal{S}^{j_{M+1}-1}(\mathcal{I}_h\varphi_{M+1},T,t)}\\
& \quad \quad \times(\rho_h^+(t),\nabla_h \phi_{m,h}^t\cdot\nabla_h \phi_{M+1,h}^t)_h\m t,
\end{align*}
which is as prescribed by \eqref{e:26aa}. The proof of \eqref{e:26bb} is analogous, and we omit it.
\end{proof}
\begin{lemma}\label{lem:7} Let $\phi$ solve the heat equation \eqref{eq:28}, with final datum $\phi^T=\varphi$ and $t\leq T$. Let $\rho_h$ be as given in Definition \ref{def:1}.
For any $2\leq j\in\mathbb{N}$, we have
\begin{align}\label{eq:32a}
\max_{t\in[0,T]}{\mean{\left|\mathcal{S}_N(\mathcal{I}_h\varphi,T,t)\right|^{j}}} \leq \left\{2N^{-1} TC\left(d,\rho_{max},\rho_{\min}\right)\max_{t\in[0,T]}\{\|\nabla_h\phi_h^t\|_{\infty}^2\} \right\}^{j/2}j^{3j}.
\end{align}
\end{lemma}
\begin{proof} It is a straightforward task to modify the computations in \eqref{e:31} by replacing the map $z\mapsto z^{j}$ with the map $z\mapsto  |z|^{j}$. As a result, we get
\begin{align*}
& \m |\mathcal{S}_N(\mathcal{I}_h\varphi,T,t)|^{j} \\
& \quad = -N^{-1/2}j|\mathcal{S}_N(\mathcal{I}_h\varphi,T,t)|^{j-1}(1-2\chi_{\mathcal{S}_N(\varphi,T,t)<0})\sum_{(\f{y},\ell)}{(\mathcal{F}_\rho(t) \f{e}^d_{h,\f{y},\ell},\nabla_h \phi_{h}^t)\m \beta_{(\f{y},\ell)}}\nonumber\\
& \quad \quad + N^{-1}\frac{j(j-1)}{2}|\mathcal{S}_N(\mathcal{I}_h\varphi,T,t)|^{j-2}(\rho^+_h(t),\nabla_h\phi_{h}^t\cdot \nabla_h\phi_{h}^t)_h\m t.
\end{align*}
Taking the expected value, we obtain
\begin{align}\label{eq:42}
& \m \mean{\left|\mathcal{S}_N(\mathcal{I}_h\varphi,T,t)\right|^j} \nonumber\\
& = N^{-1}\frac{j(j-1)}{2}\mean{|\mathcal{S}_N(\mathcal{I}_h\varphi,T,t)|^{j-2}(\rho^+_h(t),\nabla_h\phi_{h}^t\cdot \nabla_h\phi_{h}^t)_h}\m t\nonumber\\
& \leq N^{-1}j(j-1)\mean{\left|\left(\rho^+_h(t),\nabla_h\phi_{h}^t\cdot \nabla_h\phi_{h}^t\right)_h\right|^{j-1}}^{1/(j-1)}\mean{\left|\mathcal{S}_N(\mathcal{I}_h\varphi,T,t)\right|^{j-1}}^{(j-2)/(j-1)}\m t\nonumber\\
& \leq  N^{-1}j(j-1)\|\nabla_h\phi_h^t\|_{\infty}^2\mean{\|\rho_h^+(t)\|_h^{j-1}}^{1/(j-1)}\mean{\left|\mathcal{S}_N(\mathcal{I}_h\varphi,T,t)\right|^{j-1}}^{(j-2)/(j-1)}\m t\nonumber\\
& \stackrel{\mathclap{\eqref{r:2b-a}}}{\leq} \, 2N^{-1}j^4(j-1)^2\|\nabla_h\phi_h^t\|_{\infty}^2C\left(d,\rho_{min},\rho_{max}\right)\mean{\left|\mathcal{S}_N(\mathcal{I}_h\varphi,T,t)\right|^{j-1}}^{(j-2)/(j-1)}\m t.
\end{align}
Taking the supremum in time, \eqref{eq:42} promptly implies
\begin{align}\label{eq:100}
\max_{t\in[0,T]}\mean{\left|\mathcal{S}_N(\mathcal{I}_h\varphi,T,t)\right|^j} & \leq \left\{2N^{-1} TC\left(d,\rho_{min},\rho_{max}\right)\max_{t\in[0,T]}\{\|\nabla_h\phi_h^t\|_{\infty}^2\} \right\}\nonumber\\
& \quad \times j^{6}\left(\max_{t\in[0,T]}\mean{\left|\mathcal{S}_N(\mathcal{I}_h\varphi,T,t)\right|^{j-1}}\right)^{(j-2)/(j-1)}.
\end{align}
We prove \eqref{eq:32a} by induction on $j$. The case $j=2$ is easily settled. Now take $j>2$ and assume the validity of \eqref{eq:32a} for $j-1$. We use \eqref{eq:100} and close off the proof by the estimate
\begin{align*}
&\max_{t\in[0,T]}\mean{\left|\mathcal{S}_N(\mathcal{I}_h\varphi,T,t)\right|^j} \\
& \quad\stackrel{\mathclap{\eqref{eq:100}}}{\leq}  \,\,\left\{2N^{-1} TC\left(d,\rho_{min},\rho_{max}\right)\max_{t\in[0,T]}\{\|\nabla_h\phi_h^t\|_{\infty}^2\} \right\}\\
& \quad\quad\quad \times j^{6}\left(\max_{t\in[0,T]}\mean{\left|\mathcal{S}_N(\mathcal{I}_h\varphi,T,t)\right|^{j-1}}\right)^{\frac{j-2}{j-1}}\\
& \quad\leq \left\{2N^{-1} TC\left(d,\rho_{min},\rho_{max}\right)\max_{t\in[0,T]}\{\|\nabla_h\phi_h^t\|_{\infty}^2\} \right\}\\
& \quad\quad\quad \times \left\{2N^{-1} TC\left(d,\rho_{min},\rho_{max}\right)\max_{t\in[0,T]}\{\|\nabla_h\phi_h^t\|_{\infty}^2 \}\right\}^{(j-2)/2}j^6(j-1)^{3(j-2)}\\
& \quad\leq \left\{2N^{-1} TC\left(d,\rho_{min},\rho_{max}\right)\max_{t\in[0,T]}\{\|\nabla_h\phi_h^t\|_{\infty}^2\} \right\}^{j/2}j^{3j}.\qedhere
\end{align*}
\end{proof}
\begin{corol}\label{cor:1} Let $\Theta$ be as in \eqref{b:7}. Let $\rho_h$ be as given in Definition \ref{def:1}. Given a jndex $\boldsymbol{j}=(j_1,\dots,j_M)$ with $|\boldsymbol{j}|_1=j$ and a set of test functions $\boldsymbol{\varphi}=(\varphi_1,\dots,\varphi_M)\in\left[C^{1+\Theta}\right]^M$, we have
\begin{align}
\max_{t\in[0,T]}{\mean{\left|\mathcal{T}^{\boldsymbol{j}}_N(\boldsymbol{\varphi},T,t)\right|}} & \leq  \left\{N^{-1} T \right\}^{j/2}j^{j}\left(\prod_{m=1}^{M}{\|\nabla\varphi_m\|_{\infty}^{j_m}}\right),\label{eq:41}\\
\max_{t\in[0,T]}{\mean{\left|\mathcal{S}^{\boldsymbol{j}}_N(\mathcal{I}_h\boldsymbol{\varphi},T,t)\right|}} & \leq \left\{2N^{-1} TC\left(d,\rho_{min},\rho_{max}\right)\right\}^{j/2}j^{3j}\left(\prod_{m=1}^{M}{\|\varphi_{m}\|_{C^{1+\Theta}}^{j_m}}\right).\label{eq:40}
\end{align}
\end{corol}
\begin{proof}
Lemma \ref{lem:7} and a multifactor H\"older inequality promptly give the inequality
\begin{align*}
& \max_{t\in[0,T]}{\mean{\left|\mathcal{S}^{\boldsymbol{j}}_N(\mathcal{I}_h\boldsymbol{\varphi},T,t)\right|}}   \leq \left\{N^{-1} TC\left(d,\rho_{min},\rho_{max}\right) \right\}^{j/2}j^{3j}\prod_{m=1}^{M}{\max_{t\in[0,T]}\{\|\nabla_h\phi_{h,m}^t\|_{\infty}\}^{j_m}}.
\end{align*}
Inequality \eqref{eq:40} is then proved by using \eqref{5005}.

Inequality \eqref{eq:41} may be deduced from adapting \eqref{eq:42}. Namely, using the It\^o formula and the maximum principle for the continuous heat equation, we get
\begin{align}\label{eq:70}
\m\mean{\left|\mathcal{T}_N(\varphi,T,t)\right|^j} & \leq N^{-1}j(j-1)\mean{\left|\mathcal{T}_N(\varphi,T,t)\right|^{j-2}\left(N^{-1}\sum_{k=1}^{N}{\left|\nabla\phi(\f{w}_k(t))\right|^2}\right)}\m t\nonumber\\
& \leq N^{-1}j(j-1)\|\nabla\phi^t\|_{\infty}^2\mean{\left|\mathcal{T}_N(\varphi,T,t)\right|^{j-1}}^{(j-2)(j-1)}\m t.\nonumber\\
& \leq N^{-1}j(j-1)\|\nabla\varphi\|_{\infty}^2\mean{\left|\mathcal{T}_N(\varphi,T,t)\right|^{j-1}}^{(j-2)(j-1)}\m t.
\end{align}
We deal with \eqref{eq:70} using the same induction argument deployed for \eqref{eq:100}, and \eqref{eq:41} is proved, again following a multifactor H\"older inequality.
\end{proof}

\section{Finite element discretisations}\label{FEM_Appendix}

\subsection{Notation}\label{ss:notassb}

For $h>0$, we split $\mathbb{T}^d$, $d=2,3$, according to a standard admissible triangulation $\mathscr{T}_h$, namely
$
\mathbb{T}^d=\bigcup_{K\in\mathscr{T}_h}{K},
$
where $h$ bounds the diameter of each polyhedron $K$.  
We assume the triangulation to be regular
and quasi-uniform (see \cite[Chapter 1, Definition 1.30]{elman2014finite} or \cite[Section 3.1]{QuarteroniValli}). 
For $p\in\mathbb{N}$, let $X^{p}_h$ be the space of continuous finite elements of order $p$ defined on the triangulation $\mathscr{T}_h$. Furthermore, let $\mathcal{R}_h$ be the Ritz operator \cite[(5)]{nitsche1979}. Finally, the symbol $\|\cdot\|$ (respectively, $(\cdot,\cdot)$) denotes the standard $L^2$-norm (respectively, the standard $L^2$-inner product).

\subsection{Assumptions and Dean--Kawasaki model}

\begin{customthm}{FE1}[Brownian particle system]\label{ass:2b} 
This is the same as Assumption \ref{ass:2}, but with the interpolation operator $\mathcal{I}_h$ replaced by the Ritz operator $\mathcal{R}_h$.
\end{customthm}
\begin{customthm}{FE2}[Scaling of parameters]\label{ass:3b}
This is the same as Assumption \ref{ass:3}.
\end{customthm}
\begin{customthm}{FE3}[Mean-field limit]\label{ass:4b} The solution to the discrete heat equation 
\begin{equation}
 \label{e:8b}
  \left\{\quad
    \begin{aligned}
    \!\!\!\!\! \partial_t (\overline{\rho}_h,f_h)&= -\frac{1}{2}\left(\nabla \overline{\rho}_h, \nabla f_h\right),\qquad \forall f_h \in X^{p}_h,\\
    \!\!\!\!\! \overline{\rho}_h(0)&= \rho_{0,h},
    \end{aligned}
  \right.
\end{equation}
is such that $\rho_{min}\leq \overline{\rho}_h \leq \rho_{max}$ (where $\rho_{min}$ and $\rho_{max}$ have been introduced in Assumption \ref{ass:2b}) for all times up to $T$ (where $T$ has have been introduced in Assumption \ref{ass:3b}).\end{customthm}

We now introduce our finite-element discretisation of the Dean--Kawasaki equation \eqref{DeanKawasaki}.

\begin{customthm2}{FE-DK}[Finite element Dean--Kawasaki model of order $p+1$]\label{def:1b}
We say that the $X^p_h$-valued process $\rho_h$ solves a $(p+1)$-th order finite element Dean--Kawasaki model if it solves
\begin{equation}
     \left\{
    \begin{aligned}\label{eq:702}
    \m\left(\rho_h, f_h \right) & = -\frac{1}{2}\left(\nabla \rho_h,\nabla f_h\right)\m t - N^{-1/2}\m \mathcal{W}(\rho^+_h, f_h ),\qquad \forall f_h \in X^{p}_h,  \\
      \rho_h(0) & = \rho_{0,h},
    \end{aligned}
    \right.
\end{equation}
where $N^{-1/2}\mathcal{W}(\rho_h^+,e_i)$ is a real-valued martingale with quadratic variation given by
\begin{align}\label{b:5}
\left\langle N^{-1/2}\mathcal{W}(\rho_h^+(t),\phi_{1,h}),N^{-1/2}\mathcal{W}(\rho_h^+(t),\phi_{2,h})\right\rangle=N^{-1}\left(\rho^+_h(t),\nabla \phi_{1,h}\cdot \nabla \phi_{2,h}\right).
\end{align}
\end{customthm2}
\begin{rem}
Unlike in the finite-difference case, we can only provide an explicit representation of the Dean--Kawasaki noise in the case $p=1$. This is due to the fact noise is nonlinear, and only preserves piece-wise constant functions (these being gradients of test functions in $X_h^{p}$, $p=1$).
We are not aware of any finite-dimensional representation of the martingale term $N^{-1/2}\mathcal{W}(\rho^+_h,e_i)$ in \eqref{eq:702} in the case $p>1$.
\end{rem}
We now present the finite element counterparts of Theorems \ref{main1}--\ref{main2}.
\begin{theorem}[Accuracy of description of fluctuations by the finite element discretised Dean--Kawasaki model of order $p+1\in\mathbb{N}$]\label{main1_FEM}  
Assume the validity of Assumptions \ref{ass:2b}--\ref{ass:3b}. 
Let $\rho_h$ be the solution of the discretised Dean--Kawasaki model given in Definition \ref{def:1b} on the time interval $[0,T]$. 
Set 
\begin{align}\label{e:907}
\nu_p(h) := \left\{ \begin{array}{rl}
1+|\ln(h)|, & \mbox{if } p=1, \\ 
1, & \mbox{if } p>1. 
\end{array}
\right.
\end{align}
Then, for any $j\in \mathbb{N}$, 
the discrete Dean--Kawasaki model \ref{def:1b} captures the fluctuations of the empirical measure $\mu^N$ in the sense that, for any $\boldsymbol{T}=(T_1,\ldots,T_M)\in [0,T]^M$ with $0\leq T_1\leq \dots\leq T_M$, the following inequality
\begin{align*}
&d_{-(2j-1)}\left[
N^{1/2}
\begin{pmatrix}
 \int_{\domain}{\rho_h(T_1)\mathcal{R}_h\varphi_1\emph{\m}\f{x}}
\\
\vdots
\\
 \int_{\domain}{\rho_h(T_M)\mathcal{R}_h\varphi_M\emph{\m}\f{x}}
\end{pmatrix}
,~
N^{1/2}
\begin{pmatrix}
\langle\mu^N_{T_1}, \varphi_1\rangle
\\
\vdots
\\
\langle\mu^N_{T_M}, \varphi_M\rangle
\end{pmatrix}
\right]
\\&~~~~~~~~~~~~~~~~~~~
\leq C(M,p,j,\|\f{\varphi}\|_{W^{p+j+3,\infty}}, \rho_{min},\rho_{max},\f{T}) {\mean{\sup_{t\in[0,T]}\|\rho_h^-(t)\|_h^2}^{1/2}}
\\&~~~~~~~~~~~~~~~~~~~~~~
+ C(M,p,j,\|\f{\varphi}\|_{W^{p+j+3,\infty}}, \rho_{min},\rho_{max},\f{T})\nu^2_p(h) h^{p+1}
\\&~~~~~~~~~~~~~~~~~~~~~~
+ C(M,p,j,\|\f{\varphi}\|_{W^{p+j+3,\infty}}, \rho_{min},\rho_{max},\f{T}) N^{-j/2}
\\&~~~~~~~~~~~~~~~~~~~
=: \mathrm{Err}_{neg} + \mathrm{Err}_{num} + \mathrm{Err}_{fluct,rel}
\end{align*}
holds for any $\boldsymbol{\varphi}=(\varphi_1,\dots,\varphi_M)\in [W^{p+j+3,\infty}(\mathbb{T}^d)]^M$ such that $\|\varphi_m\|_{L^2}=1$ for all $m=1,\dots,M$ and $\int_{\mathbb{T}^d}{\varphi_k\varphi_l\emph{\m}\f{x}}=0$ whenever $T_k=T_l$.
Finally, we have the bound 
$$
{\mean{\sup_{t\in[0,T]}\|\rho_h^-(t)\|_h^2}^{1/2}} \leq C\mathcal{E}\!\left(N,h\right),
$$
where $\mathcal{E}\!\left(N,h\right)$ has been defined in \eqref{b:8}.
\end{theorem}

\begin{theorem}[Estimates on the error for stochastic moments] 
In the same setting of Theorem~\ref{main1_FEM},
fix times $\f{T}=(T_1,\dots,T_M)\in[0,T]^M$, a vector $\boldsymbol{j}=(j_1,\dots,j_M)$ with $j:=|\boldsymbol{j}|_1$, and a vector $\boldsymbol{\varphi}=(\varphi_1,\dots,\varphi_M)\in [W^{p+j+2,\infty}]^M$.

Then the difference of moments between $\rho_h$ and the empirical density $\mu^N$ \eqref{EmpiricalMeasure}  reads
\begin{align}\label{eq:63}
&\quad \left|\mean{\prod_{m=1}^{M}{\left(N^{1/2}\int_{\mathbb{T}^d}{(\rho_h(T_m)-\mean{\rho_h(T_m)})\mathcal{R}_h\varphi_{m}\emph{\m}\f{x}}\right)^{j_m}}}\right. \nonumber\\
&\quad \quad \quad\quad \left.-\mean{\prod_{m=1}^{M}{[N^{1/2} \langle \mu^N_{T_m}-\mean{\mu^N_{T_m}},\varphi_{m}\rangle]^{j_m}}}
\right|\nonumber\\
& \quad \leq \left\{C(C+\rho_{max}) \right\}^{j/2}\left[\prod_{m=1}^{M}{q(T_m)^{j_m/2}}\right]j^{C_1j+C_2}\nonumber\\
& \quad\quad\quad\quad \times\left[\prod_{m=1}^{M}{\|\varphi_m\|_{W^{p+j+2,\infty}}^{j_m}}\right] {\mean{\sup_{t\in[0,T]}\|\rho_h^-(t)\|_h^2}^{1/2}}\nonumber\\
& \quad \quad + h^{p+1}\nu^2_k(h)\left\{C(C+\rho_{max})\right\}^{j/2}\left[\prod_{m=1}^{M}{q(T_m)^{j_m/2}}\right]j^{C_3j+C_4}\left[\prod_{m=1}^{M}{\|\varphi_m\|_{W^{p+j+2,\infty}}^{j_m}}\right]\nonumber\\
& \quad =: \mathrm{Err}_{neg} + \mathrm{Err}_{num},
\end{align}
for some positive constants $C,C_1,\dots,C_4$ independent of $j$, $h$, $N$, and $T$, where $q$ is a polynomial vanishing at $0$, and where we have the bound 
$$
{\mean{\sup_{t\in[0,T]}\|\rho_h^-(t)\|_h^2}^{1/2}} \leq C\mathcal{E}\!\left(N,h\right),
$$
where $\mathcal{E}\!\left(N,h\right)$ has been defined in \eqref{b:8}.
\end{theorem}

{\bfseries Acknowledgements}. We thank the anonymous referee for his/her careful reading of the manuscript and valuable suggestions. FC gratefully acknowledges funding from the Austrian Science Fund (FWF) through the project F65, and from the European Union's Horizon 2020 research and innovation programme under the Marie Sk\l{}odowska-Curie grant agreement No. 754411. \\

{\bfseries Data availability statement}. The datasets generated and analysed during the current study are available from the corresponding author on reasonable request.\\

{\bfseries Conflict of interest statement}. The authors have no relevant financial or non-financial interests to disclose.

\bibliographystyle{abbrv}
\bibliography{fluctuations_NEW}

\end{document}